\documentclass[10pt]{amsart}
\usepackage{amssymb, latexsym}
\usepackage{hyperref}

\usepackage{mathrsfs} 

\usepackage{backref}

\allowdisplaybreaks

\def\Xint#1{\mathchoice {\XXint\displaystyle\textstyle{#1}}%
 {\XXint\textstyle\scriptstyle{#1}}%
{\XXint\scriptstyle\scriptscriptstyle{#1}}%
{\XXint\scriptscriptstyle\scriptscriptstyle{#1}}%
\!\int}
\def\XXint#1#2#3{{\setbox0=\hbox{$#1{#2#3}{\int}$} \vcenter{\hbox{$#2#3$}}\kern-.5\wd0}}
 \def\dashint{\Xint-}
\newcommand{\avg}{\dashint} 

\renewcommand{\phi}{\varphi}
\renewcommand{\P}{\mathbb{P}}

\newcommand{\nm}[1]{\left\|#1\right\|} 
\newcommand{\abs}[1]{\left|#1\right|} 
\newcommand{\f}{\frac}

\renewcommand{\bar}[1]{\overline{#1}}

\newcommand{\lec}{\lesssim}

\newcommand{\Th}{\Theta}

\newcommand{\paren}[1]{\left(#1\right)}
\newcommand{\bracket}[1]{\left[#1\right]}
\newcommand{\braces}[1]{\left\{#1\right\}}

\newcommand{\HH}{\mathbb{H}}

\renewcommand{\a}{\alpha} 
\renewcommand{\b}{\beta} 
\renewcommand{\aa}{{\alpha '}} 

\newcommand{\bb}{{\beta '}} 

\newcommand{\rb}{{\sf b}}
\newcommand{\rh}{{\sf h}}

\let\Re=\undefined\DeclareMathOperator*{\Re}{Re}
\let\Im=\undefined\DeclareMathOperator*{\Im}{Im}

\newcommand{\Thj}{\Theta^{(j)}}
\newcommand{\Thja}{\Theta^{(j+1)}}
\newcommand{\Thjb}{\Theta^{(j+2)}}

\newcommand{\Gj}{G^{(j)}}
\newcommand{\Tha}{\Theta^{(0)}}
\newcommand{\Thb}{\Theta^{(1)}}
\newcommand{\Thc}{\Theta^{(2)}}

\theoremstyle{plain}
\newtheorem{theorem}{Theorem}
\newtheorem{proposition}[theorem]{Proposition}
\newtheorem{lemma}[theorem]{Lemma}

\theoremstyle{definition}

\newtheorem{remark}[theorem]{Remark}

\newcounter{smalllist}

\numberwithin{equation}{section} \numberwithin{theorem}{section}
\begin{document}

\title[The quartic integrability and long time existence of steep water waves in 2D]{The quartic integrability and  long time existence of steep water waves in 2D}
\author{Sijue Wu
}
\address{Department of Mathematics, University of Michigan, Ann Arbor, MI}

\thanks{Financial support in part by NSF grants DMS-1764112.}

\begin{abstract}
It is known since the work of Dyachenko \& Zakharov \cite{zd} that for the weakly nonlinear 2d infinite depth water waves,  there are no 3-wave interactions and all of the  4-wave interaction coefficients vanish on the non-trivial resonant manifold. In this paper we study this partial integrability from a different point of view. 
 We construct, directly  in the physical space,  a sequence of energy functionals $\mathfrak E_j(t)$ which are explicit in the Riemann mapping variable and involve material derivatives of order $j$ of the solutions for the 2d water wave equation,  so that  $\frac d{dt} \mathfrak E_j(t)$
is quintic or higher order.  
We show that if some scaling invariant norm, and a norm involving one  spatial derivative above the scaling of the initial data are of size no more than $\varepsilon$, then the lifespan of the solution for the 2d water wave equation is at least of order $O(\varepsilon^{-3})$, and the solution remains as regular as the initial data during this time. If only the scaling invariant norm of the data is of size $\varepsilon$, then the lifespan of the solution is at least of order $O(\varepsilon^{-5/2})$. 
Our long time existence results do not impose size restrictions on the slope of the initial interface and the magnitude of the initial velocity, they allow the interface to have arbitrary large steepnesses  and initial velocities to have arbitrary large magnitudes.  
\end{abstract}

\maketitle

\baselineskip15pt

\section{Introduction}\label{intro}

A class of water wave problems concerns the
motion of the 
interface separating an inviscid, incompressible, irrotational fluid,
under the influence of gravity, 
from a region of zero density (i.e. air) in 
$n$-dimensional space. It is assumed that the fluid region is below the
air region. Assume that
the density  of the fluid is $1$, the gravitational field is
$-{\bold k}$, where ${\bold k}$ is the unit vector pointing in the  upward
vertical direction, and at  
 time $t\ge 0$, the free interface is $\partial\Omega(t)$, and the fluid
occupies  region
$\Omega(t)$. When surface tension is
zero, the motion of the fluid is  described by 
\begin{equation}\label{euler}
\begin{cases}   \ \bold v_t + (\bold v\cdot \nabla) \bold v = -\bold k-\nabla P
\qquad  \text{on } \Omega(t),\ t\ge 0,
\\
\ \text{div}\,\bold v=0 , \qquad \text{curl}\,\bold v=0, \qquad  \text{on }
\Omega(t),\ t\ge 
0,
\\  
\ P=0, \qquad\qquad\qquad\qquad\qquad\text{on }
\partial\Omega(t) \\ 
\ (1, \bold v) \text{ 
is tangent to   
the free surface } (t, \partial\Omega(t)),
\end{cases}
\end{equation}
where $ \bold v$ is the fluid velocity, $P$ is the fluid
pressure. 
There is an important condition for these problems:
\begin{equation}\label{taylor}
-\frac{\partial P}{\partial\bold n}\ge 0
\end{equation}
pointwise on the interface, where $\bold n$ is the outward unit normal to the fluid interface 
$\partial\Omega(t)$ \cite{ta};
it is well known that when surface tension is neglected and the Taylor sign condition \eqref{taylor} fails, the water wave motion can be subject to the Taylor instability \cite{ ta, bi, bhl, ebi}.

The study of water waves dates back centuries to  Newton \cite{newton}, Stokes \cite{st}, Levi-Civita \cite{le}, and G.I. Taylor \cite{ta}.   Nalimov
\cite{na},  Yosihara \cite{yo} and Craig \cite{cr} gave early 
 local in time existence and uniqueness results for the 2d water wave equation \eqref{euler} for small and smooth initial data. In  \cite{wu1, wu2}, Wu showed that in  dimensions $n\ge 2$, the strong Taylor sign condition
 \begin{equation}\label{taylor-s}
-\frac{\partial P}{\partial\bold n}\ge c_0>0
\end{equation}
always holds for the
 infinite depth water wave problem \eqref{euler},\footnote{More precisely, the strong Taylor sign condition \eqref{taylor-s} holds for $C^{1+\gamma}$ interfaces, where $\gamma>0$. For less regular interfaces, such as interfaces with angled crests, only the weak Taylor sign condition \eqref{taylor} holds, with degeneracies at the singularities, cf. \cite{kw, wu8}.}  and the initial value problem of equation \eqref{euler} is locally well-posed in Sobolev spaces $H^s$, $s\ge 4$ for arbitrary given data.  Since then,
local wellposedness for water waves with additional effects such as the surface tension, bottom and non-zero vorticity,  under the assumption \eqref{taylor-s},\footnote{When there is surface tension, or bottom, or vorticity,  \eqref{taylor-s} does not always hold, it needs to be assumed. See \cite{la}, \cite{su1}, \cite{su2} for some interesting computations of $-\frac{\partial P}{\partial \bold n}$ in the cases of finite depth and point vortices.} were obtained, c.f. \cite{am, cl, cs, ig1, la, li, ot, sz, zz};
 local wellposedness of \eqref{euler} in low regularity Sobolev spaces,   cf. \cite{abz, abz14, hit, ai}, and in a regime allowing for non-$C^1$ interfaces, cf. \cite{kw, wu8, agr1, agr2, agr3} were proved. 
  
  There are also numerous studies on the long time existence and regularity of solutions for the water wave equation \eqref{euler} over the last ten years. Wu \cite{wu3, wu4}, Germain, Masmoudi \& Shatah \cite{gms}, Ionescu \& Pusateri  \cite{ip} and Alazard \& Delort  \cite{ad} proved almost global and global existence of solutions for two and three dimensional water waves equation \eqref{euler} for small, smooth and sufficiently localized data, further developments were given in 
\cite{hit, it, dipp, wang1, wang2, bmsw, bd, su1, fz, ai, ait}.  There are two main ideas involved in these works.   The first is to use a normal form transformation to remove all the quadratic nonlinearities in the equation \eqref{euler}\footnote{For water waves with surface tension \cite{dipp} or non-zero vorticity \cite{su1}, only part of the quadratic nonlinearities can be removed.}; and the second is to use the dispersive decay properties of the spatially localized solutions for the equation \eqref{euler} to obtain extended life spans of the solutions. 

Despite these advances, many important questions remain open. For instance, in all the aforementioned works on the long time existence of water waves, norms involving derivatives both above and below the scaling are assumed small at the initial time; 
such smallness cannot be preserved after a rescaling. However physical phenomena do not disappear with a rescaling, a valid description of a physical phenomenon should only involve scaling invariant quantities and norms; also in all of these works,  
 it is assumed that the steepness of the initial wave is sufficiently small. The questions  are whether the smallness assumption on the steepness can be removed? and whether a long time existence result holds assuming smallness only on some scaling invariant quantities?

The study of the 2d water wave equation \eqref{euler} in the Hamiltonian point of view began in \cite{za}, where
Zakharov discovered that the 2d equation \eqref{euler} can be written as a Hamiltonian system. In \cite{zd} Dyachenko  \& Zakharov showed that (formally) there are no three-wave resonant interactions in the Hamiltonian 2d periodic water wave equation \eqref{euler} and all of the four-wave interaction coefficients vanish on the non-trivial resonant manifold.\footnote{While there are no 3-wave resonant interactions, there are two types of 4-wave resonances: the trivial ones and the Benjamin \& Feir resonances. What \cite{zd} proved was that all the coefficients of the 4th order terms in the power expansion of the  Hamiltonian vanish on the Benjamin \& Feir resonance manifold, cf. \cite{bfp} for further explanations. In the whole line case, there are in addition near resonances.} 
   Dyachenko \& Zakharov  \cite{zd}  and Craig \& Wolfolk \cite{cw} derived a formal \footnote{That is, it is unbounded and non-invertible.} symplectic transformation that maps the Hamiltonian system of the 2d periodic water waves to its Birkhoff normal form of order 4; and mapping properties of the transformation were studied in \cite{cs2}. Building on \cite{zd, cw, cs2}, Berti, Feola \& Pusateri \cite{bfp} gave a rigorous construction of a bounded and invertible (non-symplectic) transformation in a neighborhood of the origin in phase space, mapping the 2d periodic water wave equation \eqref{euler} to its Birkhoff normal form up to order 4, and showed that  for sufficiently small and smooth initial data of size $\varepsilon$,\footnote{In \cite{bfp}, the steepness, and quantities both above and below the scaling are  assumed small at the initial time.}  the 2d periodic water wave equation \eqref{euler} is solvable for time of order $O(\varepsilon^{-3})$.  
\cite{bfp} is the first and the only work putting the formal computations in \cite{zd, cw}  in a rigorous framework.  The transformation in \cite{bfp} is constructed via composing several paradifferential flow conjugations. The final resonant Poincar\'e-Birkhoff normal form system constructed in \cite{bfp}  is not a priori explicit;  an important step in \cite{bfp} is a normal form uniqueness argument that allows the authors to identify the Poincar\'e-Birkhoff normal form system of \cite{bfp} with the Birkhoff normal form system constructed in \cite{zd, cw}, up to degree 4 of homogeneity.

The computations in \cite{zd, cw, cs2, bfp} 
were all carried out in the Fourier space using tools such as the Birkhoff normal forms from the Hamiltonian system. 
As the water wave equation \eqref{euler} is explicitly formulated in the physical space,  
it is natural to ask then how do the structural results on the 3-wave and 4-wave interactions  in \cite{zd} manifest, explicitly,  in the physical space? and what are the specific features in  equation \eqref{euler} that lead to these structural  properties? The motivation is clear: by better understanding the structure of equation \eqref{euler} we hope to obtain better information on the long time behavior of  the water waves. Moreover, we hope to gain insight on the non-linear structure of a wider class of equations.  
 
In this paper
we construct,  directly in the physical space,  an infinite sequence of energy functionals $\mathfrak E_j(t)$ that involve material derivatives  of order $j$ of the solutions of the 2d water wave equation \eqref{euler},  so that $\frac d{dt}\mathfrak E_j(t)$ is quintic.\footnote{If not specified, the "quintic" in \S\ref{intro} means it is a finite sum of terms homogeneous of degree 5 or higher of the unknown functions and their derivatives; similar for "quartic".}  
Our energy functionals $\mathfrak E_j(t)$ and their time derivatives $\frac d{dt}\mathfrak E_j(t)$ are explicit in the Riemann mapping variable, and the process of constructing  the sequence of our energy functionals $\mathfrak E_j(t)$,  as well as $\mathfrak E_j(t)$, $\frac d{dt}\mathfrak E_j(t)$ themselves, reveal  remarkable novel symmetric structures of the 2d water wave equation \eqref{euler}. Moreover for $j\ge 2$, $\mathfrak E_j(t)$ are energy functionals in terms of the derivatives of the velocity and the derivatives of the steepness,  water waves with large velocity and large steepness can be small  in $\mathfrak E_j(t)$; similarly, $\frac d{dt}\mathfrak E_j(t)$ are quintic in terms of only the sizes of  the derivatives of the velocity and the derivatives of the steepness of the interface. As a consequence we show that 
if some scaling invariant norm and a norm involving one order spatial derivative above the scaling 
 of the initial data  are of size no more than $\varepsilon$, then the lifespan of the solution for the 2d water wave equation \eqref{euler} is at least of order $O(\varepsilon^{-3})$; if only the scaling invariant norm of the initial data is of size $\varepsilon$, then the solution will exist up to time of order $O(\varepsilon^{-5/2})$.\footnote{No dispersive properties of  equation \eqref{euler} are used.}  
 Our long time existence results do not impose size restrictions on the slope of the initial interface and the magnitude of the initial velocity, the initial interface can have arbitrary large steepness, and the magnitude of the initial velocity can be arbitrarily large.   To the best of our knowledge, this is the first long time existence result assuming smallness only on a scaling invariant quantity at the initial time, and it is the first allowing steep initial interfaces.

We will give rigorous statements of our main results, in the whole line setting,  in \S\ref{main1} and \S\ref{main2}. The same construction and derivation 
as in this paper work and analogous results  hold for the 2d periodic water wave equation \eqref{euler}. The only difference 
is that in the periodic case, there are some additional (spatially) constant terms  arising from the means of the quantities involved, c.f. Proposition 1 and equation (290) of \cite{kw}.  However these constants are harmless, and results analogous to Theorem 2.8 and Theorem 3.1, as well as the identities in  \S\ref{main1}  leading to Theorem 2.8 hold for the periodic case.\footnote{In the periodic case,  there are some extra (spatially) constant terms $c(t)$ in the corresponding quantities $\P_H G^{(j)}$, cf. \S\ref{main1} for the definition, however they do not affect the energy functional sequence, since they show up in the energy functionals as $\int_I\partial_\aa\Theta\, c(t)\,d\aa$. Since $\int_I\partial_\aa\Theta \,c(t)\,d\aa=0$ for any periodic function $\Theta$ on $I$,  these constants do not affect the outcome.}

Just as this work continues the investigation initiated in \cite{wu3, wu4, bmsw, wu8},  we think that the results in \S\ref{main1}
will themselves have further consequences. 
The next subsection gives a more detailed explanation of the main results,  motivation for the method and  relations to earlier work.

\subsection{Further explanations on the motivation,  the ideas and the main results in this paper}
In this paper we regard $\mathbb R^2$ as a complex plane, and identify a point $(x,y)\in\mathbb R^2$ with its complex form $x+iy\in \mathbb C$. $\Re z$, $\Im z$ are the real and imaginary parts of $z$, $\bar z=\Re z-i\,\Im z$.

Let $z=x+iy=z(\alpha,t)$, $\alpha\in\mathbb R$ be the interface $\partial\Omega(t)$ in Lagrangian coordinate $\alpha$, so $z_t(\alpha,t)=\bold v(z(\alpha,t), t)$ is the velocity of the fluid particle on the interface, and $z_{tt}(\alpha,t)=( \bold v_t+\bold v\cdot\nabla\bold v)(z(\alpha,t), t)$ is the acceleration. 
We know (cf. \cite{wu3}) the interface equation for the 2d water waves \eqref{euler} is given by
\begin{equation}\label{interface-intro}
\begin{cases}
z_{tt}+i=i\frak a z_\a,\\
\bar z_t=\mathfrak H \bar z_t,
\end{cases}
\end{equation}
where $\frak a=-\frac{\partial P}{\partial \bold n} \frac1{|z_\a|}$, and $\mathfrak H$ is the Hilbert transform associated to the fluid domain $\Omega(t)$; and the quasilinear equation, obtained by applying $\partial_t$ to the first equation in \eqref{interface-intro},  is 
\begin{equation}\label{quasi-intro1}
(\partial_t^2+i\frak a\partial_\a)\bar z_t=\frac{\frak a_t}{\frak a} (\bar z_{tt}-i).
\end{equation}
Let $\Phi: \Omega(t)\to \mathscr P_-$ be the Riemann mapping from the fluid domain $\Omega(t)$ to the lower half plane $\mathscr P_-$, satisfying $\lim_{z\to\infty}\Phi_{z}(z)=1$ and $\Phi(z(0,t), t)=0$; let $\rh(\a,t):=\Phi(z(\a,t),t)$, $b(\aa,t):=\rh_t\circ \rh^{-1}(\aa,t)$. And let $Z(\aa,t)=X(\aa,t)+iY(\aa,t):=z(\rh^{-1}(\aa,t),t)$, $Z_t(\aa,t):=z_t(\rh^{-1}(\aa,t),t)$ and $Z_{tt}(\aa,t):=z_{tt}(\rh^{-1}(\aa,t),t)$ be the position, velocity and acceleration of  the interface in the Riemann mapping variable $\aa$;   let $D_t:=\partial_t+b\partial_\aa$ be the material derivative, i.e. $Z_t=D_t Z$ and $Z_{tt}=D_t Z_t$. Let \begin{equation}\label{op-intro}
\mathfrak P:=D_t^2+i\,\frac{A_1}{|Z_{,\aa}|^2} \partial_\aa,
\end{equation}
and $\mathbb H$ be the Hilbert transform associated to the lower half plane $\mathscr P_-$:
\begin{equation}\label{ht-intro}
\mathbb Hf(\aa)=\frac1{\pi i}p.v.\int\frac1{\aa-\bb}f(\bb)\,d\bb,
\end{equation}
$\P_H=\frac12(I+\mathbb H)$, $\P_A=\frac12(I-\mathbb H)$ be the projections onto the space of holomorphic, and respectively, anti-holomorphic functions in $\mathscr P_-$. 
The quasilinear equation for the 2d water waves in the Riemann mapping variable $\alpha'$ is (cf \cite{wu1, wu8}) 
\begin{equation}\label{quasi-intro}
\mathfrak P {\bar Z}_{t}=\dfrac{\frak a_t}{\frak a}\circ \rh^{-1} ({\bar Z}_{tt}-i), 
\end{equation}
where $A_1$, $b$ and $\frac{\frak a_t}{\frak a}\circ\rh^{-1}$ are as given in \eqref{A1b} and \eqref{at}. The main operator in \eqref{quasi-intro1} is $\partial_t^2+i\frak a\partial_\a$ and   in \eqref{quasi-intro}, $\mathfrak P$.  It is clear that the quasilinear equation \eqref{quasi-intro1} or \eqref{quasi-intro} contains quadratic nonlinear terms. 

 In \cite{wu3} we found a pair of transformations which can be used to completely transform away the quadratic nonlinearities in \eqref{quasi-intro1}.  
Let $\phi$ be the velocity potential, $\psi:=\phi(z(\alpha,t),t)$ be the trace of the velocity potential on the interface. Let
 \begin{equation}\label{2d-trans}
 \Lambda=(I-\mathfrak H)\psi,\qquad \Pi=(I-\mathfrak H)(z-\bar z);\qquad \text{and} \quad \kappa=2\Re z-\rh.
 \end{equation}
In \cite{wu3} we showed that the quantities $\Lambda$ and $\Pi$, after the change of coordinates $\kappa$,  satisfy  equations of the type
 \begin{equation}\label{cubic-intro}
 (\partial_t^2-i\partial_{\tilde \alpha})(\Lambda\circ\kappa^{-1})=G_1,\qquad  (\partial_t^2-i\partial_{\tilde\alpha})(\Pi\circ\kappa^{-1})=G_2,
 \end{equation}
 where $G_1$ and $G_2$ consist of only cubic and higher order terms, $\tilde\alpha=\kappa(\alpha,t)$, c.f.  Propositions 2.3 and 2.4 of \cite{wu3}. Using equations \eqref{cubic-intro} and  applying the basic energy functional\footnote{cf. Lemma 4.1 of \cite{wu3}. Upon changing to the Riemann mapping coordinate yields
  \begin{equation}\label{basic-ef1}
 \tilde E_0(t)=\Re \int (\frac1{\frak a}|\partial_t\theta|^2+i\theta\partial_\alpha\bar{\theta})\,d\alpha=\Re \int(\frac1{A_1}|D_t\Theta|^2+i\Theta\partial_\aa\bar{\Theta})\,d\aa,
 \end{equation}
 where $\Theta=\theta\circ \rh^{-1}$. As was explained in \cite{wu3}, it is advantageous to replace the Dirichlet-to-Neumann operator $\nabla_n$ by $i\partial_\a$ when dealing with the cancellations in the quadratic nonlinearities of the 2d water wave equation, since $\nabla_n$ is nonlinear, while $i\partial_\a$ is linear. }
 \begin{equation}\label{basic-ef}
 \tilde E_0(t)=\Re \int \paren{\frac1{\frak a}|\partial_t\theta|^2+i\theta\partial_\alpha\bar{\theta}}\,d\alpha
  \end{equation}
 on $\Lambda$, $\Pi$ and their derivatives,  we constructed an infinite sequence of energy functionals $\tilde E_j(t)$, so that 
 \begin{equation}\label{quartic-energy}
 \frac d{dt}\tilde E_j(t)=\text{quartic};\footnote{The "quartic" in this paragraph means it is a finite sum of terms homogeneous of degree 4 or higher of  the amplitude, the steepness and the velocity of the interface, as well as their derivatives. }
 \end{equation}
 we then proved our almost global existence result  by employing the quartic energy identities \eqref{quartic-energy}\footnote{The quartic energy identities obtained by directly applying $E_0(t)$ to the derivatives of the basic quantities $\Lambda$ and $\Pi$ did not quite work, some further modifications were introduced to produce desired estimates, cf. \cite{wu3}.}
 and the method of vector fields, cf. \cite{wu3}. Quartic energy estimates\footnote{The basic mechanism behind these quartic energy estimates is the absence of three-wave interactions.} similar to \eqref{quartic-energy} also played key roles in the subsequent works \cite{ip, ad, hit, it}. While \cite{ip} directly improved upon the quartic energy estimates  in \cite{wu3}, \cite{ad} obtained theirs using the Zakharov-Craig-Sulem formulation of the water wave equation and a normal form transformation via resonant analysis of the 3-wave interactions in the Fourier space. In \cite{hit, it} the authors used 
$$Q:=(I+\mathbb H)(\psi\circ \rh^{-1}), \qquad\text{and }\quad  W:=\frac12(I+\mathbb H)(Z-\bar Z)=Z-\aa,$$
which are  the 
counter-parts\footnote{ Applying the basic energy functional \eqref{basic-ef1} to $D_t^j Q$  yields  $$\tilde{\tilde E}_j(t)=\Re \int  \paren{\frac1{A_1}\abs{D_t^{j+1}Q}^2+i\,(D_t^j \bar{Q})\partial_\aa{(D_t^j Q)}}\,d\aa, $$  which satisfies the quartic energy identity \eqref{quartic-energy} due to \eqref{cubic-intro}, and analogous result holds for  $D_t^j W$. However in the Riemann mapping coordinate, not all derivatives of $Q$ or $W$ will lead to energy functionals  satisfying \eqref{quartic-energy}. This is solely because of the change of the coordinates.}  of the quantities  $\Lambda$ and $\Pi$ in \eqref{2d-trans} in the Riemann mapping  coordinate, 
and a basic energy functional similar to \eqref{basic-ef};\footnote{ In the basic energy functional of \cite{hit, it} the first term is slightly different from the first term in \eqref{basic-ef1} while the second term is  the same as that in \eqref{basic-ef1}.   } they first applied their basic energy functional to $(W, Q)$ and their derivatives, and then directly modified the resulting energy functionals to obtain their quartic energy estimates.  In all of \cite{ip, ad, hit, it} dispersive properties of the water wave equation and modified scatterings were additionally used to obtain their global existence results.\footnote{The work \cite{wu3, ip, ad, hit, it} discussed here 
are all about small, smooth and sufficiently localized solutions, the smallness assumption is not scaling invariant, and the steepness of the interface is assumed small.}

In what follows we call an energy functional $\tilde E_j(t)$ satisfying \eqref{quartic-energy} a quartic energy functional. Similarly if $\frac d{dt} \mathfrak E_j(t)$ is quintic, we call $\mathfrak E_j(t)$ a quintic energy functional.

 However none of the existing quartic energy functionals  $\tilde E_j(t)$, including those  in  \cite{wu3, ip, ad, hit}, can be easily modified to turn them into quintic, since the remainders $\frac d{dt}\tilde E_j(t)$  are too complex to be used to identify  further cancellations. 
 Our strategy 
  in constructing our quintic energy functional sequence $\mathfrak E_j(t)$ is to first  re-understand and re-construct appropriate quartic energy functionals; we choose the Riemann mapping coordinate since it is advantageous in understanding the structure of the 2d water wave equation, cf. \cite{wu1, kw, wu6, wu8}. We begin with the quantity $Q$ (the counter-part of $\Lambda$ in the Riemann mapping coordinates), and look for useful symmetries and structures hidden in the equations of the quantities involved and in the energy functionals.  The construction of our new quartic energy functionals $E_j(t)$ involve three new interrelated ingredients: the sequence of quantities $\Theta^{(j)}:=(\P_H D_t)^jQ$,  their associated equations
 \begin{equation}\label{Gj-intro}
 D_t\P_HD_t\Theta^{(j)}+i\,\frac1{|Z_{,\aa}|^2}\partial_\aa\Theta^{(j)}:=G^{(j)}, 
 \end{equation}
 cf. \eqref{Thj}, \eqref{Gj},
 and the symmetric basic energy functional, cf. \eqref{energy},
\begin{equation}\label{energy-intro}
E(t)=\Re \paren{\int i\,\partial_\aa\Th_2 \overline{D_t\Th_1}\,d\aa-\int i\,\partial_\aa\Th_1 \overline{D_t\Th_2}\,d\aa}.
\end{equation}
 By  Proposition~\ref{prop:G},  $\P_H G^{(j)}$ is cubic and possesses  remarkable symmetries. Applying the basic energy $E(t)$ to $\Theta_1=\Theta^{(j)}$ and $\Theta_2=\Theta^{(j+1)}$, and using \eqref{energyid} we arrive at the new quartic energy identity \eqref{16}-\eqref{ej1}-\eqref{17}.\footnote{The "quartic" here is in terms of only the derivatives of the steepness of the interface and  the derivatives of the velocity; similar for the "quintic" energies $\frak E_j(t)$. Observe that the construction of $E_j(t)$ and $\frak E_j(t)$ is based on a sequence of equations, instead of repeated derivatives to a single transformed equation.}  What is important is that \eqref{ej1}-\eqref{17} possesses desirable symmetric structure that allows us, with further use of the new identities \eqref{eq:lemma1}-\eqref{eq:lemma2}, to systematically construct the quartic correcting functionals $C_{1,j}$ and $C_{2,j}$ in \eqref{c1}, \eqref{c2} and obtain our quintic energy functionals $\frak E_j(t)$, cf. \eqref{energyj1}.  The main result of \S\ref{main1} is stated in Theorem~\ref{th:main1}. 
 
The main difficulty in proving our second main theorem, Theorem~\ref{thm:main2},  is that the quantity $\nm{\frac1{Z_{,\aa}}-1}_{L^\infty}$ is not  small, and no norms below the scaling are assumed small.
 This requires us to explore further cancelations in the quintic energy identities in Theorem~\ref{th:main1} and employ stronger inequalities, including those in Propositions~\ref{quartic-inq} and ~\ref{prop:hhalf-2}. The details of the proof for Theorem~\ref{thm:main2} can be found in \S\ref{proof}.

 Finally we mention that the transformations  in \cite{wu3}, cf. \eqref{2d-trans},  are used in \cite{su1, chensu} to prove the recent important results on the long time behavior of point vortices and the nonlinear modulation instability of the Stokes waves  in 2d water waves. Analogous transformations as those in \eqref{2d-trans} were also constructed to remove the quadratic nonlinearities  in the 3d water wave equation \cite{wu4} for perturbations near the flat equilibrium and in the free boundary problem of a self-gravitating incompressible fluid \cite{bmsw} for perturbations near the circular equilibrium.
  
 \subsection{Outline of the paper}
We present our results  on the algebraic aspects of the water wave equation, namely the quartic integrability, and their proofs in   \S\ref{main1}. In \S\ref{main2} we state our long time existence theorem, its proof is given in \S\ref{proof}. A complete list of notations and conventions can be found 
in Appendix~\ref{notations}. 
Some of the basic equations and formulas derived in our earlier works \cite{wu1, wu3, wu8}, as well as some additional identities that are used in the derivations in this paper are collected in Appendix~\ref{iden}.  Appendix~\ref{ineq} contains the  inequalities that are used for our proofs. 
 In Appendix~\ref{quantities} we  summarize the estimates obtained in \S\ref{quan} and \S\ref{step1-4} for easy referencing.

\subsection{Conventions}
We consider solutions of 2d the water wave equation \eqref{euler} in the setting where the fluid domain $\Omega(t)$ is simply connected in  $\mathbb R^2$, with the free interface $\partial\Omega(t):=\partial\Omega(t)$ being a non-self-intersecting curve,
 $${\bold v}(z, t)\to 0,\qquad\text{as } |z|\to\infty$$
 and the interface $\partial\Omega(t)$ tending to horizontal lines at infinity.
We will primarily use the Riemann mapping variable in this paper. 

\subsection{Acknowledgement} The author would like to thank Jeffrey Rauch for carefully going through the first draft of the paper and for his helpful suggestions.

\section{The main results on the structure of the water wave equation}\label{main1}
In this section we  
construct a sequence of energy functionals $\mathfrak E_j(t)$ which involve material derivatives of order $j$  
of the solutions,  such that  $\frac d{dt} \mathfrak E_j(t)$ is quintic.\footnote{The "quintic" in \S\ref{main1} in general means it is a finite sum of terms homogeneous of degree 5 or higher of the unknown functions $Z_t$, $\frac1{Z_{,\aa}}-1$ and their derivatives. However at the end what we construct in this section are the energy functionals $\frak E_j(t)$, cf. Theorem~\ref{th:main1},  which are quintic in terms of only the spatial derivatives of $\paren{Z_t, \frac1{Z_{,\aa}}-1}$.}
  The construction is based on  a series of observations, given in Lemma~\ref{prop:q} through Proposition~\ref{prop:ic}, and equations \eqref{ej1}-\eqref{17}, \eqref{24}-\eqref{25}.

We continue with the notations introduced in \S\ref{intro}, further notations can be found in Appendix~\ref{notations}.

As derived in earlier work,  cf. \cite{wu1, wu6} or \S2.2 of \cite{wu8}, we have
\begin{equation}\label{b}
b=\Re (I-\mathbb H)\paren{\frac{Z_t}{Z_{,\aa}}}=\mathbb P_H \paren{\frac{\bar Z_t}{\bar Z_{,\aa}}}+ \mathbb P_A \paren{\frac{Z_t}{Z_{,\aa}}}.
\end{equation}
 
 Since the fluid is incompressible and irrotational, so $\bold v=\nabla \phi$, where the velocity potential $\phi$ satisfies the Bernoulli 
equation in the fluid domain $\Omega(t)$:
\begin{equation}\label{bernoulli}
\begin{cases}
\phi_t+\frac12|\nabla\phi|^2+P+y=0,\qquad
\Delta \phi=0,\qquad\text{in }\Omega(t),\\
P=0,\qquad \text{on }\partial\Omega(t).
\end{cases}
\end{equation}
Let $\psi(\alpha,t):=\phi(z(\alpha,t),t)$ be the trace of the velocity potential on the interface, and   
$$Q:=(I+\mathbb H)(\psi\circ \rh^{-1})=2\,\P_H (\psi\circ \rh^{-1}).$$
\begin{lemma}\label{prop:q}
We have\footnote{Applying $D_\aa$ to both sides of \eqref{q1} gives the first equation in the interface system \eqref{2dinterface}. }
\begin{equation}\label{q1}
D_tQ=i (Z-\aa) +\P_A \paren{|Z_t|^2},
 \end{equation}
\begin{equation}\label{q2}
D_t\mathbb P_H D_t Q+i\,\frac1{|Z_{,\aa}|^2}\partial_\aa Q=i\,\mathbb P_A\paren{Z_t\paren{1-\frac1{Z_{,\aa}}}+\bar Z_t\paren{\frac1{\bar{Z}_{,\aa}}-1}}.
\end{equation}
\end{lemma}

\begin{proof}
By chain rule, \eqref{bernoulli} and the fact that $\nabla\phi(z,t)=z_t$ on the interface $z$, we have
\begin{equation}\label{10}
\begin{cases}
\psi_t=\phi_t+ \nabla\phi \cdot z_t=-y+\frac12|z_t|^2,\\
\psi_\a= \nabla\phi \cdot z_\a= z_t \cdot z_\a.
\end{cases}
\end{equation}
Changing to the Riemann mapping variable $\aa$ in the second equation yields $\partial_\aa (\psi\circ \rh^{-1})=Z_t\cdot Z_{,\aa}=\Re (\bar Z_t Z_{,\aa})$, which in turn gives
\begin{equation}\label{qa}
\partial_\aa Q=\bar Z_t Z_{,\aa},
\end{equation}
because the holomorphic quantities $\partial_\aa Q$ and $\bar Z_t Z_{,\aa}$ have the same real parts. A change of variables in the first equation in \eqref{10} yields $D_t(\psi\circ \rh^{-1})=-Y+\frac12|Z_t|^2$.

We compute, by first commuting $D_t$ with $\HH$, then applying  \eqref{b}, Proposition~\ref{prop:comm-hilbe}, \eqref{paph} and \eqref{qa}, that 
$$\begin{aligned}
D_tQ&=(I+\mathbb H)D_t(\psi\circ \rh^{-1})+[D_t, \mathbb H](\psi\circ \rh^{-1})\\
&=(I+\mathbb H)(-Y+\frac12|Z_t|^2)+\bracket{\frac {Z_t}{Z_{,\aa}},\HH}\partial_\aa \mathbb P_H(\psi\circ \rh^{-1})+\bracket{\frac {\bar Z_t}{\bar Z_{,\aa}},\HH}\partial_\aa \mathbb P_A(\psi\circ \rh^{-1})\\
&= (I+\mathbb H)(-Y+\frac12|Z_t|^2)+\P_A\paren{\frac {Z_t}{Z_{,\aa}}\partial_\aa Q }-\P_H\paren{\frac {\bar Z_t}{\bar Z_{,\aa}}\partial_\aa \bar Q }\\
&=i (Z-\aa) +\P_A \paren{|Z_t|^2},
\end{aligned}
$$
where $(I+\mathbb H)(-Y)=i(Z-\aa)$ because both of the holomorphic quantities $(I+\mathbb H)(-Y)$ and $i(Z-\aa)$ have the same real parts. This gives \eqref{q1}.\footnote{Equation \eqref{q1} was also derived in \cite{hit} using a slightly different approach.} 

Observe that $\P_H D_t Q=i(Z-\aa)$ by \eqref{q1}. We further compute, by \eqref{b} and the fact that $\P_A Z_t=Z_t$ and $\P_A \bar Z_t=0$, that 
$$
D_t \braces{i\,(Z-\aa)}=i\, (Z_t-b)=-i\,\frac {\bar Z_t}{\bar Z_{,\aa}}+i\,\mathbb P_A\paren{Z_t\paren{1-\frac1{Z_{,\aa}}}+\bar Z_t\paren{\frac1{\bar{Z}_{,\aa}}-1}}.
$$
Observe that $ i\,\frac {\bar Z_t}{\bar Z_{,\aa}}=i \,\frac1{|Z_{,\aa}|^2}\partial_\aa Q   $ by \eqref{qa}. This gives \eqref{q2}.
\end{proof}
 
 Let 
\begin{equation}\label{Thj}
\Theta^{(0)}:=Q,\qquad \Thj:=(\P_HD_t)^jQ, \qquad 
\qquad\text{and }
\end{equation}
\begin{equation}\label{Gj}
D_t\mathbb P_H D_t \Thj+i\,\frac1{|Z_{,\aa}|^2}\partial_\aa \Thj=:G^{(j)}.
\end{equation}
We know $\P_H G^{(0)}=0$ by \eqref{q2}. We would like to find a formula for $\P_H \Gj$. To this end we derive a recursive relation, which in turn gives the formula for $\P_H \Gj$.

\begin{proposition}\label{prop:G}
1. Let $\Theta$ be holomorphic, i.e.  $\P_A\Th=0$, and $\Theta_{1}=\P_HD_t\Theta$. Assume that
\begin{equation}\label{eq:1}
D_t\mathbb P_H D_t \Theta+i\,\frac1{|Z_{,\aa}|^2}\partial_\aa \Theta=G,\qquad \text{and }\quad D_t\mathbb P_H D_t \Theta_{1}+i\,\frac1{|Z_{,\aa}|^2}\partial_\aa \Theta_{1}=G_{1}.
\end{equation}
Then 
\begin{equation}\label{eq:G}
\P_H G_{1}-\P_H D_t \P_H G=\frac12\P_H\braces{\frac1{\bar Z_{,\aa}} \paren{<\bar Z_t, i\,\frac1{\bar Z_{,\aa}}, D_\aa \Theta>+  <-i\,\frac1{ Z_{,\aa}}, Z_t,  D_\aa \Theta>  }  },
\end{equation}
where the cubic-form $<\cdot,\cdot,\cdot>$ is defined by
$$
<f,g,h>:=\frac1{\pi i}\int\frac{(f(\aa)-f(\bb))(g(\aa)-g(\bb))(h(\aa)-h(\bb))}{(\aa-\bb)^2}\,d\bb.
$$
2. We have 
\begin{equation}\label{eq:Gj}
\P_H(\Gj)=\sum_{l=0}^{j-1}(\P_H D_t)^l\paren{\P_H (G^{(j-l)})-\P_H D_t\P_H(G^{(j-l-1)})},
\end{equation}
where $\P_H (G^{(j-l)})-\P_H D_t\P_H(G^{(j-l-1)})$ is given by \eqref{eq:G} with $\Th=\Th^{(j-l-1)}$. 

\end{proposition}

\begin{proof}
Let $F=\P_A D_t\Th$ and $F_1=\P_A D_t\Th_1$. We know by \eqref{b} and the identity \eqref{projid} that
\begin{equation}\label{11}
F=\P_A D_t\Th=\P_A\paren{\frac{Z_t}{Z_{,\aa}}\partial_\aa\Th}.
\end{equation}
It is easy to check, by the definitions and the decomposition identity \eqref{paph} that
$$D_t\P_H G- G_1=-D_t\P_A G+D_t F_1+i\,\frac1{|Z_{,\aa}|^2}\partial_\aa F+\bracket{D_t, \, i\,\frac1{|Z_{,\aa}|^2}\partial_\aa}\Th;
$$
using the definitions again yields
$$\P_A G=F_1+\P_A\paren{i\,\frac1{|Z_{,\aa}|^2}\partial_\aa\Th},$$
so
$$D_t\P_H G- G_1=- D_t \P_A\paren{i\,\frac1{|Z_{,\aa}|^2}\partial_\aa\Th}+i\,\frac1{|Z_{,\aa}|^2}\partial_\aa F+\bracket{D_t, \, i\,\frac1{|Z_{,\aa}|^2}\partial_\aa}\Th;
$$
this gives,  by \eqref{b}, \eqref{projid}, \eqref{11} and \eqref{eq:c28}, that
\begin{equation}\label{12}
\begin{aligned}
\P_H (D_t\P_H G- G_1)&=- \P_H\paren{ \frac{\bar Z_t}{\bar Z_{,\aa}}\partial_\aa \P_A\paren{i\,\frac1{|Z_{,\aa}|^2}\partial_\aa\Th}}\\&+\P_H\paren{i\,\frac1{|Z_{,\aa}|^2}\partial_\aa \P_A \paren{\frac{Z_t}{Z_{,\aa}}\partial_\aa\Th}}+i\,\P_H  \paren{\frac{b_\aa-2\Re D_\aa Z_t}{|Z_{,\aa}|^2}\partial_\aa   \Th}.
\end{aligned}
\end{equation}
Now by identity \eqref{projid}  and the definition of $D_\aa\Th$ we can rewrite the first two terms in \eqref{12} as
$$
\begin{aligned}-\P_H\paren{ \frac{\bar Z_t}{\bar Z_{,\aa}} \partial_\aa\P_A\paren{i\,\frac1{|Z_{,\aa}|^2}\partial_\aa\Th}}&=-\P_H\paren{ \frac1{\bar Z_{,\aa}}\P_H \paren{\bar Z_t \partial_\aa \P_A\paren{i\,\frac1{\bar Z_{,\aa}}D_\aa\Th}}}, \\
\P_H\paren{i\,\frac1{|Z_{,\aa}|^2}\partial_\aa \P_A \paren{\frac{Z_t}{Z_{,\aa}}\partial_\aa\Th}}&=\P_H\paren{\frac1{\bar Z_{,\aa}}\P_H\paren{i\, \frac 1{Z_{,\aa}}\partial_\aa \P_A \paren{Z_tD_\aa\Th}}}.
\end{aligned}
$$
We use \eqref{eq:c30} and the fact that $\bar Z_t$ and $\frac1{Z_{,\aa}}$ are holomorphic to further rewrite  
$$-\P_H \paren{\bar Z_t \partial_\aa \P_A\paren{i\,\frac1{\bar Z_{,\aa}}D_\aa\Th}}=\frac12\P_H\bracket{\bar Z_t, i\,\frac1{\bar Z_{,\aa}}; D_\aa\Th}
$$
and
$$\P_H\paren{i\, \frac 1{Z_{,\aa}}\partial_\aa \P_A \paren{Z_tD_\aa\Th}}=\frac12\P_H\bracket{ - i\,\frac1{Z_{,\aa}}, Z_t; D_\aa\Th}.
$$
Replacing $b_\aa-2\Re D_\aa Z_t$ in \eqref{12} by \eqref{ba} and using part 2. of Proposition~\ref{prop:comm-hilbe} gives \eqref{eq:G}.

It is easy to check \eqref{eq:Gj} using the fact that $\P_H G^{(0)}=0$. 
\end{proof}

We next present a new energy  identity. The basic energy form, such as $\int |\partial_t\theta|^2+\frak a \nabla_n \theta\, \bar\theta\,d\a$, is commonly used in the proof of the local well-posedness of the water wave equations. 
In \cite{wu3} we found that the energy form
$$\int \frac1{\frak a}|\partial_t\theta|^2+i\partial_\a \theta\, \bar\theta\,d\a,$$
with the coefficient $\frak a$ moved to the first term and the Dirichlet-Neumann operator $\nabla_n$ replaced by $i\partial_\a$ was advantageous in the study of the quadratic and cubic cancellations in the energy functionals and proving long time existence, see Lemma 4.1 of \cite{wu3}; however it is not adequate for the purpose of this paper. 
In the following Proposition we introduce  a new, symmetric,  basic energy form. We will see that this new energy form makes it easier for us to find the quartic correcting functionals to cancel out the quartic terms in the time derivatives of the  basic energies.

\begin{proposition}\label{prop:energyid} 
Let $\Theta_1$, $\Theta_2$ be holomorphic, i.e. $\P_A\Th_1=\P_A\Th_2=0$, smooth and decay fast at infinity. 
Define
\begin{equation}\label{energy}
E(t)=\Re \paren{\int i\,\partial_\aa\Th_2 \overline{D_t\Th_1}\,d\aa-\int i\,\partial_\aa\Th_1 \overline{D_t\Th_2}\,d\aa}.\end{equation}
Then 
\begin{equation}\label{energyid}
\frac d{dt} E(t)=\Re\paren{\int i\,\partial_\aa\Th_2\overline{(\P_HG_1)}\,d\aa-\int i\,\partial_\aa\Th_1\overline{(\P_HG_2)}\,d\aa}.
\end{equation}
where $G_{k}:=D_t\mathbb P_H D_t \Theta_{k}+i\,\frac1{|Z_{,\aa}|^2}\partial_\aa \Theta_{k}$, for $k=1,2$.
\end{proposition}
\begin{proof}
We use \eqref{dte} to compute $E'(t)$. We have, after applying \eqref{eq:c7} and cancelling out equal terms,\footnote{Observe that \eqref{14} holds without the assumption that $\Th_1$, $\Th_2$ are holomorphic.}
\begin{equation}\label{14}
\begin{aligned}\frac d{dt} E(t)&=
\Re \paren{\int i\,\partial_\aa\Th_2 \overline{D_t^2\Th_1}\,d\aa-\int i\,\partial_\aa\Th_1 \overline{D_t^2\Th_2}\,d\aa}\\
&=\Re \paren{\int i\,\partial_\aa\Th_2 \overline{(D_t^2\Th_1+i\,\frac1{|Z_{,\aa}|^2}\partial_\aa \Theta_{1})}\,d\aa-\int i\,\partial_\aa\Th_1 \overline{(D_t^2\Th_2+i\,\frac1{|Z_{,\aa}|^2}\partial_\aa \Theta_{2})}\,d\aa};
\end{aligned}
\end{equation}
it is clear that the terms inserted in the second step above sum up to zero. We now want to show
\begin{equation}\label{13}
\Re \int i\,\partial_\aa\Th_2 \overline{D_t\P_A D_t \Th_1}\,d\aa=\Re \int i\,\partial_\aa\Th_1 \overline{D_t\P_A D_t\Th_2}\,d\aa.
\end{equation}
We begin with the term on the left hand side. Using Cauchy integral formula to insert a $\P_H$ gives
$$\int i\,\partial_\aa\Th_2 \overline{D_t\P_A D_t \Th_1}\,d\aa=\int i\,\partial_\aa\Th_2 \overline{\P_H D_t\P_A D_t \Th_1}\,d\aa.$$
By \eqref{b} and \eqref{paph},
$\P_H D_t\P_A D_t \Th_1=\P_H(\frac{\bar Z_t}{\bar Z_{,\aa}}\partial_\aa\P_A D_t \Th_1)$, and using Cauchy integral formula again to remove the $\P_H$ and insert a $\P_A$, and then use \eqref{11}, we get 
\begin{equation}\label{15}\int i\,\partial_\aa\Th_2 \overline{\P_H D_t\P_A D_t \Th_1}\,d\aa=\int i\,\partial_\aa\Th_2 \overline{\frac{\bar Z_t}{\bar Z_{,\aa}}\partial_\aa\P_A D_t \Th_1}\,d\aa=i\,\int\P_AD_t\Th_2\partial_\aa\overline{\P_A D_t \Th_1}\,d\aa.
\end{equation}
 Exchanging roles  between $\Th_1$ and $\Th_2$ in \eqref{15} gives us a similar identity for the right hand term in \eqref{13},  a further integration by parts then yields \eqref{13}. From \eqref{14} and \eqref{13}, and using Cauchy integral formula again to insert a $\P_H$ gives \eqref{energyid}. 
\end{proof}

Let $\Th_1=\Thj$, $\Th_2=\Thja$, and define
\begin{equation}\label{16}
E_j(t)=\Re \paren{\int i\,\partial_\aa\Thja \overline{\Thja}\,d\aa-\int i\,\partial_\aa\Thj \overline{\Thjb}\,d\aa}.\end{equation}
From \eqref{energyid}  we get, 
$$\frac d{dt} E_j(t)=\Re\paren{\int i\,\partial_\aa\Thja \overline{\P_H\Gj}\,d\aa-\int i\,\partial_\aa\Thj\overline{\P_HG^{(j+1)}}\,d\aa};
$$
compute further by using  \eqref{dte}, \eqref{eq:Gj}, \eqref{q2} and Proposition~\ref{prop:cif}, we obtain
\begin{equation}\label{ej1}
\begin{aligned}
\frac d{dt} E_j(t)&=\frac d{dt}\Re\int i\,\partial_\aa\Thj\overline{\P_H \Gj}\,d\aa\\
&+2\Re\sum_{l=0}^{j-1}\braces{\int i\, \partial_\aa\bar\Thj(\P_H D_t)^{l+1}\P_H\paren{G^{(j-l)}-D_t\P_H G^{(j-1-l)}}\,d\aa}\\&+
\Re\braces{\int i\,\partial_\aa \bar\Thj\P_H\paren{G^{(j+1)}-D_t\P_H G^{(j)}}\,d\aa}
.
\end{aligned}
\end{equation}
where  
\begin{equation}\label{17}
\begin{aligned}
&\P_H\paren{G^{(j-l)}-D_t\P_H G^{(j-1-l)}}\\&=\frac12 \P_H\braces{\frac1{\bar Z_{,\aa}} \paren{<\bar Z_t, i\,\frac1{\bar Z_{,\aa}}, D_\aa \Theta^{(j-l-1)}>+  <-i\,\frac1{ Z_{,\aa}}, Z_t,  D_\aa \Theta^{(j-l-1)}>  }  },
\end{aligned}
\end{equation}
for $l=-1, 0,\dots, j-1$, by Proposition~\ref{prop:G}.

Observe that $\frac d{dt} E_j(t)$ is quartic with a few desirable features: 1. the integrands depend only on the spatial derivatives of the quantities $\bar Z_t$, $\frac 1{Z_{,\aa}}$ and $\Thj$; 2. it has remarkable symmetries, with the order of derivatives evenly distributed among the factors: we note that $D_\aa\Tha=\bar Z_t$. 
This should allow us to readily derive very low regularity estimates. 
 However our main interest here is to find a quartic correcting functional for $E_j(t)$, 
 so that the time derivative of the corrected energy functional is quintic in terms of the spatial derivatives of the quantities $\bar Z_t$, $\frac 1{Z_{,\aa}}$ and $\Thj$  
 and prove the new long time existence result. The feature of $\frac d{dt} E_j(t)$ allows us to do so, thanks to a few additional observations on the structure of the water wave equation. 
 
 We  note that by \eqref{2dinterface}, $-i\,\frac1{Z_{,\aa}}+i= \bar Z_{tt}+\text{\it quadratic}$, and by \eqref{qa}, $D_\aa \Theta^{(k)}=D_t^{k}\bar Z_t+\text{\it quadratic}$.
 Let\footnote{Observe that for any function $f=f(\aa,t)$, $\mathfrak D_t f=D_t f$ and $\mathcal Pf=(D_t^2+i\,\frac{A_1}{|Z_{,\aa}|^2}\partial_\aa)f=:\mathfrak P f$, where $\mathfrak P$ is defined in \eqref{op}.}
\begin{align}\label{19}
\mathfrak D_t&:=\partial_t+b(\aa,t)\partial_\aa+b(\bb,t)\partial_\bb,
\\
\label{20}
\mathcal P&:=\mathfrak D_t^2+i\,\frac{A_1(\aa,t)}{|Z_{,\aa}|^2}\partial_\aa+i\,\frac{A_1(\bb,t)}{|Z_{,\bb}|^2}\partial_\bb.
\end{align}
And  let 
 \begin{equation}\label{18}
 \theta:=  \bar Z_t(\aa,t)-\bar Z_t(\bb,t),
 \end{equation}
 observe that 
\begin{equation}\label{23}
\mathfrak D_t^j\theta=D_t^j\bar Z_t(\aa,t)-D_t^j\bar Z_t(\bb,t), \qquad \text{for }j\ge 0. 
\end{equation}
 So the quartic terms in 
 $$2\Re\sum_{l=0}^{j-1}\braces{\int i\, \partial_\aa\bar\Thj(\P_H D_t)^{l+1}\P_H\paren{G^{(j-l)}-D_t\P_H G^{(j-1-l)}}\,d\aa}$$
 consists of 
 \begin{equation}\label{21}
I_{1,j}:=\frac1{\pi } \Re\sum_{l=0}^{j-1}\iint D_t^j Z_t(\aa,t) \frac{\mathfrak D_t^{l+1}\braces{\mathfrak D_t\paren{\theta(\aa,\bb,t)\overline{\theta(\aa, \bb,t)}}\mathfrak D_t^{j-l-1}\theta(\aa,\bb,t)}}{(\aa-\bb)^2} \,d\bb\,d\aa.
\end{equation}
 Similarly the quartic terms in $\Re\int i\,\partial_\aa \bar\Thj\P_H\paren{G^{(j+1)}-D_t\P_H G^{(j)}}\,d\aa$ 
 are
 \begin{equation}\label{22}
I_{2,j}:=\frac1{2\pi } \Re\iint D_t^jZ_t(\aa,t) \frac{\mathfrak D_t\paren{\theta(\aa,\bb,t)\overline{\theta(\aa, \bb,t)}}\mathfrak D_t^{j}\theta(\aa,\bb,t)}{(\aa-\bb)^2} \,d\bb\,d\aa.
\end{equation}
We know it is possible to find quartic functionals $C_{i,j}$, $i=1,2$, so that $\frac d{dt} C_{i, j}(t)=I_{i,j}+\text{\it quintic}$, based on the following observations.

\begin{lemma}\label{lemma1}
Assume that $f$ and $\mathcal G$ are smooth and decay fast at infinity. We have
\begin{equation}\label{eq:lemma1}
\begin{aligned}
&\frac d{dt}\iint \frac{\bar f \,\mathfrak D_t \mathcal G-\mathfrak D_t \bar f \,\mathcal G}{(\aa-\bb)^2}\,d\aa d\bb=\iint\frac{\bar f\mathcal P \mathcal G}{(\aa-\bb)^2} \,d\aa d\bb\\&
+\iint\paren{-\frac{\overline{\mathcal P f}\mathcal G}{(\aa-\bb)^2}+\paren{b_\aa+b_\bb-2\frac{b(\aa)-b(\bb)}{\aa-\bb}}\frac{\bar f \,\mathfrak D_t \mathcal G-\mathfrak D_t \bar f \,\mathcal G}{(\aa-\bb)^2}}\,d\aa d\bb\\&+
i\,\iint\paren{\partial_\aa\frac{A_1(\aa)}{|Z_{,\aa}|^2}+\partial_\bb\frac{A_1(\bb)}{|Z_{,\bb}|^2}-2\frac{\frac{A_1(\aa)}{|Z_{,\aa}|^2}-\frac{A_1(\bb)}{|Z_{,\bb}|^2}}{\aa-\bb}}\frac{\bar f \,\mathcal G}{(\aa-\bb)^2}\,d\aa d\bb.
\end{aligned}
\end{equation}
\end{lemma}
\eqref{eq:lemma1} follows easily from \eqref{dteab} and integration by parts. We omit the details.

\begin{lemma}\label{lemma2}
Let $\mathcal G= g \bar h  q$. Then 
\begin{equation}\label{eq:lemma2}
\mathcal P\mathcal  G=  2\mathfrak D_t\paren{g\,\bar h }\,\mathfrak D_t q+2 \mathfrak D_t\paren{g\,\mathfrak D_t \bar h }\, q
+ \bar h\, q\,\mathcal P g+ g\,\bar h \,\mathcal P q- g\,q\,\overline{\mathcal P h}.
\end{equation}
\end{lemma}
 The verification of \eqref{eq:lemma2} is straightfoward. We omit the details. 
 
To find the quartic functionals $C_{i,j}$, we will use Lemma~\ref{lemma1} with $\mathcal G=g\bar h q$.   
From \eqref{eq:lemma1}, \eqref{eq:lemma2},  we have
\begin{equation}\label{27}
\frac d{dt}\iint \frac{\bar f \,\mathfrak D_t \mathcal (g\,\bar h\, q)-(\mathfrak D_t \bar f) \,g\,\bar h\, q }{(\aa-\bb)^2}\,d\aa d\bb
=2\iint\frac{\bar f\,  \mathfrak D_t\paren{g\,\bar h }\,\mathfrak D_t q+ \bar f \,\mathfrak D_t\paren{g\,\mathfrak D_t \bar h }\, q  }{(\aa-\bb)^2} \,d\aa d\bb
+quintic,
\end{equation}
provided $\mathcal P f$, $\mathcal P g$, $\mathcal P h$ and $\mathcal P q$ are quadratic.
Observe that on the right hand side of \eqref{27} the operator 
$\mathfrak D_t$ acting on $q$ in the first term  is moved to the second term, acting on $\bar h$.  And we know 
$\mathcal P\mathfrak D_t^j\theta=\text{\it quadratic}$ by \eqref{quasi}. Instead of going through the straightforward but tedious process of finding the correcting functionals $C_{i,j}$, we will directly give the results. Before doing so, we use Lemmas~\ref{lemma1} and ~\ref{lemma2} to write the following equation. 
 We have, for $j, k, i, l, m\ge 0$, and $D_t^j Z_t=D_t^j Z_t(\aa,t)$,
\begin{equation}\label{24}
\begin{aligned}
\frac d{dt}&\iint\frac{\paren{D_t^j Z_t\,\mathfrak D_t-\mathfrak D_t (D_t^j Z_t)}\mathfrak D_t^{m}\paren{\mathfrak D_t^l\theta\,\mathfrak D_t^i\bar\theta\,\mathfrak D_t^k\theta }}{(\aa-\bb)^2}\,d\aa\,d\bb\\&=
2\iint\frac{D_t^j Z_t\,\mathfrak D_t^m\braces{\mathfrak D_t\paren{\mathfrak D_t^l\theta\,\mathfrak D_t^i\bar\theta}\mathfrak D_t^{k+1}\theta+ \mathfrak D_t\paren{\mathfrak D_t^l\theta\,\mathfrak D_t^{i+1}\bar\theta}\mathfrak D_t^{k}\theta  }}{(\aa-\bb)^2}\,d\aa\,d\bb+ R^{(m)}_{\bar j; l, \bar i, k},
\end{aligned}
\end{equation}
where
\begin{equation}\label{25}
\begin{aligned}
&R^{(m)}_{\bar j; l, \bar i, k}=\iint\frac{ D_t^jZ_t\, \bracket{\mathcal P, \mathfrak D_t^{m}}\paren{\mathfrak D_t^l\theta\,\mathfrak D_t^i\bar\theta\,\mathfrak D_t^k\theta }}{(\aa-\bb)^2}\,d\aa\,d\bb\\&+
\iint\frac{ D_t^jZ_t\,  \mathfrak D_t^{m}\braces{(\mathcal P\mathfrak D_t^l\theta)\,\mathfrak D_t^i\bar\theta\,\mathfrak D_t^k\theta - \mathfrak D_t^l\theta\, \overline{(\mathcal P\mathfrak D_t^i\theta)}\,\mathfrak D_t^k\theta +\mathfrak D_t^l\theta\,\mathfrak D_t^i\bar\theta\,(\mathcal P\mathfrak D_t^k\theta )}}{(\aa-\bb)^2}\,d\aa\,d\bb\\&
-\iint\frac{\overline{(\mathcal P D_t^j\bar Z_t)}\,\mathfrak D_t^{m}\paren{\mathfrak D_t^l\theta\,\mathfrak D_t^i\bar\theta\,\mathfrak D_t^k\theta }}{(\aa-\bb)^2}\,d\aa\,d\bb\\&
+\iint\paren{b_\aa+b_\bb-2\frac{b(\aa)-b(\bb)}{\aa-\bb}}\frac{\paren{D_t^j Z_t\mathfrak D_t-\mathfrak D_t (D_t^j Z_t)}\mathfrak D_t^{m}\paren{\mathfrak D_t^l\theta\,\mathfrak D_t^i\bar\theta\,\mathfrak D_t^k\theta }}{(\aa-\bb)^2}\,d\aa\,d\bb
\\&
+i\iint \paren{\partial_\aa\frac{A_1(\aa)}{|Z_{,\aa}|^2}+\partial_\bb\frac{A_1(\bb)}{|Z_{,\bb}|^2}-2\frac{\frac{A_1(\aa)}{|Z_{,\aa}|^2}-\frac{A_1(\bb)}{|Z_{,\bb}|^2}}{\aa-\bb}}\frac{D_t^j Z_t\,\mathfrak D_t^{m}\paren{\mathfrak D_t^l\theta\,\mathfrak D_t^i\bar\theta\,\mathfrak D_t^k\theta }}{(\aa-\bb)^2}\,d\aa\,d\bb.
\end{aligned}
\end{equation}
It is clear that $R^{(m)}_{\bar j; l, \bar i, k}$ is quintic.

We now give the correcting functionals. 
Let\footnote{We define $\sum_{l=0}^{-1}=\sum_{l=0}^{-2}=0$. $D_t^j Z_t=D_t^j Z_t(\aa,t)$, $D_t^{j+1} Z_t=D_t^{j+1} Z_t(\aa,t)$ in \eqref{c1} and \eqref{c2}. }
\begin{equation}\label{c1}
\begin{aligned}
C_{1,j}&=\frac1{2\pi}\sum_{l=0}^{j-1}\sum_{k=0}^l\iint \frac{\paren{D_t^j Z_t\mathfrak D_t- D_t^{j+1}Z_t}\mathfrak D_t^{l-k} \paren{\mathfrak D_t^k\theta\,\,\bar \theta \,\mathfrak D_t^{j-l-1}\theta}}{(\aa-\bb)^2}\,d\bb\,d\aa\\
&+\frac1{4\pi} \sum_{l=0}^{j-2}\sum_{k=0}^{j-l-2}(-1)^k \iint\frac{\paren{D_t^jZ_t\mathfrak D_t-D_t^{j+1}Z_t}\mathfrak D_t^{1+l}\theta\,\mathfrak D_t^k\bar\theta \,\mathfrak D_t^{j-l-2-k}\theta}{(\aa-\bb)^2}\,d\bb\,d\aa\\&
-\frac1{8\pi} \sum_{l=0}^{j-2}\sum_{k=0}^{j-l-2}(-1)^{k}\iint\frac{\paren{\theta\mathfrak D_t-\mathfrak D_t\theta}\mathfrak D_t^{j-l-1}\bar\theta \,\mathfrak D_t^{j-k-1}\theta\,\mathfrak D_t^{k+l+1}\bar\theta}{(\aa-\bb)^2}\,d\bb\,d\aa\\&
+\frac1{2\pi}\sum_{l=0}^{j-1}\iint D_t^jZ_t\, \frac{ \,\mathfrak D_t^{j-l-1}\theta\,\,\bar\theta \,\mathfrak D_t^{1+l}\theta}{(\aa-\bb)^2}\,d\bb\,d\aa,
\end{aligned}
\end{equation}
and 
\begin{equation}\label{c2}
\begin{aligned}
C_{2,j}&=\frac1{4\pi}\sum_{k=0}^{j-1}(-1)^k\iint \frac{\paren{D_t^jZ_t\mathfrak D_t-D_t^{j+1}Z_t} \theta \,\mathfrak D_t^k\bar\theta\,\mathfrak D_t^{j-k-1}\theta}{(\aa-\bb)^2}\,d\bb\,d\aa\\
&+\frac1{4\pi}(-1)^j\iint D_t^jZ_t\frac{ \,\theta\, \mathfrak D_t^j\bar\theta\, \,\theta}{(\aa-\bb)^2}\,d\bb\,d\aa.
\end{aligned}
\end{equation}

\begin{proposition}\label{prop:ic}  Let $j\ge 0$.  We have
\begin{equation}\label{eq:ic}
I_{1,j}+I_{2,j}- \frac d{dt}  \Re \paren{ C_{1,j}+C_{2,j}}=  R_{IC,j}
\end{equation}
where 
\begin{equation}\label{remainder1}
\begin{aligned}
&  R_{IC,j}=-\frac1{2\pi}\sum_{l=0}^{j-1}\Re\iint \paren{b_\aa+b_\bb-2\frac{b(\aa)-b(\bb)}{\aa-\bb}}\frac{  D_t^j Z_t\,\mathfrak D_t^{j-l-1}\theta\,\,\bar\theta \, \mathfrak D_t^{l+1}\theta}{(\aa-\bb)^2}\,d\aa\,d\bb\\&-
\frac1{4\pi}\Re\paren{2\sum_{l=0}^{j-1}\sum_{k=0}^l R^{(l-k)}_{\bar j; k, \bar 0, j-l-1}+  \sum_{l=0}^{j-2}\sum_{k=0}^{j-l-2} (-1)^k \paren{R^{(0)}_{\bar j; 1+l, \bar k, j-l-2-k}-   \overline {R^{(0)}_{\bar 0; j-l-1, \bar {j-k-1}, l+k+1}}}}\\&
-(-1)^j\frac{1}{4\pi}\Re\iint \paren{b_\aa+b_\bb-2\frac{b(\aa)-b(\bb)}{\aa-\bb}}\frac{  D_t^j Z_t\,\,\theta \, \mathfrak D_t^{j}\bar \theta\, \, \theta}{(\aa-\bb)^2}\,d\aa\,d\bb\\&-
\frac1{4\pi}\Re\sum_{k=0}^{j-1}(-1)^k R^{(0)}_{\bar j; 0, \bar k, j-k-1}.
\end{aligned}
\end{equation}

\end{proposition}

Observe that $\mathfrak D_t D_t^jZ_t=D_t^{j+1}Z_t$. It is easy to prove Proposition~\ref{prop:ic} by \eqref{eq:lemma1}-\eqref{eq:lemma2} or \eqref{24}-\eqref{25},\footnote{When computating $\frac d{dt}C_{1,j}$, use the identity
$$\mathfrak D_t^{l-k}\paren{\mathfrak D_t(\mathfrak D_t^k\theta\,\bar\theta)\mathfrak D_t^{j-l}\theta+\mathfrak D_t(\mathfrak D_t^k\theta\,\mathfrak D_t\bar\theta)\mathfrak D_t^{j-l-1}\theta}=\mathfrak D_t^{l-k+1}\paren{ \mathfrak D_t(\mathfrak D_t^k\theta\,\bar\theta)\mathfrak D_t^{j-l-1}\theta   } -\mathfrak D_t^{l-k}\paren{ \mathfrak D_t(\mathfrak D_t^{k+1}\theta\,\bar\theta)\mathfrak D_t^{j-l-1}\theta   }.
$$
 Also, when using \eqref{eq:lemma1}-\eqref{eq:lemma2} or \eqref{24}-\eqref{25} to compute, apply $\mathfrak D_t$ to the factors in the exact order given in the right hand side of \eqref{eq:lemma2} or \eqref{24} to facilitate cancelations.}
  \eqref{dteab} and the symmetry,\footnote{By interchanging  $\aa$ with $\bb$ and use symmetry we have
 \begin{equation}\label{26}
 \begin{aligned}
&\iint\frac{f(\aa)(g(\aa)-g(\bb))(h(\aa)-h(\bb))(q(\aa)-q(\bb))}{(\aa-\bb)^2}\,d\aa\,d\bb\\&=
\frac12 \iint\frac{ (f(\aa)-f(\bb))(g(\aa)-g(\bb))(h(\aa)-h(\bb))(q(\aa)-q(\bb))  }{(\aa-\bb)^2}\,d\aa\,d\bb.
\end{aligned}
\end{equation}
}
 we omit the details.
 \begin{remark}
 It is clear from \eqref{remainder1} that the remainder $R_{IC,j}$ is quintic. 
 \end{remark}

We now sum up \eqref{ej1}-\eqref{17}-\eqref{21}-\eqref{22} and Proposition~\ref{prop:ic}, 
and present it in the following. Let 
\begin{align}\label{energyj1}
\mathfrak E_j(t)=E_j(t)-\Re\paren{\int i\partial_\aa\Thj\bar{\P_H \Gj}\,d\aa+C_{1,j}(t)+ C_{2,j}(t)}.
\end{align}
\begin{theorem}\label{th:main1} 
We have
\begin{equation}\label{eq:main1}
\frac d{dt} \mathfrak E_j(t)=\mathfrak R_j(t),
\end{equation}
where
\begin{equation}\label{remainder2}
\begin{aligned}
\mathfrak R_j&=\paren{2\Re\sum_{l=0}^{j-1}\int i\, \partial_\aa\bar\Thj(\P_H D_t)^{l+1}\P_H\paren{G^{(j-l)}-D_t\P_H G^{(j-1-l)}}\,d\aa-I_{1,j}}\\&+
\paren{\Re\int i\,\partial_\aa \bar\Thj\P_H\paren{G^{(j+1)}-D_t\P_H G^{(j)}}\,d\aa-I_{2,j}}+ R_{IC,j}
\end{aligned}
\end{equation}
 is quintic. 
\end{theorem}
In particular for $j=0$, equations \eqref{energyj1}-\eqref{remainder2} gives
\begin{equation}\label{28}
\mathfrak E_0(t)=\int\paren{ i\,\partial_\aa(Z-\aa)\bar{(Z-\aa)}+|Z_t|^2} \,d\aa-\frac1{8\pi}\iint\frac{|Z_t(\aa)-Z_t(\bb)|^4}{(\aa-\bb)^2}\,d\aa\,d\bb
\end{equation}
and
\begin{equation}\label{29}
\begin{aligned}
\frac d{dt} \mathfrak E_0(t)&=\frac12\int i\,Z_t\paren{<\bar Z_t, i\,\frac{1-A_1}{\bar Z_{\aa}}, \bar Z_t>+< -i\,\frac{1-A_1}{ Z_{\aa}}, Z_t,\bar Z_t>}\,d\aa\\&-\frac1{8\pi}\iint\paren{b_\aa+b_\bb-2\frac{b(\aa)-b(\bb)}{\aa-\bb}}\frac{|Z_t(\aa)-Z_t(\bb)|^4}{(\aa-\bb)^2}\,d\aa\,d\bb.
\end{aligned}
\end{equation}
In our proof for the long time existence result, Theorem~\ref{thm:main2}, we will only use \eqref{eq:main1}-\eqref{remainder2} for $1\le j\le 4$.
Similar to  \eqref{ej1}-\eqref{17}, equations \eqref{eq:main1}-\eqref{remainder2} can be readily used to derive very low regularity results. However our focus in this paper is on the new long time existence result for solutions to the water wave equation.

\subsection{Scaling} The solutions for the 2d water wave equation \eqref{2dinterface}-\eqref{A1b} obey the following scaling law: 
If $\paren{\bar Z_t, Z}$ is a solution of  \eqref{2dinterface}-\eqref{A1b}, then 
\begin{equation}\label{39}
\paren{\bar Z_t^\lambda, Z^\lambda}:=\paren{\lambda^{-1/2}\bar Z_t(\lambda\aa,\lambda^{1/2} t), \lambda^{-1}Z(\lambda\aa,\lambda^{1/2} t)}
\end{equation}
is also a solution of the equation \eqref{2dinterface}-\eqref{A1b}. The following  norms are scaling invariant for the water wave equation \eqref{2dinterface}-\eqref{A1b}:
$\nm{\frac1{Z_{,\aa}}}_{\dot H^{1/2}(\mathbb R)}$, $\nm{\bar Z_{t,\aa}}_{L^2(\mathbb R)}$, and $\nm{\frac1{Z_{,\aa}}-1}_{L^\infty(\mathbb R)}$.

\section{The  long time existence of steep water waves}\label{main2}

We are now ready to present  our new long time existence result, Theorem~\ref{thm:main2}.  

Let
\begin{equation}\label{283}
L(t)=\nm{\frac1{Z_{,\aa}}(t)}_{\dot H^{1/2}(\mathbb R)}+\nm{\bar Z_{t,\aa}(t)}_{L^2(\mathbb R)}+\nm{\partial_\aa \frac1{Z_{,\aa}}(t)}_{\dot H^{1/2}(\mathbb R)}+\nm{\partial_\aa^2\bar Z_{t}(t)}_{L^2(\mathbb R)}.
\end{equation}

\begin{theorem}\label{thm:main2}
1. Let $J\ge 2$. Assume that  the initial data $\paren{\bar Z_t(0), \frac 1{Z_{,\aa}}(0)-1}\in \cap_{\frac12\le s\le J}\dot H^{s}(\mathbb R)\times \dot H^{s-\frac12}(\mathbb R)$. Then there are constants $m_0>0$, and $0<\varepsilon_0\le 1$, such that for all $0<\varepsilon\le \varepsilon_0$, if the data satisfies 
\begin{equation}\label{37}
L(0)\le \varepsilon,\qquad \nm{\frac1{Z_{,\aa}}(0)-1}_{L^\infty}<1,\qquad\text{and }\quad  E_1(0) E_3(0)\le m_0^2,
\end{equation}
there is a constant $\mathcal T_0>0$, depending only on $m_0$, so that the initial value problem for the water wave equation \eqref{2dinterface}-\eqref{A1b} or equivalently \eqref{euler} has a unique classical solution for the time period $ [0, \frac{\mathcal T_0}{\varepsilon^3}]$. During this time, the solution is as regular as the initial data, $L(t)\lec\varepsilon$, $\nm{\frac1{Z_{,\aa}}(t)-1}_{L^\infty}< 1$    and $E_1(t)E_3(t)\lec m_0^2$.

2. If instead of \eqref{37} the data satisfies 
\begin{equation}\label{38}
\nm{\frac1{Z_{,\aa}}(0)}_{\dot H^{1/2}(\mathbb R)}+\nm{\bar Z_{t,\aa}(0)}_{L^2(\mathbb R)}\le \varepsilon, \qquad \nm{\frac1{Z_{,\aa}}(0)-1}_{L^\infty}< 1,\quad\text{and }\quad E_1(0) E_3(0)\le m_0^2,
\end{equation}
then there is a constant $\mathcal T_1>0$, depending on $m_0$, $\nm{\partial_\aa \frac1{Z_{,\aa}}(0)}_{\dot H^{1/2}}$ and $\nm{\partial_\aa^2\bar Z_{t}(0)}_{L^2}$, 
so that the initial value problem for the water wave equation \eqref{2dinterface}-\eqref{A1b} or equivalently \eqref{euler} has a unique classical solution for the time period $[0, \frac{\mathcal T_1}{\varepsilon^{5/2} }]$. During this time, the solution is as regular as the initial data.

\end{theorem}
\begin{remark}
Observe that $ E_1(t) E_3(t)$ is scaling invariant. The sole reason for  the assumption on
$ E_1(0) E_3(0)$ is to control the evolution of the norm $\nm{\frac1{Z_{,\aa}}(t)-1}_{L^\infty}$.
We will show  that for $\varepsilon$ small enough, $ E_1(t)$ controls $\nm{\frac1{Z_{,\aa}}(t)-1}_{L^2}^2 +\nm{\Theta^{(2)}(t)}_{\dot H^{1/2}}^2$, and 
$ E_3(t)$ controls $\dfrac14 \nm{D_\aa \frac1{Z_{,\aa}^2}(t)}_{L^2}^2+\nm{\Theta^{(4)}(t)}_{\dot H^{1/2}}^2$, see Proposition~\ref{prop:4.4}; and the quantities $\Theta^{(2)}(t)$, $\Theta^{(4)}(t)$ are mainly related to the velocity $Z_t$, see {\sf Step 1}, {\sf Step 3} in \S\ref{step1-4}. We know by Lemma~\ref{lemma:4.2} that, assuming $\nm{1-\frac1{Z_{,\aa}}}_{L^\infty}\le 1$,    there is a constant $c>0$, such that
\begin{equation}\label{40}
\nm{\frac1{Z_{,\aa}}-1}_{L^\infty(\mathbb R)}^2\le c \nm{\frac1{Z_{,\aa}}-1}_{L^2(\mathbb R)}\nm{D_\aa \frac1{Z_{,\aa}^2}}_{L^2(\mathbb R)},
\end{equation}
so
\begin{equation}\label{41}
\nm{\frac1{Z_{,\aa}}(t)-1}_{L^\infty}^4\le 4c^2 E_1(t) E_3(t).
\end{equation}
 We will show  that the growth of $E_1(t) E_3(t)$ can be controlled for time of order $O(\varepsilon^{-3})$, which in turn gives control of $\nm{\frac1{Z_{,\aa}}(t)-1}_{L^\infty}$ for the same time period.  The constants $m_0>0$ is chosen so that $ \nm{\frac1{Z_{,\aa}}(0)-1}_{L^\infty}<1$.\footnote{Assuming $c$ is the optimal constant so that \eqref{40} holds, we choose $m_0$ such that $m_0^2<\frac{1}{4c^2}$.}  
Since  $ \nm{\frac1{Z_{,\aa}}(0)-1}_{L^\infty}$  can be arbitrarily close to 1, the slope of the (initial) interface can be arbitrary large.  Observe also that no assumption is imposed on the magnitude of $Z_t$. So the magnitude of the (initial) velocity can be arbitrary large.
\end{remark}
\begin{remark}\label{3.3}
Part 2 of Theorem~\ref{thm:main2} is a direct consequence of part 1. Observe that the rescaled data $(\bar Z_t^\varepsilon(0), Z^\varepsilon(0))$ satisfies 
$$\nm{\partial_\aa \frac1{Z^\varepsilon_{,\aa}}(0)}_{\dot H^{1/2}}+\nm{\partial_\aa^2 \bar Z_t^\varepsilon(0)}_{L^2}=\varepsilon \nm{\partial_\aa \frac1{Z_{,\aa}}(0)}_{\dot H^{1/2}}+\varepsilon\nm{\partial_\aa^2 \bar Z_t(0)}_{L^2}.$$
 So by part 1, the assumption \eqref{38} implies that the lifespan of the rescaled solution $(\bar Z_t^\varepsilon, Z^\varepsilon)$ is of order $O(\varepsilon^{-3})$, therefore the solution $(\bar Z_t, Z)$ has a lifespan of order $O(\varepsilon^{-5/2})$. In fact, the same scaling argument shows that for data satisfying \eqref{38} and $(Z_t(0),\frac 1{Z_{,\aa}}(0)-1)\in\cap_{\frac12\le s\le J}\dot H^s\times \dot H^{s-1/2}$, the life span of the solution is at least of order $O(\varepsilon^{-3+\frac1{2J-2}})$, and the solution remains as smooth as the initial data during this time.
\end{remark}

By Remark~\ref{3.3}, it suffices to prove part 1 of Theorem~\ref{thm:main2}.

\begin{remark}
In part 1 of Theorem~\ref{thm:main2}, the smallness assumption is only imposed on the scaling invariant quantity $\nm{\frac1{Z_{,\aa}}(0)}_{\dot H^{1/2}(\mathbb R)}+\nm{\bar Z_{t,\aa}(0)}_{L^2(\mathbb R)}$ and the quantity 
$\nm{\partial_\aa \frac1{Z_{,\aa}}(0)}_{\dot H^{1/2}(\mathbb R)}+\nm{\partial_\aa^2\bar Z_{t}(0)}_{L^2(\mathbb R)}$
which involves one higher order derivative.  In this paper we do not try to lower this higher than scaling derivative to avoid technicality. We note that in all earlier works on long time existence \cite{wu3, wu4,  gms, ad, ip, dipp, hit,  it, wang1,  ai, bfp}, smallness is assumed on norms involving derivatives both above and below the scaling,  so the smallness of the quantities can not be preserved after rescaling. And in all these earlier results, smallness is assumed on the slope of the initial interface. 

\end{remark}

\begin{remark}
In the regime where $\nm{\frac1{Z_{,\aa}}(t)-1}_{L^\infty}< 1$, the interface $Z=Z(\aa,t)$, $\aa\in\mathbb R$ is a graph.
\end{remark}

\section{The proof of Theorem~\ref{thm:main2}}\label{proof}

\subsection{Some additional quintic correcting functionals}

 To prove Theorem~\ref{thm:main2}, we need some additional quintic correcting functionals to the energies $\mathfrak E_j(t)$, for $j\ge 2$. The goal is to have no loss of derivatives when controlling the higher order energy growth,  to not involve $\nm{\frac 1 {Z_{,\aa}}-1}_{L^\infty(\mathbb R)}$ and any norms of  order  lower than scaling
 when bounding the cubic growth rates for the lower order energies,  and to have proper control on the growth of $E_1(t)E_3(t)$.

Let $j\ge 2$. Observe that the factors $ D_t^{j+1}Z_t$ and $\bar{\mathcal P D_t^j\bar Z_t}$ in $R^{(m)}_{\bar j, l, \bar i, k}$ of \eqref{25} have more derivatives than are controlled by $\mathfrak E_j(t)$;
we perform an "integration by parts" in the time variable, moving one $D_t$  from the factors $ D_t^{j+1}Z_t$ and $\bar{\mathcal P D_t^j\bar Z_t}$  
onto  others,  resulting in the following quintic correcting functional.

 Let 
\begin{equation}\label{f}
\begin{aligned}
F^{(m)}_{\bar j;l,\bar i, k}(t)&=
\iint \frac{\bar{\mathcal P D_t^{j-1}\bar Z_t} \mathfrak D_t^{m}(\mathfrak D_t^l\theta\,\mathfrak D_t^i \bar\theta\,\mathfrak D_t^{k}\theta)}{(\aa-\bb)^2}\,d\aa d\bb
\\&+
\iint\paren{b_\aa+b_\bb-2\frac{b(\aa)-b(\bb)}{\aa-\bb}}\frac{D_t^jZ_t\,\mathfrak D_t^{m}(\mathfrak D_t^l\theta\,\mathfrak D_t^i\bar\theta\,\mathfrak D_t^{k}\theta)}{(\aa-\bb)^2}\,d\aa d\bb
\end{aligned}
\end{equation}
Computing using \eqref{dteab} and \eqref{eq:c7} yields
\begin{equation}\label{31}
\begin{aligned}
&\frac d{dt} F^{(m)}_{\bar j;l,\bar i, k}(t)=
\iint \frac{\paren{\bar{\mathcal P D_t^j\bar Z_t+\bracket{D_t,\mathcal P}D_t^{j-1}\bar Z_t}} \mathfrak D_t^{m}(\mathfrak D_t^l\theta\,\mathfrak D_t^i \bar\theta\,\mathfrak D_t^{k}\theta)}{(\aa-\bb)^2}\,d\aa d\bb
\\&+
\iint \frac{ \bar{\mathcal P D_t^{j-1}\bar Z_t} \mathfrak D_t^{m+1}(\mathfrak D_t^l\theta\,\mathfrak D_t^i \bar\theta\,\mathfrak D_t^{k}\theta)}{(\aa-\bb)^2}\,d\aa d\bb
\\&+
\iint \paren{b_\aa+b_\bb-2\frac{b(\aa)-b(\bb)}{\aa-\bb}}\frac{ \bar{\mathcal P D_t^{j-1}\bar Z_t}   \mathfrak D_t^{m}(\mathfrak D_t^l\theta\,\mathfrak D_t^i \bar\theta\,\mathfrak D_t^{k}\theta)}{(\aa-\bb)^2}\,d\aa d\bb
\\&+
\iint\paren{b_\aa+b_\bb-2\frac{b(\aa)-b(\bb)}{\aa-\bb}}\frac{D_t^{j+1}Z_t\,\mathfrak D_t^{m}(\mathfrak D_t^l\theta\,\mathfrak D_t^i\bar\theta\,\mathfrak D_t^{k}\theta)}{(\aa-\bb)^2}\,d\aa d\bb\\&+
\iint\paren{b_\aa+b_\bb-2\frac{b(\aa)-b(\bb)}{\aa-\bb}}\frac{D_t^{j}Z_t\,\mathfrak D_t^{m+1}(\mathfrak D_t^l\theta\,\mathfrak D_t^i\bar\theta\,\mathfrak D_t^{k}\theta)}{(\aa-\bb)^2}\,d\aa d\bb\\&+
\iint\paren{\partial_\aa D_t b+\partial_\bb D_tb-2\frac{D_tb(\aa)-D_t b(\bb)}{\aa-\bb}}\frac{D_t^{j}Z_t\,\mathfrak D_t^{m}(\mathfrak D_t^l\theta\,\mathfrak D_t^i\bar\theta\,\mathfrak D_t^{k}\theta)}{(\aa-\bb)^2}\,d\aa d\bb\\&+
\iint
\paren{6\frac{(b(\aa)-b(\bb))^2}{(\aa-\bb)^2}-4(b_\aa+b_\bb)\frac {(b(\aa)-b(\bb))}{(\aa-\bb)} +2b_\aa b_\bb}\frac{D_t^{j}Z_t\,\mathfrak D_t^{m}(\mathfrak D_t^l\theta\,\mathfrak D_t^i\bar\theta\,\mathfrak D_t^{k}\theta)}{(\aa-\bb)^2}\,d\aa d\bb;
\end{aligned}
\end{equation}
in the sum $R^{(m)}_{\bar j; l, \bar i, k}(t)+ \frac d{dt} F^{(m)}_{\bar j;l,\bar i, k}(t)$ those terms containing the factors $D_t^{j+1}Z_t$ and $\mathcal P D_t^j \bar Z_t$ are all cancelled out. 

We also need the following correcting term for $\frac1{4\pi}\bar{R^{(0)}_{\bar 0; 1,\bar 1,1}}$ in $R_{IC, 2}$, cf. \eqref{remainder1}-\eqref{25}. Define
\begin{equation}\label{225}
D_2(t)=\frac1{4\pi}\iint\paren{\mathbb H b_\aa }
\frac{\bar\theta\,\mathfrak D_t\theta\mathfrak D_t\bar\theta\mathfrak D_t\theta}{(\aa-\bb)^2}\,d\aa\,d\bb,\qquad D_3(t)=D_4(t)=0;
\end{equation}
we have by \eqref{dteab}, \eqref{eq:c7},
\begin{equation}\label{226}
\begin{aligned}
\frac d{dt} D_2(t)&= \frac1{4\pi}\iint \partial_\aa D_t \mathbb H b 
\frac{\bar\theta\,\mathfrak D_t\theta\mathfrak D_t\bar\theta\mathfrak D_t\theta}{(\aa-\bb)^2}\,d\aa\,d\bb+\frac1{4\pi}\iint  \mathbb H b_\aa 
\frac{\mathfrak D_t\paren{\bar\theta\,\mathfrak D_t\theta\mathfrak D_t\bar\theta\mathfrak D_t\theta}}{(\aa-\bb)^2}\,d\aa\,d\bb\\&+
\frac1{4\pi}\iint \paren{b_\bb-2\frac{b(\aa)-b(\bb)}{\aa-\bb}} \mathbb H b_\aa
\frac{\bar\theta\,\mathfrak D_t\theta\mathfrak D_t\bar\theta\mathfrak D_t\theta}{(\aa-\bb)^2}\,d\aa\,d\bb.
\end{aligned}
\end{equation}

Now define
\begin{equation}\label{62}
\lambda^{j}:=D_\aa\Theta^{(j)}(\aa,t)-D_\bb \Theta^{(j)}(\bb, t).
\end{equation}
 We next introduce 
 \begin{equation}\label{63}
H_j(t):=\frac1{4\pi}\Re\paren{\iint  \frac{\bar{\lambda^j}\, \lambda^j \,\theta\,\bar\theta+\bar{\lambda^j}\,\bar{\lambda^j  }\,\theta\,\theta}{(\aa-\bb)^2}\,d\aa\,d\bb- \iint  \frac{\mathfrak D_t^j \bar\theta\,\mathfrak D_t^j \theta\, \theta\,\bar\theta+\mathfrak D_t^j \bar\theta\,\mathfrak D_t^j \bar{\theta}\,\theta\,\theta}{(\aa-\bb)^2}\,d\aa\,d\bb}
\end{equation}
 to deal with the derivative loss in the term 
$$2\Re\sum_{l=0}^{j-1}\int i\, \partial_\aa\bar\Thj(\P_H D_t)^{l+1}\P_H\paren{G^{(j-l)}-D_t\P_H G^{(j-1-l)}}\,d\aa-I_{1,j},
$$
and to ensure the control of $E_1(t)E_3(t)$. We compute by \eqref{dteab} and using the symmetry \eqref{26} to get
\begin{equation}\label{64} 
\begin{aligned}
\frac d{dt} H_j&=\frac1{\pi}\Re\paren{ \iint  \bar{D_\aa\Theta^{(j)} }\frac{\mathfrak D_t \lambda^j \,\theta\,\bar\theta+\mathfrak D_t\bar{\lambda^j  }\,\theta\,\theta}{(\aa-\bb)^2}\,d\aa\,d\bb- \iint D_t^j Z_t \frac{\mathfrak D_t^{j+1} \theta\, \theta\,\bar\theta+\mathfrak D_t^{j+1} \bar{\theta}\,\theta\,\theta}{(\aa-\bb)^2}\,d\aa\,d\bb}\\&+R_{H, j}
\end{aligned}
\end{equation}
where
\begin{equation}\label{rhj}
\begin{aligned}
R_{H,j}&=\frac1{2\pi}\Re\iint\frac{\paren{|\lambda^j|^2-|\mathfrak D_t^j\theta|^2}\,\theta\,\mathfrak D_t\bar\theta+\paren{(\bar{\lambda^j})^2-(\mathfrak D_t^j\bar\theta)^2}\,\theta\,\mathfrak D_t\theta}{(\aa-\bb)^2}\,d\aa\,d\bb\\&+
\frac1{4\pi}\Re\iint\paren{b_\aa+b_\bb-2\frac{b(\aa)-b(\bb)}{\aa-\bb}}\frac{\paren{|\lambda^j|^2-|\mathfrak D_t^j\theta|^2}\,|\theta|^2+\paren{(\bar{\lambda^j})^2-(\mathfrak D_t^j\bar\theta)^2}\,\theta^2}{(\aa-\bb)^2}\,d\aa\,d\bb.
\end{aligned}
\end{equation}
Observe that in the sum $ I_{1,j}+\frac d{dt}H_j$, the term 
$$\frac1{\pi}\Re \iint D_t^j Z_t \frac{\mathfrak D_t^{j+1} \theta\, \theta\,\bar\theta+\mathfrak D_t^{j+1} \bar{\theta}\,\theta\,\theta}{(\aa-\bb)^2}\,d\aa\,d\bb$$
in $ I_{1,j}$ is replaced with 
$$\frac1{\pi}\Re \iint  \bar{D_\aa\Theta^{(j)} }\frac{\mathfrak D_t \lambda^j \,\theta\,\bar\theta+\mathfrak D_t\bar{\lambda^j  }\,\theta\,\theta}{(\aa-\bb)^2}\,d\aa\,d\bb,$$
which matches better with the corresponding term in $2\Re \int i\, \partial_\aa\bar\Thj(\P_H D_t)^{j}\P_H(G^{(1)})\,d\aa$, and the harmless terms in $R_{H,j}$.

Let
\begin{equation}\label{dj}
F_j(t):=\frac1{2\pi}\sum_{l=0}^{j-1}\sum_{k=0}^l F^{(l-k)}_{\bar j; k, \bar 0, j-l-1}+
\frac1{4\pi}\sum_{l=0}^{j-2}\sum_{k=0}^{j-l-2} (-1)^k F^{(0)}_{\bar j; 1+l,\bar k, j-l-2-k}+
\frac1{4\pi}\sum_{k=0}^{j-1}(-1)^k F^{(0)}_{\bar j; 0,\bar k, j-k-1},
\end{equation}
\begin{equation}\label{32}
\mathcal R^{(m)}_{\bar j; l, \bar i, k}(t):=R^{(m)}_{\bar j; l, \bar i, k}(t)+ \frac d{dt} F^{(m)}_{\bar j;l,\bar i, k}(t);
\end{equation}
and define
\begin{equation}\label{energyj}
\begin{aligned}
\mathcal E_j(t):&=\mathfrak E_j(t)-\Re F_j(t)+\Re D_j(t)-H_j(t)\\&=E_j(t)-\Re\paren{\int i\partial_\aa\Thj\bar{\P_H \Gj}\,d\aa+C_{1,j}(t)+ C_{2,j}(t)+F_j(t)-D_j(t)+H_j(t)}.
\end{aligned}
\end{equation}
By Theorem~\ref{th:main1} and \eqref{32}, we have
\begin{equation}\label{34}
\frac d{dt} \mathcal E_j(t)=\mathcal R_j(t)
\end{equation}
where
\begin{equation}\label{33}
\begin{aligned}
\mathcal R_j(t)&=\paren{2\Re\sum_{l=0}^{j-1}\int i\, \partial_\aa\bar\Thj(\P_H D_t)^{l+1}\P_H\paren{G^{(j-l)}-D_t\P_H G^{(j-1-l)}}\,d\aa-I_{1,j}-\frac d{dt}H_j}\\&+
\paren{\Re\int i\,\partial_\aa \bar\Thj\P_H\paren{G^{(j+1)}-D_t\P_H G^{(j)}}\,d\aa-I_{2,j}}+ \mathcal R_{IC,j}+\Re\frac d{dt} D_j(t),
\end{aligned}
\end{equation}
and 
\begin{equation}\label{remainder3}
\mathcal R_{IC,j}:=R_{IC,j}-\frac d{dt} \Re F_j(t)
\end{equation}
 is as given in \eqref{remainder1}, except that the quantities $R^{(l-k)}_{\bar j; k, \bar 0, j-l-1}$, $R^{(0)}_{\bar j; 1+l, \bar k, j-l-2-k}$ and $R^{(0)}_{\bar j; 0, \bar k, j-k-1}$ are replaced by $\mathcal R^{(l-k)}_{\bar j; k, \bar 0, j-l-1}$, $\mathcal R^{(0)}_{\bar j; 1+l, \bar k, j-l-2-k}$ and $\mathcal R^{(0)}_{\bar j; 0, \bar k, j-k-1}$ respectively.

\subsection{The main components in the proof of Theorem~\ref{thm:main2}}We are now ready to prove Theorem~\ref{thm:main2}. We begin with the following lemmas.

\begin{lemma}\label{lemma:4.1}
Let $0<\delta<1$. Assume that 
\begin{equation}\label{59}
\nm{\frac1{Z_{,\aa}}-1}_{L^\infty(\mathbb R)}\le 1-\delta.
\end{equation}
Then for any holomorphic function $f$, i.e. $\P_H f=f$, we have 
\begin{equation}\label{58}
\delta\nm{f}_{\dot H^{1/2}(\mathbb R)}\le \nm{\P_H\paren{\frac f{\bar Z_{,\aa}}}}_{\dot H^{1/2}(\mathbb R)}.
\end{equation}

\end{lemma}
\begin{proof}
We know 
$$f=\P_H f=\P_H\paren{\frac f{\bar Z_{,\aa}}}+\P_H \braces{f\paren{1-\frac1{\bar Z_{,\aa}}}}$$
By \eqref{halfholo},
$$\nm{\P_H \braces{f\paren{1-\frac1{\bar Z_{,\aa}}}}}_{\dot H^{1/2}}^2=\frac1{2\pi}\iint\frac{|f(\aa)-f(\bb)|^2}{(\aa-\bb)^2}\abs{1-\frac1{Z_{,\aa}}}^2\,d\aa\,d\bb\le(1- \delta)^2\nm{f}_{\dot H^{1/2}}^2.
$$
This gives \eqref{58}.
\end{proof}

\begin{lemma}[Sobolev]\label{lemma:4.2} Assume that $\nm{1-\frac1{Z_{,\aa}}}_{L^\infty}\le 1$, and $\frac1{Z_{,\aa}}$ is sufficiently smooth.
Then
\begin{equation}\label{53}
\nm{\frac1{Z_{,\aa}}-1}_{L^\infty(\mathbb R)}^2\le 18 \nm{\frac1{Z_{,\aa}}-1}_{L^2(\mathbb R)}\nm{D_\aa \frac1{Z_{,\aa}^2}}_{L^2(\mathbb R)}.
\end{equation}

\end{lemma}
\begin{proof}
\eqref{53} is essentially a Sobolev inequality. By Fundamental Theorem of Calculus and  Cauchy-Schwarz inequality, we know for any smooth complex valued function $f$ tending to 1 at infinity, 
$$\abs{\frac14\paren{f^4-1}-\frac13\paren{f^3-1}}\le \int \abs{\paren{f^3-f^2}\partial_\aa f}\,d\aa\le \nm{f-1}_{L^2}\nm{f^2\partial_\aa f}_{L^2}.$$
On the other hand,
$$\frac14\paren{f^4-1}-\frac13\paren{f^3-1}=\paren{f-1}^2\paren{\frac14 f^2+\frac16 f+\frac1{12}}.$$
It is easy to check that if $\abs{1-f}\le 1$, then $\Re \paren{\frac14 f^2+\frac16 f+\frac1{12}}\ge \frac 1{36}$, 
which implies 
$$\abs{\frac14\paren{f^4-1}-\frac13\paren{f^3-1}}\ge \frac1{36} \abs{f-1}^2.$$
Replacing $f$ by $\frac1{Z_{,\aa}}$ yields \eqref{53}.
\end{proof}

Part 1 of Theorem~\ref{thm:main2} follows from the following  Propositions.

\begin{proposition}\label{prop:4.3}
Let $0<\delta<1$, and $T_0>0$. Assume that 
\begin{equation}\label{42}
\sup_{t\in [0, T_0]}\nm{\frac1{Z_{,\aa}}(t)-1}_{L^\infty}\le 1-\delta,\qquad \sup_{t\in[0, T_0]} L(t)\le 2\epsilon.
\end{equation}
There is a constant $0<\epsilon_0(\delta)\le 1$, such that for all $0<\epsilon\le\epsilon_0(\delta)$, we have 1.
\begin{equation}\label{46}
c_2(\delta)L(t)^2\le \sum_{j=2}^4\mathcal E_j(t)\le c_1 L(t)^2, \qquad\forall\  t\in [0, T_0]
\end{equation}
where $c_1$ and  $c_2(\delta)$ are some positive constants. 2.
\begin{equation}\label{47}
\sum_{j=2}^4 \mathcal E_j(t)\le \sum_{j=2}^4 \mathcal E_j(0)+\epsilon^5 t, \qquad\forall \ t\in [0, T_0]
\end{equation}

\end{proposition}

\begin{proposition}\label{prop:4.4}
Assume that the assumption of Proposition~\ref{prop:4.3} holds. There is a constant $0<\epsilon_0(\delta)\le 1$, such that for all $0<\epsilon\le \epsilon_0(\delta)$,  we have
1. 
\begin{align}\label{48}
(1+\delta)^{-1}E_1(t)\le 
\nm{\frac1{Z_{,\aa}}(t)-1}_{L^2}^2 +\nm{\Theta^{(2)}(t)}_{\dot H^{1/2}}^2\le (1+\delta) E_1(t), \qquad\forall\  t\in[0, T_0];\\
\label{49}
(1+\delta)^{-1}E_3(t)\le   
\frac14 \nm{D_\aa \frac1{Z_{,\aa}^2}(t)}_{L^2}^2+\nm{\Theta^{(4)}(t)}_{\dot H^{1/2}}^2  \le (1+\delta) E_3(t), \qquad\forall \ t\in[0, T_0];\\ \label{50}
(1+\delta)^{-1}E_1(t)E_3(t)\le \mathfrak E_1(t)\mathcal E_3(t)\le (1+\delta)E_1(t)E_3(t),\qquad\forall \ t\in[0, T_0].
\end{align}
2. There is a constant $c_3(\delta^{-1})>0$, such that
\begin{equation}\label{51}
\mathfrak E_1(t)\mathcal E_3(t)\le \mathfrak E_1(0)\mathcal E_3(0)e^{c_3(\delta^{-1})\epsilon^3t},\qquad\forall \ t\in[0, T_0].
\end{equation}

\end{proposition}

Let $J\ge 2$ and assume that  the initial data satisfies $\paren{\bar Z_t(0), \frac 1{Z_{,\aa}}(0)-1}\in \cap_{\frac12\le s\le J}\dot H^{s}(\mathbb R)\times \dot H^{s-\frac12}(\mathbb R)$. By the local existence result in \cite{agr2}, we know there is a unique classical solution for the Cauchy problem of the water wave equation \eqref{2dinterface}-\eqref{A1b} for some positive time period  $[0, T]$, and the solution exists and is as regular as the initial data for as long as $L(t)$ and $\nm{Z_{,\aa}(t)}_{L^\infty}$ remain finite.

Now assume that the initial data satisfies \eqref{37}. 
Let $\delta>0$ be fixed, such that $\nm{1-\frac1{Z_{,\aa}}(0)}_{L^\infty}\le 1-2\delta$,  
 $\epsilon:= \max\braces{1, \sqrt{\frac{c_1}{c_2(\delta)}}}\varepsilon$, so $L(0)\le \min\braces{1, \sqrt{\frac{c_2(\delta)}{c_1}}}\epsilon$. 
Let $T_0>0$ be the maximum time  such  that  the solution satisfies 
\begin{equation}\label{60}
\sup_{t\in [0, T_0]}\nm{\frac1{Z_{,\aa}}(t)-1}_{L^\infty}\le 1-\delta,\qquad \sup_{t\in[0, T_0]} L(t)\le 2\epsilon.
\end{equation}
By Proposition ~\ref{prop:4.4},   there is a $\epsilon_0=\epsilon_0(\delta)>0$, such that for all $\epsilon\le \epsilon_0$, 
\begin{equation}\label{54}
E_1(t)E_3(t)\le (1+\delta)^2E_1(0)E_3(0)e^{c_3(\delta^{-1})\epsilon^3t}, \qquad\forall\ \, 0\le t\le T_0
\end{equation}
so 
\begin{equation}\label{55}
E_1(t)E_3(t)\le (1+\delta)^2 m_0^2 (1+\delta),\qquad\forall \ \, 0\le t\le T_1:=\min\{T_0, \frac{\delta}{2c_3(\delta^{-1})\epsilon^3}\}.
\end{equation}
By Lemma~\ref{lemma:4.2} and \eqref{48}, \eqref{49},  there is a constant $c_0$ such that 
\begin{equation}\label{56}
\nm{\frac1{Z_{,\aa}}(t)-1}_{L^\infty(\mathbb R)}^4\le c_0  (1+\delta)^2 E_1(t) E_3(t), \qquad\forall\ \, t\in[0, T_0].
\end{equation}
Let
$m_0^2=\frac{(1-2\delta)^4}{c_0(1+\delta)}$. Then we have
\begin{equation}\label{57} 
\nm{\frac1{Z_{,\aa}}(t)-1}_{L^\infty(\mathbb R)}^4\le c_0 (1+\delta)^5m_0^2 =(1+\delta)^4(1-2\delta)^4=(1-\delta-2\delta^2)^4,\qquad \forall\ \, t\in [0, T_1].
\end{equation}
By Proposition~\ref{prop:4.3}, there is $\epsilon_0=\epsilon_0(\delta)$, such that for all $\epsilon\le\epsilon_0$,
\begin{equation}
L(t)^2\le \frac1{c_2(\delta)}\paren{\sum_{j=2}^4 \mathcal E_j(0)+ \epsilon^5 t}\le \epsilon^2+\frac{1}{c_2(\delta)} \epsilon^5 t\le 2\epsilon^2,\quad\text{for all }t\in [0, T_2],
\end{equation}
 where $T_2=\min\{T_0, \frac{c_2(\delta)}{\epsilon^3}\}$. By the maximally of $T_0$ we must have 
 $\min\{T_1, T_2\}< T_0$, so there is $\mathcal T_0=\mathcal T_0(\delta)>0$, such that $T_0\ge \frac{\mathcal T_0}{\epsilon^3}$. 
 This proves part 1 of Theorem~\ref{thm:main2}. 
 
In the remainder of this paper we prove Propositions~\ref{prop:4.3} and~\ref{prop:4.4}.\footnote{We will be brief on straightforward but tedious details  for the sake of a concise presentation. We suggest the reader get familiar with the basic tools in Appendices~\ref{iden},~\ref{ineq} before continuing.
}
 We assume that  \eqref{42} holds, namely for some $0<\delta<1$,  $0<\epsilon\le 1$ and $T_0>0$,  
 \begin{equation}\label{65}
\sup_{t\in [0, T_0]}\nm{\frac1{Z_{,\aa}}(t)-1}_{L^\infty}\le 1-\delta,\qquad \sup_{t\in[0, T_0]} L(t)\le 2\epsilon.
\end{equation}
We  assume  $t\in [0, T_0]$ in what follows.

 \subsection{The proof of Propositions~\ref{prop:4.3} and ~\ref{prop:4.4}}

From definition \eqref{Thj},  \eqref{Gj}, and equations \eqref{qa}, \eqref{q1} we know 
\begin{align}\label{66}
\Theta^{(0)}=Q, \quad \Theta^{(1)}=i(Z-\aa),\quad  \Theta^{(2)}=-i\P_H b=-i\P_H\frac{\bar Z_t}{\bar Z_{,\aa}},\\
\label{67}
D_\aa\Theta^{(0)}=\bar Z_t,\qquad D_\aa\Theta^{(1)}=i\paren{1-\frac1{Z_{,\aa}}};\qquad\text{and}
\\ \label{68}
\Theta^{(j+2)}=-i\P_H(\frac1{| Z_{,\aa}|^2}\partial_\aa\Theta^{(j)})+\P_H(G^{(j)}), \qquad \text{for }j\ge 1.
\end{align}
Therefore
\begin{equation}\label{69}
\begin{aligned}
E_j(t)&=\Re \paren{\int i\,\partial_\aa\Thja \overline{\Thja}\,d\aa-\int i\,\partial_\aa\Thj \overline{\Thjb}\,d\aa}
\\&=\Re \paren{\int i\,\partial_\aa\Thja \overline{\Thja}\,d\aa+\int |D_\aa\Thj|^2\,d\aa+\int  i\,\partial_\aa\bar\Thj \P_H(G^{(j)})\,d\aa}.
\end{aligned}
\end{equation}

\subsubsection{Quantities controlled by $L(t)$, $\nm{1-\frac1{Z_{,\aa}}}_{L^2}+\nm{Z_t}_{\dot H^{1/2}}$, and $\|Z_{t,\aa}\|_{\dot H^{1/2}}+\nm{\partial_\aa\frac1{Z_{,\aa}}}_{L^2}$}\label{quan}
We begin with deriving some basic estimates for  the quantities involved in the proof of Propositions~\ref{prop:4.3}, ~\ref{prop:4.4}. First note that by assumption \eqref{65}, 
\begin{equation}\label{75}
\sup_{t\in [0, T_0]}\nm{\frac1{Z_{,\aa}}(t)}_{L^\infty}\le 2, \qquad \sup_{t\in [0, T_0]}\nm{{Z_{,\aa}}(t)}_{L^\infty}\le \delta^{-1};
\end{equation}
and by interpolation and Sobolev embedding \eqref{eq:sobolev}, 
\begin{equation}\label{76}
\nm{Z_{t,\aa}}_{L^\infty}+\|Z_{t,\aa}\|_{\dot H^{1/2}}+\nm{\partial_\aa\frac1{Z_{,\aa}}}_{L^2}\lec L(t)\le\epsilon.
\end{equation}
We estimate, by \eqref{2dinterface}, 
$$\bar Z_{tt}=-i\frac{A_1-1}{Z_{,\aa}}+i\paren{1-\frac1{Z_{,\aa}}},$$
\eqref{A1b} and \eqref{eq:b10}, that
\begin{align}\label{185}
\nm{A_1-1}_{L^2}&\lec \nm{Z_t}_{\dot H^{1/2}}\|Z_{t,\aa}\|_{L^2}\lec \epsilon\nm{Z_t}_{\dot H^{1/2}},\\
\label{71}
\nm{Z_{tt}}_{L^2}&\lec \nm{1-\frac1{Z_{,\aa}}}_{L^2}+\nm{A_1-1}_{L^2} \lec \nm{1-\frac1{Z_{,\aa}}}_{L^2}+\epsilon\nm{Z_t}_{\dot H^{1/2}};
\end{align}
 by \eqref{eq:b13}, \eqref{hhalf42}, \eqref{A1b} that
\begin{equation}\label{73}
A_1\ge 1, \quad \nm{A_1-1}_{L^\infty}+\nm{A_1}_{\dot H^{1/2}}\lec \|Z_{t,\aa}\|_{L^2}^2\lec \epsilon^2;
\end{equation}
by \eqref{hhalf-1}, \eqref{2dinterface} and \eqref{73},
\begin{equation}\label{74}
\nm{Z_{tt}}_{\dot H^{1/2}}\lec \|Z_{t,\aa}\|_{L^2}^2+\nm{\frac1{Z_{,\aa}}}_{\dot H^{1/2}}\lec \epsilon, \quad \nm{Z_{tt}}_{L^\infty}\lec 1;
\end{equation}
by \eqref{dta1}, \eqref{ba}, \eqref{eq:b22} and \eqref{eq:b12} that
\begin{align}\label{72}
\nm{D_t A_1}_{L^2}&\lec \|Z_{t,\aa}\|_{L^2}\|Z_{tt}\|_{\dot H^{1/2}}+\|Z_{t,\aa}\|_{L^2}^2\nm{b_\aa}_{L^2}, \\\label{79} \ \nm{b_\aa-2\Re D_\aa Z_t}_{L^2}&\lec \|Z_{t,\aa}\|_{L^2}\nm{\frac1{Z_{,\aa}}}_{\dot H^{1/2}}\lec \epsilon^2;
\end{align}
this implies that 
\begin{equation}\label{77}
\nm{b_\aa}_{L^2}\lec \|Z_{t,\aa}\|_{L^2}\nm{\frac1{Z_{,\aa}}}_{\dot H^{1/2}}+\|Z_{t,\aa}\|_{L^2}\lec\epsilon, \quad\text{and} \ 
 \ \nm{D_t A_1}_{L^2}\lec \epsilon^2;
\end{equation}
and by \eqref{66} and  \eqref{quasi} that
\begin{equation}\label{78}
\nm{\partial_\aa\Th^{(2)}}_{L^2}\lec \epsilon,\quad\nm{\mathcal P \bar Z_t}_{L^2}\lec \epsilon^2,\quad  \nm{D_t^2 Z_t}_{L^2}\lec \|Z_{t,\aa}\|_{L^2}+ \epsilon^2\lec \epsilon.
\end{equation}

In what follows we pay particular attention to  the bounds $\nm{\partial_\aa\frac1{Z_{,\aa}}}_{L^2}$ and $\|Z_{t,\aa}\|_{\dot H^{1/2}}$ in the estimates for the sake of proving Proposition~\ref{prop:4.4}.
We  estimate, by \eqref{ba},  and \eqref{hhalf42}, \eqref{eq:b13}, 
\begin{equation}\label{98}
\nm{b_\aa-2\Re D_\aa Z_t}_{\dot H^{1/2}}+\nm{b_\aa-2\Re D_\aa Z_t}_{L^\infty}\lec \|Z_{t,\aa}\|_{L^2}\nm{\partial_\aa\frac1{Z_{,\aa}}}_{L^2}\lec\epsilon \nm{\partial_\aa\frac1{Z_{,\aa}}}_{L^2},
\end{equation}
so by \eqref{hhalf-2}, \eqref{eq:c26},
\begin{equation}\label{99}
\nm{b_\aa}_{\dot H^{1/2}}+ \nm{D_t\frac1{Z_{,\aa}}}_{\dot H^{1/2}}\lec \epsilon\nm{\partial_\aa\frac1{Z_{,\aa}}}_{L^2}+\|Z_{t,\aa}\|_{\dot H^{1/2}};
\end{equation}
this gives, by using \eqref{66} and \eqref{hhalf-2}, that 
\begin{equation}\label{145}
\nm{\partial_\aa\Th^{(2)}}_{\dot H^{1/2}}+ \nm{D_\aa\Th^{(2)}}_{\dot H^{1/2}}\lec \nm{b_\aa}_{\dot H^{1/2}}+\nm{\partial_\aa\frac1{Z_{,\aa}}}_{L^2}\nm{b_\aa}_{L^2}\lec \epsilon\nm{\partial_\aa\frac1{Z_{,\aa}}}_{L^2}+\|Z_{t,\aa}\|_{\dot H^{1/2}};
\end{equation}
and as a consequence of \eqref{98} and \eqref{76}, we also have
\begin{equation}\label{116}
\nm{b_\aa}_{L^\infty}\lec  \|Z_{t,\aa}\|_{L^2}\nm{\partial_\aa\frac1{Z_{,\aa}}}_{L^2}+\nm{Z_{t,\aa}}_{L^\infty}\lec \epsilon;
\end{equation}
and we record, by \eqref{eq:c26},
\begin{equation}\label{178}
\nm{D_t\frac1{Z_{,\aa}}}_{L^\infty}\lec \epsilon.
\end{equation}
Now by \eqref{A1b}, \eqref{eq:b10}, \eqref{eq:b11}, 
\begin{equation}\label{146}
\nm{\partial_\aa A_1}_{L^2}
\lec \nm{Z_{t,\aa}}_{L^2}\nm{Z_{t,\aa}}_{\dot H^{1/2}}\lec\epsilon \nm{Z_{t,\aa}}_{\dot H^{1/2}};
\end{equation}
and by \eqref{2dinterface}, 
\begin{equation}\label{111}
\nm{Z_{tt,\aa}}_{L^2}\lec \nm{\partial_\aa\frac1{Z_{,\aa}}}_{L^2}+\nm{Z_{t,\aa}}_{L^2}\nm{Z_{t,\aa}}_{\dot H^{1/2}}\lec  \nm{\partial_\aa\frac1{Z_{,\aa}}}_{L^2}+\epsilon\nm{Z_{t,\aa}}_{\dot H^{1/2}}.
\end{equation}
We write
\begin{equation}\label{112}
\bracket{Z_t, b; \bar Z_{t,\aa}}=-<Z_t, \,b,\, \bar Z_{t,\aa}>+\bar Z_{t,\aa}\bracket{Z_t, b; 1},
\end{equation}
estimating $\nm{<Z_t, \,b,\, \bar Z_{t,\aa}>}_{\dot H^{1/2}}$ by \eqref{eq:b20}, and the $\dot H^{1/2}$ norm of the second term by \eqref{hhalf-2} and \eqref{eq:b12}, \eqref{eq:b22} yield
\begin{equation}\label{113}
\nm{\bracket{Z_t, b; \bar Z_{t,\aa}}}_{\dot H^{1/2}}\lec \nm{Z_{t,\aa}}_{\dot H^{1/2}}\nm{b_\aa}_{L^2}\nm{Z_{t,\aa}}_{L^2}+
\nm{Z_{t,\aa}}_{L^2}\nm{Z_{t,\aa}}_{L^2}\nm{b_\aa}_{\dot H^{1/2}};
\end{equation}
from \eqref{dta1}, further using \eqref{hhalf42} gives
\begin{equation}\label{114}
\nm{D_t A_1}_{\dot H^{1/2}}\lec \nm{Z_{t,\aa}}_{L^2} \nm{Z_{tt,\aa}}_{L^2}+\nm{\bracket{Z_t, b; \bar Z_{t,\aa}}}_{\dot H^{1/2}}\lec \epsilon \nm{\partial_\aa\frac1{Z_{,\aa}}}_{L^2}+\epsilon\nm{Z_{t,\aa}}_{\dot H^{1/2}};
\end{equation}
by \eqref{quasi} and \eqref{hhalf-1}, \eqref{hhalf-2}, this gives
\begin{equation}\label{115}
\nm{Z_{ttt}}_{\dot H^{1/2}}\lec \epsilon \nm{\partial_\aa\frac1{Z_{,\aa}}}_{L^2}+\nm{Z_{t,\aa}}_{\dot H^{1/2}}.
\end{equation}
We also record, by \eqref{quasi}, \eqref{98}, \eqref{114}, \eqref{146} and  other relevant earlier estimates that
\begin{equation}\label{117}
\nm{\mathcal P\bar Z_t}_{\dot H^{1/2}}\lec \epsilon \nm{\partial_\aa\frac1{Z_{,\aa}}}_{L^2}+\epsilon\nm{Z_{t,\aa}}_{\dot H^{1/2}}.
\end{equation}
By \eqref{dta1}, \eqref{eq:b13}, \eqref{eq:b115}, \eqref{hhalf44},
\begin{equation}\label{169}
\nm{D_t A_1}_{L^\infty}\lec \nm{Z_{t,\aa}}_{L^2}\nm{Z_{tt,\aa}}_{L^2}+\nm{Z_{t,\aa}}_{L^4}^2\nm{b_\aa}_{L^2}\lec \epsilon \nm{\partial_\aa\frac1{Z_{,\aa}}}_{L^2}+\epsilon\nm{Z_{t,\aa}}_{\dot H^{1/2}},
\end{equation}
therefore by \eqref{quasi}, we also have
\begin{equation}\label{170}
\nm{\mathcal P\bar Z_t}_{L^\infty}\lec \epsilon \nm{\partial_\aa\frac1{Z_{,\aa}}}_{L^2}+\epsilon\nm{Z_{t,\aa}}_{\dot H^{1/2}},
\end{equation}
and 
\begin{equation}\label{171}
\nm{Z_{ttt}}_{L^\infty}\lec \nm{Z_{t,\aa}}_{L^\infty}+ \epsilon^2\lec \epsilon.
\end{equation}
Now by \eqref{ba}, \eqref{eq:c14}, and \eqref{eq:b10}, \eqref{eq:b11}, \eqref{eq:b12}, \eqref{99},
\begin{equation}\label{118}
\begin{aligned}
&\nm{D_t(b_\aa-2\Re D_\aa Z_t)}_{L^2}\lec \nm{Z_{tt}}_{\dot H^{1/2}}\nm{\partial_\aa\frac1{Z_{,\aa}}}_{L^2}+\nm{Z_{t,\aa}}_{L^2}\nm{D_t\frac1{Z_{,\aa}}}_{\dot H^{1/2}}\\&+\nm{Z_{t,\aa}}_{L^2}\nm{b_\aa}_{L^2}\nm{\partial_\aa\frac1{Z_{,\aa}}}_{L^2}\lec \epsilon \nm{\partial_\aa\frac1{Z_{,\aa}}}_{L^2}+\epsilon\nm{Z_{t,\aa}}_{\dot H^{1/2}},
\end{aligned}
\end{equation}
this gives, by \eqref{111}, \eqref{eq:c1-1},  \eqref{eq:c7}, \eqref{hhalf44} and earlier relevant estimates in this section that
\begin{equation}\label{120}
\nm{\partial_\aa D_t b}_{L^2}\lec \nm{D_t(b_\aa-2\Re D_\aa Z_t)}_{L^2}+ \nm{D_t D_\aa Z_t}_{L^2}+\nm{b_\aa}_{L^4}^2\lec  \nm{\partial_\aa\frac1{Z_{,\aa}}}_{L^2}+\epsilon\nm{Z_{t,\aa}}_{\dot H^{1/2}}.
\end{equation}
We estimate similarly, by \eqref{dta1}, \eqref{eq:c14}, \eqref{eq:c40}, and \eqref{eq:b10}, \eqref{eq:b11}, \eqref{eq:b12}, \eqref{eq:b115}, \eqref{hhalf44} and the estimates above in this section to obtain 
\begin{equation}\label{121}
\nm{D_t^2A_1}_{L^2} \lec \epsilon \nm{\partial_\aa\frac1{Z_{,\aa}}}_{L^2}+\epsilon\nm{Z_{t,\aa}}_{\dot H^{1/2}},
\end{equation}
this in turn gives, by \eqref{quasi}, \eqref{eq:c32}, because $\mathcal P \bar Z_{tt}=D_t\mathcal P\bar Z_t+\bracket{\mathcal P, D_t}\bar Z_t$,  that
\begin{align}\label{122}
\nm{\mathcal P\bar Z_{tt}}_{L^2}+\nm{\bracket{\mathcal P, D_t}\bar Z_{t}}_{L^2} &\lec \epsilon \nm{\partial_\aa\frac1{Z_{,\aa}}}_{L^2}+\epsilon\nm{Z_{t,\aa}}_{\dot H^{1/2}},\\ \label{123}
\nm{D_t^3 Z_t}_{L^2}&\lec \nm{\partial_\aa\frac1{Z_{,\aa}}}_{L^2}+\epsilon\nm{Z_{t,\aa}}_{\dot H^{1/2}}.
\end{align}
We compute $D_t^2\frac1{Z_{,\aa}}$ starting from \eqref{eq:c26} and use \eqref{eq:c28}, \eqref{hhalf44} and the earlier estimates in this section and get
\begin{equation}\label{119}
\nm{D_t^2\frac1{Z_{,\aa}}}_{L^2}\lec \nm{\partial_\aa\frac1{Z_{,\aa}}}_{L^2}+\nm{Z_{t,\aa}}_{\dot H^{1/2}}.
\end{equation}
From \eqref{81} we estimate by \eqref{eq:G}, \eqref{eq:b18}, \eqref{eq:b21} to get
\begin{equation}\label{156}
\nm{\partial_\aa\Th^{(3)}}_{L^2}\lec \nm{\partial_\aa\frac1{Z_{,\aa}}}_{L^2}+\nm{\partial_\aa\frac1{Z_{,\aa}}}_{L^2}\nm{Z_{t,\aa}}_{L^2}^2\lec \nm{\partial_\aa\frac1{Z_{,\aa}}}_{L^2};
\end{equation}
we compute by \eqref{eq:c1-1} and the definition \eqref{Thj}:
\begin{equation}\label{157}
D_tD_\aa\Th^{(2)}=D_\aa \Th^{(3)}+D_\aa\bracket{\P_A, b}\partial_\aa \Th^{(2)}-D_\aa Z_t D_\aa\Th^{(2)},
\end{equation}
and using \eqref{eq:b10}, \eqref{eq:b23}, \eqref{hhalf44} and the earlier estimates we obtain
\begin{equation}\label{158}
\begin{aligned}
\nm{D_tD_\aa\Th^{(2)}}_{L^2}&\lec \nm{\partial_\aa\Th^{(3)}}_{L^2}+\nm{b_\aa}_{\dot H^{1/2}}\nm{\partial_\aa\Th^{(2)}}_{L^2}+\nm{Z_{t,\aa}}_{L^4}\nm{\partial_\aa\Th^{(2)}}_{L^4}\\&\lec \nm{\partial_\aa\frac1{Z_{,\aa}}}_{L^2}+\epsilon\nm{Z_{t,\aa}}_{\dot H^{1/2}}.
\end{aligned}
\end{equation}

Now by \eqref{ba}, \eqref{eq:b11}, \eqref{eq:b23},
\begin{equation}\label{140}
\nm{\partial_\aa(b_\aa-2\Re D_\aa Z_t)}_{L^2}\lec \paren{\nm{Z_{t,\aa}}_{\dot H^{1/2}}+\nm{Z_{t,\aa}}_{L^\infty}}\nm{\partial_\aa\frac1{Z_{,\aa}}}_{L^2}\lec \epsilon \nm{\partial_\aa\frac1{Z_{,\aa}}}_{L^2};
\end{equation}
therefore
\begin{equation}\label{142}
\nm{\partial_\aa b_\aa}_{L^2}\lec \nm{\partial_\aa(b_\aa-2\Re D_\aa Z_t)}_{L^2}+\nm{Z_{t,\aa\aa}}_{L^2}+\nm{Z_{t,\aa}}_{L^\infty}\nm{\partial_\aa\frac1{Z_{,\aa}}}_{L^2};
\end{equation}
and we compute $\partial_\aa D_t\frac1{Z_{,\aa}}$ by \eqref{eq:c26}, and use \eqref{140}, \eqref{98}, \eqref{76} to obtain
\begin{equation}\label{141}
\nm{\partial_\aa D_t\frac1{Z_{,\aa}}}_{L^2}\lec \nm{Z_{t,\aa\aa}}_{L^2}+\epsilon\nm{\partial_\aa\frac 1{Z_{,\aa}}}_{L^2}.
\end{equation}
Now by \eqref{A1b}, 
and 
\eqref{hhalf42},
\begin{equation}\label{147}
\nm{\partial_\aa A_1}_{\dot H^{1/2}}\lec \nm{Z_{t,\aa}}_{L^2}\nm{Z_{t,\aa\aa}}_{L^2}\lec \epsilon\nm{Z_{t,\aa\aa}}_{L^2};
\end{equation}
therefore by \eqref{2dinterface}, \eqref{hhalf-2},
\begin{equation}\label{148}
\begin{aligned}
\nm{Z_{tt,\aa}}_{\dot H^{1/2}}&\lec \nm{\partial_\aa A_1}_{\dot H^{1/2}}+\nm{\partial_\aa\frac 1{Z_{,\aa}}}_{L^2}\nm{\partial_\aa A_1}_{L^2}+\nm{A_1}_{L^\infty}\nm{\partial_\aa\frac 1{Z_{,\aa}}}_{\dot H^{1/2}}\\&\lec
\nm{\partial_\aa\frac 1{Z_{,\aa}}}_{\dot H^{1/2}}+\epsilon \nm{Z_{t,\aa\aa}}_{L^2}+\epsilon^2\nm{\partial_\aa\frac 1{Z_{,\aa}}}_{L^2}.
\end{aligned}
\end{equation}
Starting from \eqref{ba}, we use \eqref{eq:c14} to expand, and use the Sobolev inequality \eqref{eq:sobolev} to estimate the $\dot H^{1/2}$ norm of $\bracket{\frac1{Z_{,\aa}}, b; Z_{t,\aa}}$ and $\bracket{Z_{t}, b; \partial_\aa \frac1{Z_{,\aa}}}$,
we have, by \eqref{hhalf42}, \eqref{eq:b12}, \eqref{eq:b22}, \eqref{eq:b112}, \eqref{eq:b111} and \eqref{eq:sobolev},
\begin{equation}\label{149}
\begin{aligned}
\nm{D_t(b_\aa-2\Re D_\aa Z_t)}_{\dot H^{1/2}}&\lec \nm{\partial_\aa D_t\frac1{Z_{,\aa}}}_{L^2}\nm{Z_{t,\aa}}_{L^2}+\nm{\partial_\aa \frac1{Z_{,\aa}}}_{L^2}\nm{Z_{tt,\aa}}_{L^2}\\&+\epsilon^2 \paren{ \nm{\partial_\aa\frac 1{Z_{,\aa}}}_{\dot H^{1/2}}+\nm{Z_{t,\aa\aa}}_{L^2} +\nm{\partial_\aa\frac 1{Z_{,\aa}}}_{L^2}}\\&\lec 
\epsilon \paren{ \nm{\partial_\aa\frac 1{Z_{,\aa}}}_{\dot H^{1/2}}+\nm{Z_{t,\aa\aa}}_{L^2}} +\epsilon\nm{\partial_\aa\frac 1{Z_{,\aa}}}_{L^2};
\end{aligned}
\end{equation}
therefore by \eqref{eq:c1-1}, \eqref{eq:c7}, \eqref{hhalf-1}, \eqref{hhalf-2},
\begin{equation}\label{177}
\begin{aligned}
\nm{\partial_\aa D_t b}_{\dot H^{1/2}}&\lec \nm{D_t(b_\aa-2\Re D_\aa Z_t)}_{\dot H^{1/2}}+ \nm{(b_\aa)^2}_{\dot H^{1/2}}+\nm{D_\aa Z_{tt}}_{\dot H^{1/2}}+\nm{(D_\aa Z_t)^2}_{\dot H^{1/2}}\\&
\lec \nm{\partial_\aa\frac 1{Z_{,\aa}}}_{\dot H^{1/2}}+\epsilon \nm{Z_{t,\aa\aa}}_{L^2}+\epsilon\nm{\partial_\aa\frac 1{Z_{,\aa}}}_{L^2}+\epsilon\nm{Z_{t,\aa}}_{\dot H^{1/2}};
\end{aligned}
\end{equation}
we also have by \eqref{eq:b13}, \eqref{eq:b15} that 
\begin{equation}\label{149-1}
\begin{aligned}
\nm{D_t(b_\aa-2\Re D_\aa Z_t)}_{L^\infty}&\lec \nm{\partial_\aa D_t\frac1{Z_{,\aa}}}_{L^2}\nm{Z_{t,\aa}}_{L^2}+\nm{\partial_\aa \frac1{Z_{,\aa}}}_{L^2}\nm{Z_{tt,\aa}}_{L^2}\\&+\nm{b_\aa}_{L^\infty} \nm{\partial_\aa \frac1{Z_{,\aa}}}_{L^2}\nm{Z_{t,\aa}}_{L^2}\lec 
\epsilon^2.
\end{aligned}
\end{equation}
Compute $D_t^2\frac1{Z_{,\aa}}$ starting from \eqref{eq:c26},  and use \eqref{eq:c28},  \eqref{hhalf-1}, \eqref{hhalf-2}, \eqref{98},  \eqref{99},  \eqref{111},  \eqref{118}, \eqref{148} and \eqref{149}, and Sobolev embedding \eqref{eq:sobolev} to find
\begin{equation}\label{150}
\nm{D_t^2\frac1{Z_{,\aa}}}_{\dot H^{1/2}}\lec \nm{\partial_\aa\frac 1{Z_{,\aa}}}_{\dot H^{1/2}}+\epsilon\nm{Z_{t,\aa\aa}}_{L^2}+\epsilon \nm{\partial_\aa\frac 1{Z_{,\aa}}}_{L^2}.
\end{equation}
Now from \eqref{dta1}, using \eqref{eq:b11}, \eqref{eq:b23}, and \eqref{eq:b22}, \eqref{eq:b111}, \eqref{eq:b112}, we have
\begin{equation}\label{152}
\begin{aligned}
\nm{\partial_\aa D_tA_1}_{L^2}&\lec \nm{Z_{tt,\aa}}_{L^2}\paren{\nm{Z_{t,\aa}}_{L^\infty}+\nm{Z_{t,\aa}}_{\dot H^{1/2}}}\\&+\nm{b_\aa}_{L^\infty}\nm{Z_{t,\aa}}_{L^2}\nm{Z_{t,\aa}}_{\dot H^{1/2}}+\nm{b_\aa}_{\dot H^{1/2}}\nm{Z_{t,\aa}}_{L^2}\nm{Z_{t,\aa}}_{L^\infty}\\&\lec
\epsilon\paren{\nm{\partial_\aa\frac1{Z_{,\aa}}}_{L^2}+\nm{Z_{t,\aa}}_{\dot H^{1/2}}};
\end{aligned}
\end{equation}
therefore by \eqref{quasi}, 
\begin{equation}\label{153}
\nm{Z_{ttt,\aa}}_{L^2}\lec \nm{Z_{t,\aa\aa}}_{L^2}+\epsilon\nm{\partial_\aa\frac1{Z_{,\aa}}}_{\dot H^{1/2}}+\epsilon\nm{\partial_\aa\frac1{Z_{,\aa}}}_{L^2}+\epsilon\nm{Z_{t,\aa}}_{\dot H^{1/2}}.
\end{equation}
We begin with \eqref{ba}, and use \eqref{eq:c49} to expand. By \eqref{eq:b10}, \eqref{eq:b11}, \eqref{eq:b23}, \eqref{eq:b12}, \eqref{eq:b115}, \eqref{hhalf44}, \eqref{eq:b112}, and \eqref{eq:b111}, we have
\begin{equation}\label{151}
\begin{aligned}
&\nm{D_t^2(b_\aa-2\Re D_\aa Z_t)}_{L^2}\lec \nm{D_t^2\frac1{Z_{,\aa}}}_{L^2}\paren{\nm{Z_{t,\aa}}_{L^\infty}+\nm{Z_{t,\aa}}_{\dot H^{1/2}}}\\&+\nm{ D_t\frac1{Z_{,\aa}}}_{\dot H^{1/2}}\nm{Z_{t,\aa}}_{L^2}\nm{b_\aa}_{L^\infty}+
\nm{\partial_\aa \frac1{Z_{,\aa}}}_{L^2}\nm{Z_{t,\aa}}_{L^2}\nm{\partial_\aa D_tb}_{L^2}\\&+\nm{\partial_\aa \frac1{Z_{,\aa}}}_{L^2}\nm{Z_{t,\aa}}_{L^2}\nm{ b_\aa}_{L^4}^2+\nm{D_t\frac1{Z_{,\aa}}}_{\dot H^{1/2}}\nm{Z_{tt,\aa}}_{L^2}\\&+
\nm{\partial_\aa \frac1{Z_{,\aa}}}_{L^2}\nm{Z_{tt,\aa}}_{L^2}\nm{b_\aa}_{L^2}+\nm{\partial_\aa\frac1{Z_{,\aa}}}_{L^2}\nm{Z_{ttt}}_{\dot H^{1/2}}\\&\lec \epsilon\paren{\nm{\partial_\aa\frac 1{Z_{,\aa}}}_{\dot H^{1/2}}+\nm{Z_{t,\aa\aa}}_{L^2}+\nm{\partial_\aa\frac 1{Z_{,\aa}}}_{L^2}+\nm{Z_{t,\aa}}_{\dot H^{1/2}}};
\end{aligned}
\end{equation}
therefore by \eqref{eq:c7}, \eqref{eq:c1-1}, 
\begin{equation}\label{154}
\nm{\partial_\aa D_t^2 b}_{L^2}\lec \nm{Z_{t,\aa\aa}}_{L^2}+\epsilon\nm{\partial_\aa\frac1{Z_{,\aa}}}_{\dot H^{1/2}}+\epsilon\nm{\partial_\aa\frac1{Z_{,\aa}}}_{L^2}+\epsilon\nm{Z_{t,\aa}}_{\dot H^{1/2}}.
\end{equation} 
Now begin with \eqref{eq:c26}, and use \eqref{eq:c1-1}, we get
\begin{equation}\label{155}
\begin{aligned}
&\nm{D_t^3\frac1{Z_{,\aa}}}_{L^2}\lec \nm{b_\aa-2\Re D_\aa Z_t}_{L^\infty}\nm{D_t^2\frac1{Z_{,\aa}}}_{L^2}+\nm{D_t^2(b_\aa-2\Re D_\aa Z_t)}_{L^2}\\&+ \nm{D_t(b_\aa-2\Re D_\aa Z_t)}_{L^2}\nm{D_t\frac1{Z_{,\aa}}}_{L^\infty}+
\nm{D_\aa Z_t}_{L^\infty}\nm{D_t^2\frac1{Z_{,\aa}}}_{L^2}+\nm{D_t^2D_\aa Z_t}_{L^2}\\&+\nm{D_tD_\aa Z_t}_{L^2}\nm{D_t\frac1{Z_{,\aa}}}_{L^\infty}\lec \nm{Z_{t,\aa\aa}}_{L^2}+\epsilon\nm{\partial_\aa\frac1{Z_{,\aa}}}_{\dot H^{1/2}}+\epsilon\nm{\partial_\aa\frac1{Z_{,\aa}}}_{L^2}+\epsilon\nm{Z_{t,\aa}}_{\dot H^{1/2}}.
\end{aligned}
\end{equation}
We compute $D_t^2 A_1$ by \eqref{dta1}, \eqref{eq:c14}, \eqref{eq:c40}, estimating the $\dot H^{1/2}$ norms of 
$\bracket{Z_{tt}, b; \bar Z_{t,\aa}}$, $\bracket{Z_t, b; \bar Z_{tt,\aa}}$ and $D_t\bracket{Z_t, b; \bar Z_{t,\aa}}$ by the Sobolev inequality \eqref{eq:sobolev}, and estimating the  $\dot H^{1/2}$ norms of the remaining commutators in the expansion of $D_t^2 A_1$ by \eqref{hhalf42}. We get, by further using \eqref{eq:b12}, \eqref{eq:b22}, \eqref{eq:b112}, \eqref{eq:b111}, \eqref{eq:b115}, 
\begin{equation}\label{161}
\nm{D_t^2 A_1}_{\dot H^{1/2}}\lec \epsilon\paren{\nm{\partial_\aa\frac 1{Z_{,\aa}}}_{\dot H^{1/2}}+\nm{Z_{t,\aa\aa}}_{L^2}+\nm{\partial_\aa\frac 1{Z_{,\aa}}}_{L^2}+\nm{Z_{t,\aa}}_{\dot H^{1/2}}};
\end{equation}
this gives, by \eqref{2dinterface} and the estimates above in this section, 
\begin{equation}\label{162}
\nm{D_t^2 Z_{tt}}_{\dot H^{1/2}}\lec \nm{\partial_\aa\frac 1{Z_{,\aa}}}_{\dot H^{1/2}}+\epsilon\paren{\nm{Z_{t,\aa\aa}}_{L^2}+\nm{\partial_\aa\frac 1{Z_{,\aa}}}_{L^2}+\nm{Z_{t,\aa}}_{\dot H^{1/2}}}.
\end{equation}
We also have, by \eqref{eq:b13} and \eqref{eq:b115} that
\begin{equation}\label{161-1}
\nm{D_t^2 A_1}_{L^\infty}\lec \epsilon^2.
\end{equation}
We compute $D_t^3 A_1$ by \eqref{dta1}, \eqref{eq:c49}, \eqref{eq:c40}, and obtain, by using the inequalities in Appendix~\ref{ineq},
\begin{equation}\label{163}
\nm{D_t^3 A_1}_{L^2}\lec \epsilon\paren{\nm{\partial_\aa\frac 1{Z_{,\aa}}}_{\dot H^{1/2}}+\nm{Z_{t,\aa\aa}}_{L^2}+\nm{\partial_\aa\frac 1{Z_{,\aa}}}_{L^2}+\nm{Z_{t,\aa}}_{\dot H^{1/2}}};
\end{equation}
therefore by \eqref{quasi}, 
\begin{equation}\label{164}
\nm{D_t^3 Z_{tt}}_{L^2}\lec \nm{Z_{t,\aa\aa}}_{L^2}+\epsilon\paren{\nm{\partial_\aa\frac 1{Z_{,\aa}}}_{\dot H^{1/2}}+\nm{\partial_\aa\frac 1{Z_{,\aa}}}_{L^2}+\nm{Z_{t,\aa}}_{\dot H^{1/2}}}.
\end{equation}

We note that \eqref{151} also gives
\begin{equation}\label{165}
\nm{D_t^2(b_\aa-2\Re D_\aa Z_t)}_{L^2}\lec  \epsilon\paren{\nm{\partial_\aa\frac 1{Z_{,\aa}}}_{L^2}+\nm{Z_{t,\aa}}_{\dot H^{1/2}}};
\end{equation}
\eqref{165} is needed for the sake of estimating $\frac d{dt}\mathcal E_3(t)$ and proving \eqref{51}. We claim that a similar estimate also holds for $\nm{D_t^3 A_1}_{L^2}$, that is
\begin{equation}\label{166}
\nm{D_t^3 A_1}_{L^2}\lec  \epsilon\paren{\nm{\partial_\aa\frac 1{Z_{,\aa}}}_{L^2}+\nm{Z_{t,\aa}}_{\dot H^{1/2}}};
\end{equation}
 expanding $D_t^3 A_1$ via \eqref{dta1}, \eqref{eq:c49}, \eqref{eq:c40},  \eqref{166} can be checked straightforwardly using the inequalities in Appenddix~\ref{ineq}, the only term in the expansion that needs some extra explanation is  the term $\bracket{Z_t, D_t^2 b; \bar Z_{t,\aa}}$. Observe that by \eqref{eq:c1-1}, \eqref{eq:c7},
 \begin{equation}\label{167}
 \begin{aligned}
 \partial_\aa D_t^2 b&=D_t^2\paren{b_\aa-2\Re D_\aa Z_t}+3b_\aa \partial_\aa D_t b-2 (b_\aa)^3\\&-2\Re \paren{3D_\aa Z_t D_\aa Z_{tt}-2(D_\aa Z_t)^3}-2 \Re \paren{Z_{ttt}\partial_\aa\frac1{Z_{,\aa}}}+2\Re \partial_\aa\frac{Z_{ttt}}{Z_{,\aa}},
 \end{aligned}
 \end{equation}
so we have, by \eqref{165}, \eqref{120}, \eqref{116}, \eqref{99}, \eqref{77}, \eqref{hhalf44}, \eqref{76}, \eqref{111}, \eqref{171},
\begin{equation}\label{168}
\nm{\partial_\aa\paren{ D_t^2 b-2\Re \frac{Z_{ttt}}{Z_{,\aa}}}}_{L^2}\lec \epsilon\paren{\nm{\partial_\aa\frac 1{Z_{,\aa}}}_{L^2}+\nm{Z_{t,\aa}}_{\dot H^{1/2}}}.
\end{equation}
 We write
\begin{equation}\label{172}
\bracket{Z_t, D_t^2 b; \bar Z_{t,\aa}}=\bracket{Z_t, D_t^2 b-2\Re \frac{Z_{ttt}}{Z_{,\aa}}; \bar Z_{t,\aa}}+\bracket{Z_t, 2\Re \frac{Z_{ttt}}{Z_{,\aa}}; \bar Z_{t,\aa}}
\end{equation}
and apply \eqref{eq:b12} to the first term and \eqref{eq:b22} to the second term, we get, by further using \eqref{hhalf-2},
\begin{equation}\label{173}
\nm{\bracket{Z_t, D_t^2 b; \bar Z_{t,\aa}}}_{L^2}\lec \epsilon^2\paren{\nm{\partial_\aa\frac 1{Z_{,\aa}}}_{L^2}+\nm{Z_{t,\aa}}_{\dot H^{1/2}}}.
\end{equation}
This proves \eqref{166}. By \eqref{quasi}, \eqref{eq:c32}, \eqref{eq:c33}, \eqref{165}, \eqref{166} and the earlier estimates in this section, we conclude
\begin{equation}\label{174}
\nm{\mathcal P D_t^2 \bar Z_{t}}_{L^2}+\nm{\bracket{D_t, \mathcal P}\bar Z_{tt}}_{L^2}+\nm{\bracket{D_t^2, \mathcal P}\bar Z_{t}}_{L^2}\lec \epsilon \paren{\nm{\partial_\aa\frac 1{Z_{,\aa}}}_{L^2}+\nm{Z_{t,\aa}}_{\dot H^{1/2}}}.
\end{equation}
Finally we consider $\mathcal P D_t^3 \bar Z_t$. We write it as 
$$\mathcal P D_t^3 \bar Z_t=D_t^3\mathcal P \bar Z_t- \bracket{ D_t, \mathcal P} D_t^2\bar Z_t- D_t\bracket{ D_t, \mathcal P} D_t\bar Z_t-D_t^2 \bracket{ D_t, \mathcal P} \bar Z_t$$
and use \eqref{quasi}, \eqref{eq:c32} to compute. By the earlier estimates in this section, we have
\begin{equation}\label{213}
\nm{\bracket{ D_t, \mathcal P} \bar Z_{ttt}}_{L^2}+\nm{D_t\bracket{ D_t, \mathcal P} \bar Z_{tt}}_{L^2}+\nm{D_t^2 \bracket{ D_t, \mathcal P} \bar Z_t}_{L^2}\lec \epsilon^2;
\end{equation}
we compute  $D_t^3 (b_\aa-2\Re D_\aa Z_t)$ from \eqref{ba}, expanding by \eqref{eq:c49}, \eqref{eq:c40}. We have, by \eqref{eq:b10}, \eqref{eq:b11}, \eqref{eq:b23}, \eqref{eq:b12}, \eqref{eq:b22}, \eqref{eq:b121}, \eqref{eq:b111}, \eqref{eq:b115}, 
\begin{equation}\label{214}
\nm{D_t^3(b_\aa-2\Re D_\aa Z_t)}_{L^2}\lec \epsilon^2.
\end{equation}
we compute by  \eqref{eq:c7}, 
\begin{equation}\label{215} 
\partial_\aa D_t^3 b=D_t \partial_\aa D_t^2b+ b_\aa\partial_\aa D_t^2 b
 \end{equation}
and use \eqref{eq:c1-1}, \eqref{eq:c7}, and \eqref{167} to expand, by \eqref{hhalf44} and earlier estimates in this section, we have
\begin{equation}\label{216} 
\nm{\partial_\aa \paren{ D_t^3 b-2\Re\frac{D_t^3 Z_t}{Z_{,\aa}}}}_{L^2}\lec \epsilon^2;
 \end{equation}
 now we compute $D_t^4 A_1$ from \eqref{dta1}, expanding by \eqref{eq:c49}, \eqref{eq:c40}. Using a similar argument as that of \eqref{172} for the term $\bracket{Z_t, D_t^3 b; \bar Z_{t,\aa}}$ and use the inequalities in Appendix~\ref{ineq} to estimate the remaining terms in $D_t^4 A_1$ we get
 \begin{equation}\label{217}
 \nm{D_t^4 A_1}_{L^2}\lec \epsilon^2.
 \end{equation}
 This gives that
 \begin{equation}\label{218}
 \nm{\mathcal P D_t^3 \bar Z_t}_{L^2}\lec \epsilon^2.
 \end{equation}

\subsubsection{The estimates for $E_1(t)$, $\mathfrak E_1(t)$ and $E_j(t)$, $\mathcal E_j(t)$, $j\ge 2$}\label{step1-4} In this section we  prove the inequalities \eqref{46} and  \eqref{48}, \eqref{49}, \eqref{50} of Propositions~\ref{prop:4.3},~\ref{prop:4.4}.

{\sf Step 1}. We begin with $E_1(t)$. From \eqref{69}, \eqref{67}, 
$$E_1(t)=\nm{\Thc}_{\dot H^{1/2}}^2+\nm{1-\frac1{Z_{,\aa}}}_{L^2}^2+\int  i\,\partial_\aa\bar\Thb \P_H(G^{(1)})\,d\aa,
$$
and by Proposition~\ref{prop:G}, and \eqref{eq:b18}, 
\begin{equation}
\abs{\int  i\,\partial_\aa\bar\Thb \P_H(G^{(1)})\,d\aa}\lec \nm{1-\frac1{Z_{,\aa}}}_{L^2}^2\nm{Z_{t,\aa}}_{L^2}^2\lec \epsilon^2 \nm{1-\frac1{Z_{,\aa}}}_{L^2}^2,
\end{equation}
so there is $\epsilon_0=\epsilon_0(\delta)>0$, such that for all $0<\epsilon\le \epsilon_0$,  \eqref{48} holds. 
We note that by Lemma~\ref{lemma:4.1},
\begin{equation}\label{70}
\delta\nm{Z_t}_{\dot H^{1/2}}\le \nm{\Thc}_{\dot H^{1/2}}. 
\end{equation}
Now we consider $\mathfrak E_1(t)$. We estimate, by \eqref{eq:b18}, \eqref{eq:b19} and \eqref{71}, \eqref{78}, \eqref{70} that
\begin{equation*}
\begin{aligned}
|C_{1,1}(t)|+|C_{2,1}(t)|&\lec \nm{Z_{tt}}_{L^2}^2\nm{Z_{t,\aa}}_{L^2}^2+\nm{Z_{ttt}}_{L^2}\nm{Z_{t,\aa}}_{L^2}\nm{Z_t}_{\dot H^{1/2}}^2\\&\lec \frac{\epsilon^2}{\delta^2} \paren{\nm{\Thc}_{\dot H^{1/2}}^2+\nm{1-\frac1{Z_{,\aa}}}_{L^2}^2},
\end{aligned}
\end{equation*}
so there is $\epsilon_0=\epsilon_0(\delta)>0$, such that for all $0<\epsilon\le \epsilon_0$, 
\begin{equation}\label{80}
(1+\delta)^{-1/2} E_1(t)\le \mathfrak E_1(t)\le (1+\delta)^{1/2} E_1(t).
\end{equation}

{\sf Step 2}.  We next consider $E_2(t)$ and $\mathcal E_2(t)$. By   \eqref{69},
\begin{equation}\label{85}
E_2(t)=\nm{\Th^{(3)}}_{\dot H^{1/2}}^2+\nm{D_\aa\Thc}_{L^2}^2+\int  i\,\partial_\aa\bar\Thc \P_H(G^{(2)})\,d\aa.
\end{equation}
First we have by Lemma~\ref{lemma:4.1}, and \eqref{hhalfp}, \eqref{hhalf-1}
\begin{equation}\label{83}
 \delta\nm{1-\frac1{Z_{,\aa}}}_{\dot H^{1/2}}\le \nm{\P_H\paren{\frac1{\bar Z_{,\aa}}\paren{1-\frac1{Z_{,\aa}}}} }_{\dot H^{1/2}}\lec \nm{1-\frac1{Z_{,\aa}}}_{\dot H^{1/2}};
\end{equation}
and from the identity
$$\frac{\bar Z_{t,\aa}}{\bar Z_{,\aa}}=\partial_\aa\P_H\frac{\bar Z_t}{\bar Z_{,\aa}}-\bracket{\P_H,\bar Z_t}\partial_\aa\paren{\frac1 {\bar Z_{,\aa}}}+\bracket{\P_A, \frac1{\bar Z_{,\aa}}} {\bar Z_{t,\aa}},
$$
and \eqref{66} \eqref{eq:b10}, \eqref{eq:b11}, that
\begin{equation}\label{89}
  \nm{\frac{\bar Z_{t,\aa}}{\bar Z_{,\aa}}-i\partial_\aa\Thc}_{L^2}\lec\nm{\frac1{\bar Z_{,\aa}}}_{\dot H^{1/2}}\nm{\bar Z_{t,\aa}}_{L^2}\lec \epsilon \nm{\bar Z_{t,\aa}}_{L^2},\end{equation}
so by \eqref{75}, there is $\epsilon_0=\epsilon_0(\delta)>0$, such that for $0<\epsilon\le \epsilon_0$,
\begin{equation}\label{84}
\delta  \nm{\bar Z_{t,\aa}}_{L^2}\lec \nm{\partial_\aa\Thc}_{L^2}\lec \nm{\bar Z_{t,\aa}}_{L^2},\qquad \delta ^2 \nm{\bar Z_{t,\aa}}_{L^2}\lec \nm{D_\aa\Thc}_{L^2}\lec \nm{\bar Z_{t,\aa}}_{L^2}.
\end{equation}

Now by \eqref{69}, \eqref{66},
\begin{equation}\label{81}
\Theta^{(3)}=\P_H\paren{\frac1{\bar Z_{,\aa}}\paren{1-\frac1{Z_{,\aa}}}}+\P_H(G^{(1)});
\end{equation}
we estimate, by \eqref{eq:Gj}, \eqref{75}, \eqref{83} and \eqref{hhalf-1}, \eqref{eq:b20}, \eqref{eq:b18} to get
\begin{equation}\label{82}
\nm{\P_H(G^{(1)})}_{\dot H^{1/2}}\lec \nm{\frac1{Z_{,\aa}}}_{\dot H^{1/2}}\nm{Z_{t,\aa}}_{L^2}^2\lec \frac{\epsilon^2}{\delta} \nm{\P_H\paren{\frac1{\bar Z_{,\aa}}\paren{1-\frac1{Z_{,\aa}}}}}_{\dot H^{1/2}},
\end{equation}
and by \eqref{eq:Gj}-\eqref{17}, \eqref{eq:c46},  \eqref{3.21},  \eqref{eq:c37}, \eqref{eq:b18}, \eqref{eq:b19}, \eqref{eq:b115}, \eqref{75}   to yield
\begin{equation}\label{86}
\begin{aligned}
&\nm{\P_H(G^{(2)})}_{L^2}\lec \nm{b_\aa}_{L^2}\nm{Z_{t,\aa}}_{L^2}^2+\nm{Z_{t,\aa}}_{L^2}\nm{\frac1{Z_{,\aa}}}_{\dot H^{1/2}}\paren{\nm{Z_{tt}}_{\dot H^{1/2}}+\nm{\frac1{Z_{,\aa}}}_{\dot H^{1/2}}}\\&+\nm{Z_{t,\aa}}_{L^2}^2\nm{D_t\frac1{Z_{,\aa}}}_{L^2}
\lec \frac{\epsilon^2}{\delta^2} \nm{D_\aa\Thc}_{L^2},
\end{aligned}
\end{equation}
so there is $\epsilon_0=\epsilon_0(\delta)>0$, such that for $0<\epsilon\le \epsilon_0$,
\begin{equation}\label{87}
(1+\delta)^{-1}E_2(t)\le \nm{\P_H\paren{\frac1{\bar Z_{,\aa}}\paren{1-\frac1{Z_{,\aa}}}}  }_{\dot H^{1/2}}^2+\nm{D_\aa\Thc}_{L^2}^2\le (1+\delta)E_2(t).
\end{equation}
Now we  consider the  terms $C_{1,2}+C_{2,2}+F_2+H_2$ in $\mathcal E_2(t)$. Observe that
those terms with the factor $D_t^3 Z_t$ in $C_{1,2}$, $C_{2,2}$ can be combined with those in $F_2$ with the factor $\bar{\mathcal P D_t\bar Z_t}$, and we know
\begin{equation}\label{100}
-D_t^3 Z_t+\bar{\mathcal P D_t\bar Z_t}=-i\frac{A_1}{|Z_{,\aa}|^2}\partial_\aa Z_{tt}.
\end{equation}
So by \eqref{eq:b18}, \eqref{eq:b19}, \eqref{eq:b20}, \eqref{hhalf-1}, \eqref{eq:b115},  \eqref{eq:b48},  we have
\begin{equation}\label{88}
\begin{aligned}
&|C_{1,2}+C_{2,2}+F_2+H_2|\lec \nm{Z_{ttt}}_{L^2}^2\nm{Z_{t,\aa}}_{L^2}^2+\nm{Z_{ttt}}_{L^2}\nm{Z_{t,\aa}}_{L^2}\nm{Z_{tt}}^2_{\dot H^{1/2}}\\&+\nm{Z_{t,\aa}}_{L^2}^2\nm{Z_{tt}}_{\dot H^{1/2}}\paren{\nm{Z_{tt}}_{\dot H^{1/2}}+\nm{\frac{A_1}{|Z_{,\aa}|^2}}_{\dot H^{1/2}}}+\nm{b_\aa}_{L^2}\nm{Z_{ttt}}_{L^2}\nm{Z_{t,\aa}}_{L^2}^2\\&+ \nm{Z_{tt}}_{\dot H^{1/2}}^4+\nm{D_\aa\Thc}_{L^2}^2\nm{Z_{t,\aa}}_{L^2}^2\\&\lec\frac{\epsilon^2}{\delta^4}
\paren{\nm{D_\aa\Thc}_{L^2}^2+ \nm{\P_H\paren{\frac1{\bar Z_{,\aa}}\paren{1-\frac1{Z_{,\aa}}}}  }_{\dot H^{1/2}}^2};
\end{aligned}
\end{equation}
we also have by \eqref{eq:b47},
\begin{equation}\label{88-1}
|D_2(t)|\lec \nm{b_\aa}_{L^2}\nm{Z_{t,\aa}}_{L^2}\nm{Z_{tt}}_{\dot H^{1/2}}^2\lec \frac{\epsilon^2}{\delta^4}
\paren{\nm{D_\aa\Thc}_{L^2}^2+ \nm{\P_H\paren{\frac1{\bar Z_{,\aa}}\paren{1-\frac1{Z_{,\aa}}}}  }_{\dot H^{1/2}}^2}.
\end{equation}
By \eqref{83}, \eqref{84},  there is $\epsilon_0=\epsilon_0(\delta)>0$, such that for $0<\epsilon\le \epsilon_0$,
\begin{equation}\label{90}
 c_2(\delta)\paren{\nm{1-\frac1{Z_{,\aa}}}_{\dot H^{1/2}}^2+\nm{Z_{t,\aa}}_{L^2}^2}\le \mathcal E_2(t)\le c_1\paren{\nm{1-\frac1{Z_{,\aa}}}_{\dot H^{1/2}}^2+\nm{Z_{t,\aa}}_{L^2}^2}.
 \end{equation}
for some constants $c_1>0$ and $c_2(\delta)>0$.

{\sf Step 3}. We now study $E_3(t)$ and $\mathcal E_3(t)$. By \eqref{66}, \eqref{68},  \eqref{69}, we know
\begin{equation}\label{91}
E_3(t)=\nm{\Th^{(4)}}_{\dot H^{1/2}}^2+\nm{D_\aa\Th^{(3)}}_{L^2}^2+\int  i\,\partial_\aa\bar\Th^{(3)} \P_H(G^{(3)})\,d\aa,
\end{equation}
where
$$
\Th^{(4)}=-\P_H\paren{\frac1{|Z_{,\aa}|^2}\partial_\aa \P_H\frac{\bar Z_t}{\bar Z_{,\aa}}}+\P_H(G^{(2)}),
\quad D_\aa\Theta^{(3)}=-D_\aa\P_H\paren{\frac1{| Z_{,\aa}|^2}}+D_\aa\P_H(G^{(1)}).
$$
We compute
\begin{equation}\label{92}
\begin{aligned}
&D_\aa\P_H\paren{\frac1{| Z_{,\aa}|^2}}- \frac1{|Z_{,\aa}|^2}\partial_\aa \frac1{Z_{,\aa}} =     \\&
=\bracket{\frac1{Z_{,\aa}},\P_H}\partial_\aa\frac1{|Z_{,\aa}|^2}+\bracket{\P_H, \frac1{Z_{,\aa}^2}}\partial_\aa\frac1{\bar Z_{,\aa}}-\bracket{\P_A,\frac1{|Z_{,\aa}|^2}}\partial_\aa \frac1{Z_{,\aa}},
\end{aligned}
\end{equation}
so by \eqref{eq:b10}, \eqref{hhalf-1}, and \eqref{75},
\begin{equation}\label{93}
\nm{D_\aa\P_H\paren{\frac1{| Z_{,\aa}|^2}}- \frac1{|Z_{,\aa}|^2}\partial_\aa \frac1{Z_{,\aa}}}_{L^2}\lec \nm{\frac1{Z_{,\aa}}}_{\dot H^{1/2}}\nm{\partial_\aa\frac1{Z_{,\aa}} }_{L^2}\lec \frac{\epsilon}{\delta^2}\nm{\frac1{|Z_{,\aa}|^2}\partial_\aa \frac1{Z_{,\aa}}}_{L^2}.
\end{equation}
We estimate by \eqref{eq:G}, \eqref{eq:b21} and \eqref{eq:b18}, and get
\begin{equation}\label{126}
\nm{\partial_\aa \P_H(G^{(1)})}_{L^2}\lec \nm{\partial_\aa \frac1{Z_{,\aa}}}_{L^2}\nm{Z_{t,\aa}}_{L^2}^2\lec\frac{\epsilon^2}{\delta^2}\nm{\frac1{|Z_{,\aa}|^2}\partial_\aa \frac1{Z_{,\aa}}}_{L^2}.
\end{equation}
Now 
\begin{equation}\label{103}
 \P_H\frac{\bar Z_{t,\aa}}{\bar Z_{,\aa}}=\partial_\aa \P_H\frac{\bar Z_t}{\bar Z_{,\aa}}-\bracket{\P_H, \bar Z_t}\partial_\aa \frac1{\bar Z_{,\aa}},
\end{equation}
and by Lemma~\ref{lemma:4.1}, \eqref{hhalf42} and \eqref{Hhalf},
\begin{equation}\label{102}
\begin{aligned}
&\delta\nm{Z_{t,\aa}}_{\dot H^{1/2}}\le \nm{ \P_H\frac{\bar Z_{t,\aa}}{\bar Z_{,\aa}}}_{\dot H^{1/2}}\lec \nm{\partial_\aa \P_H\frac{\bar Z_t}{\bar Z_{,\aa}}}_{\dot H^{1/2}}+\nm{Z_{t,\aa}}_{L^2}\nm{\partial_\aa \frac1{Z_{,\aa}}}_{L^2}\\&
\lec\frac1{\delta}\paren{ \nm{D_\aa \P_H\frac{\bar Z_t}{\bar Z_{,\aa}}}_{\dot H^{1/2}}+ \nm{\partial_\aa \frac1{Z_{,\aa}}}_{L^2}\nm{\partial_\aa \P_H\frac{\bar Z_t}{\bar Z_{,\aa}}}_{L^2}}+\nm{Z_{t,\aa}}_{L^2}\nm{\partial_\aa \frac1{Z_{,\aa}}}_{L^2};
\end{aligned}
\end{equation}
using Lemma~\ref{lemma:4.1} again gives
\begin{equation}\label{101}
\delta\nm{D_\aa \P_H\frac{\bar Z_t}{\bar Z_{,\aa}}}_{\dot H^{1/2}}\le \nm{\P_H\paren{\frac1{|Z_{,\aa}|^2}\partial_\aa \P_H\frac{\bar Z_t}{\bar Z_{,\aa}}}}_{\dot H^{1/2}},
\end{equation}
therefore
\begin{equation}\label{104}
\delta^3\nm{Z_{t,\aa}}_{\dot H^{1/2}}\lec  \nm{\P_H\paren{\frac1{|Z_{,\aa}|^2}\partial_\aa \P_H\frac{\bar Z_t}{\bar Z_{,\aa}}}}_{\dot H^{1/2}}+\epsilon\delta \nm{\partial_\aa \frac1{Z_{,\aa}}}_{L^2};
\end{equation}
on the other hand, by \eqref{hhalf-2}, \eqref{78} ,\eqref{145} we  also have
\begin{equation}\label{105}
\nm{\P_H\paren{\frac1{|Z_{,\aa}|^2}\partial_\aa \P_H\frac{\bar Z_t}{\bar Z_{,\aa}}}}_{\dot H^{1/2}}\lec \epsilon \nm{\partial_\aa \frac1{Z_{,\aa}}}_{L^2}+\nm{Z_{t,\aa}}_{\dot H^{1/2}}.
\end{equation}
We next estimate $\nm{\P_H\paren{G^{(2)}}}_{\dot H^{1/2}}$. By \eqref{eq:Gj}-\eqref{17},  \eqref{hhalf-1}, \eqref{eq:b20}, and \eqref{eq:b18},
\begin{equation}\label{94}
\nm{\P_H\paren{G^{(2)}-D_t\P_H (G^{(1)})}}_{\dot H^{1/2}}\lec \nm{Z_{t,\aa}}_{L^2}\nm{\partial_\aa \frac1{Z_{,\aa}}}_{L^2}\nm{ \frac1{Z_{,\aa}}}_{\dot H^{1/2}}\lec \epsilon^2 \nm{\partial_\aa \frac1{Z_{,\aa}}}_{L^2},
\end{equation}
by \eqref{eq:c46}, \eqref{eq:c37}, \eqref{hhalf42}, \eqref{eq:b21},  \eqref{eq:b18},  \eqref{hhalf-2}, \eqref{eq:b19}, \eqref{eq:b115}, 
\begin{equation}\label{95}
\begin{aligned}
&\nm{\P_H D_t\P_H( G^{(1)})}_{\dot H^{1/2}}\lec \nm{b_\aa}_{L^2}\nm{Z_{t,\aa}}_{L^2}^2\nm{\partial_\aa \frac1{Z_{,\aa}}}_{L^2}+\nm{D_t\frac1{Z_{,\aa}}}_{\dot H^{1/2}}\nm{Z_{t,\aa}}_{L^2}^2\\&+\nm{D_t\frac1{Z_{,\aa}}}_{L^2}\nm{Z_{t,\aa}}_{L^2}^2\nm{\partial_\aa \frac1{Z_{,\aa}}}_{L^2}+\nm{Z_{t,\aa}}_{L^2}\nm{ Z_{tt}}_{\dot H^{1/2}}\nm{ \frac1{Z_{,\aa}}}_{\dot H^{1/2}}\nm{\partial_\aa \frac1{Z_{,\aa}}}_{L^2}\\&+
\nm{D_t< \bar Z_t, i\frac1{\bar Z_{,\aa}}, \bar Z_t>}_{\dot H^{1/2}}+\nm{D_t< Z_t, -i\frac1{ Z_{,\aa}}, \bar Z_t>}_{\dot H^{1/2}}.
\end{aligned}
\end{equation}
To estimate $\nm{D_t< \bar Z_t, i\frac1{\bar Z_{,\aa}}, \bar Z_t>}_{\dot H^{1/2}}$ and $\nm{D_t< Z_t, -i\frac1{ Z_{,\aa}}, \bar Z_t>}_{\dot H^{1/2}}$, we use \eqref{eq:c37} to expand. We use Sobolev inequality \eqref{eq:sobolev} to estimate the following $\dot H^{1/2}$ norm. By  \eqref{eq:b115},  \eqref{hhalf44}, we have
\begin{equation}\label{96}
\begin{aligned}
&\nm{\int\partial_\bb\mathfrak D_t\paren{\frac1{\aa-\bb}} \theta\theta \paren{\frac1{\bar Z_{,\aa}}-\frac1{\bar Z_{,\bb}}}\,d\bb}_{\dot H^{1/2}}^2\\&\lec \nm{b_\aa}_{L^2}\nm{Z_{t,\aa}}_{L^2}^2\nm{b_\aa}_{L^2}\nm{Z_{t,\aa}}^2_{L^4}\nm{\partial_\aa\frac1{Z_{,\aa}}}_{L^2}\lec \epsilon^5 \nm{Z_{t,\aa}}_{\dot H^{1/2}}\nm{\partial_\aa\frac1{Z_{,\aa}}}_{L^2}
\end{aligned}
\end{equation}
For the $\dot H^{1/2}$ norm of the remaining term in $D_t< \bar Z_t, i\frac1{\bar Z_{,\aa}}, \bar Z_t>$, we use \eqref{eq:b20}. We get
\begin{equation}\label{97}
\begin{aligned}
\nm{D_t< \bar Z_t, i\frac1{\bar Z_{,\aa}}, \bar Z_t>}_{\dot H^{1/2}}&\lec \epsilon^2 \nm{Z_{t,\aa}}_{\dot H^{1/2}}^{1/2}\nm{\partial_\aa\frac1{Z_{,\aa}}}_{L^2}^{1/2}\\&+\nm{Z_{t,\aa}}_{L^2}\nm{Z_{tt}}_{\dot H^{1/2}}\nm{\partial_\aa\frac1{Z_{,\aa}}}_{L^2}+\nm{Z_{t,\aa}}_{L^2}^2\nm{D_t\frac1{Z_{,\aa}}}_{\dot H^{1/2}}.
\end{aligned}
\end{equation}
Therefore
\begin{equation}\label{106}
\nm{\P_H\paren{G^{(2)}}}_{\dot H^{1/2}}\lec \epsilon^2\nm{\partial_\aa\frac1{Z_{,\aa}}}_{L^2}+\epsilon^2 \nm{Z_{t,\aa}}_{\dot H^{1/2}}.
\end{equation}

We now consider the term $\int  i\,\partial_\aa\bar\Th^{(3)} \P_H(G^{(3)})\,d\aa $ in \eqref{91}. 
The estimate for $\nm{\P_H(G^{(3)})}_{L^2}$ is as usual, by using the formula \eqref{eq:Gj}-\eqref{17}, the expansions \eqref{eq:c46} and \eqref{eq:c38}, \eqref{eq:c14},  and the inequalities \eqref{eq:b18}, \eqref{eq:b19}, \eqref{eq:b11}, \eqref{eq:b23}, \eqref{eq:b115} and \eqref{hhalf44}; observe that  by \eqref{eq:b11}, the estimates \eqref{94}-\eqref{97} for \eqref{106} can be used to treat the second term in the expansion by \eqref{eq:c46}.  We have
\begin{equation}\label{107}
\begin{aligned}
\nm{\P_H(G^{(3)})}_{L^2}&\lec \epsilon^2\nm{\partial_\aa\frac1{Z_{,\aa}}}_{L^2}+\epsilon^2 \nm{Z_{t,\aa}}_{\dot H^{1/2}}+\epsilon^2\nm{Z_{tt,\aa}}_{L^2}+\epsilon^2\nm{D_t\frac1{Z_{,\aa}}}_{\dot H^{1/2}}\\&+\epsilon^2\nm{D_t^2\frac1{Z_{,\aa}}}_{L^2}+\epsilon^2\nm{b_\aa}_{\dot H^{1/2}}+\epsilon^2\nm{\partial_\aa D_t b}_{L^2}+\epsilon^2\nm{Z_{ttt}}_{\dot H^{1/2}}\\&\lec \epsilon^2\nm{\partial_\aa\frac1{Z_{,\aa}}}_{L^2}+\epsilon^2 \nm{Z_{t,\aa}}_{\dot H^{1/2}};
\end{aligned}
\end{equation}
so there is a $\epsilon_0=\epsilon(\delta)>0$, such that for all $0<\epsilon\le\epsilon_0$, the inequality \eqref{49} in Proposition~\ref{prop:4.4} holds.
Finally we can estimate the correcting terms $C_{1,3}+C_{2,3}+F_3+H_3$ as in {\sf Step 2} by combining the terms in $C_{1,3}$ and $C_{2,3}$ with the factor $D_t^4Z_t$ with the terms in $F_3$ with the factor $\bar {\mathcal P D_t^2 \bar Z_t}$ and use \eqref{eq:b18},    \eqref{eq:b19}, \eqref{eq:b115}, \eqref{eq:b20}, \eqref{hhalf-2},  and \eqref{hhalf44} to obtain
\begin{equation}\label{108}
\begin{aligned}
&|C_{1,3}+C_{2,3}+F_3+H_3|\lec \nm{D_t^2 Z_t}_{L^2}\nm{Z_{t,\aa}}_{L^2}\nm{Z_{ttt}}_{\dot H^{1/2}}^2+\nm{D_t^2 Z_t}_{L^2}^2\nm{Z_{tt,\aa}}_{L^2}^2
\\&+ \nm{D_t^3 Z_t}_{L^2}\paren{\nm{Z_{t,\aa}}_{L^2}^2\nm{D_t^3 Z_t}_{L^2}+\nm{Z_{t,\aa}}_{L^2}\nm{Z_{tt,\aa}}_{L^2}\nm{D_t^2 Z_t}_{L^2}+\nm{Z_{tt,\aa}}_{L^2}\nm{Z_{tt}}_{\dot H^{1/2}}^2}\\&+
\nm{D_t^3 Z_t}_{L^2}\paren{\nm{b_\aa}_{L^4}\nm{D_t^2 Z_t}_{L^4}\nm{Z_{t,\aa}}_{L^2}^2+ \nm{b_\aa}_{L^2}\nm{Z_{t,\aa}}_{L^2}\nm{Z_{tt,\aa}}_{L^2}\nm{Z_{tt}}_{L^\infty}}\\&+\nm{Z_{t,\aa}}^2_{L^2}\nm{D_\aa\Th^{(3)}}_{L^2}^2+\nm{Z_{ttt}}_{\dot H^{1/2}}\paren{\nm{Z_{t,\aa}}_{L^2}^2\nm{Z_{ttt}}_{\dot H^{1/2}}+\nm{Z_{t,\aa}}_{L^2}\nm{Z_{tt,\aa}}_{L^2}\nm{Z_{tt}}_{\dot H^{1/2}}}\\&+\nm{Z_{ttt}}_{\dot H^{1/2}}\nm{\partial_\aa\frac{A_1}{|Z_{,\aa}|^2}}_{L^2}\paren{\nm{Z_{t,\aa}}_{L^2}^2\nm{Z_{ttt}}_{L^2}+\nm{Z_{t,\aa}}_{L^2}\nm{Z_{tt}}^2_{\dot H^{1/2}}}
\\&\lec \epsilon^2 \paren{  \nm{\partial_\aa\frac1{Z_{,\aa}}}_{L^2}^2+ \nm{Z_{t,\aa}}_{\dot H^{1/2}}^2+\nm{D_\aa\Th^{(3)}}_{L^2}^2}.
\end{aligned}
\end{equation}
So there is $\epsilon_0=\epsilon(\delta)>0$, such that for all $0<\epsilon\le\epsilon_0$, 
\begin{align}
\label{109}
(1+\delta)^{-1/2}  E_3(t)\le \mathcal E_3(t)\le (1+\delta)^{1/2} E_3(t).
\end{align}
This together with \eqref{80} proves the inequality \eqref{50} in Proposition~\ref{prop:4.4}. Moreover, there are constants $c_1>0$ and $c_2(\delta)>0$, such that for all $0<\epsilon\le\epsilon_0(\delta)$, 
\begin{equation}\label{110}
c_2(\delta)\paren{\nm{\partial_\aa\frac1{Z_{,\aa}}}_{L^2}^2+ \nm{Z_{t,\aa}}_{\dot H^{1/2}}^2} \le \mathcal E_3(t)\le c_1 \paren{\nm{\partial_\aa\frac1{Z_{,\aa}}}_{L^2}^2+ \nm{Z_{t,\aa}}_{\dot H^{1/2}}^2} .
\end{equation}

We also note that by  \eqref{105}, \eqref{106}, 
\begin{equation}\label{143}
\nm{\Th^{(4)}}_{\dot H^{1/2}}\lec  \nm{\partial_\aa\frac1{Z_{,\aa}}}_{L^2}^2+ \nm{Z_{t,\aa}}_{\dot H^{1/2}}\lec \epsilon.
\end{equation}
Because 
\begin{align}\label{291}
D_t\Th^{(3)}=\Th^{(4)}+\P_A D_t\Th^{(3)}=\Th^{(4)}+\bracket{\P_A, b}\partial_\aa\Th^{(3)},\\ \label{292}
D_t^2\Th^{(2)}
=\Th^{(4)}+\bracket{\P_A, b}\partial_\aa\Th^{(3)}+D_t\bracket{\P_A, b}\partial_\aa \Th^{(2)},
\end{align}
using the identity
\begin{equation}\label{284}
\bracket{b,b; \partial_\aa\Thc}=-<b,b,\partial_\aa\Thc>+\partial_\aa\Thc\bracket{b,b;1},
\end{equation}
 we have, by \eqref{hhalf42}, \eqref{eq:c14}, \eqref{eq:b20}, \eqref{hhalf-2}, \eqref{eq:b22}, \eqref{eq:b12}, 
\begin{equation}\label{293}
\nm{D_t\Th^{(3)}}_{\dot H^{1/2}}+\nm{D_t^2\Th^{(2)}}_{\dot H^{1/2}}\lec \nm{\partial_\aa\frac1{Z_{,\aa}}}_{L^2}^2+ \nm{Z_{t,\aa}}_{\dot H^{1/2}}\lec \epsilon.
\end{equation}
{\sf Step 4}. We now consider $E_4(t)$ and $\mathcal E_4(t)$. By \eqref{69},
\begin{equation}\label{124}
E_4(t)=\nm{\Th^{(5)}}_{\dot H^{1/2}}^2+\nm{D_\aa\Th^{(4)}}_{L^2}^2+\int  i\,\partial_\aa\bar\Th^{(4)} \P_H(G^{(4)})\,d\aa,
\end{equation}
and by \eqref{68}, \eqref{81}, 
\begin{equation}\label{125}
\begin{aligned}
\Theta^{(5)}&=-i\P_H\paren{\frac1{| Z_{,\aa}|^2}\partial_\aa\Theta^{(3)}}+\P_H(G^{(3)})\\&=
i\P_H\paren{  \frac1{| Z_{,\aa}|^2}\partial_\aa \P_H    \frac1{| Z_{,\aa}|^2}  }-i\P_H\paren{\frac1{| Z_{,\aa}|^2}\partial_\aa\P_H(G^{(1)})}+\P_H(G^{(3)}),
\end{aligned}
\end{equation}
\begin{equation}\label{127}
D_\aa\Th^{(4)}
=-D_\aa\P_H\paren{\frac1{|Z_{,\aa}|^2}\partial_\aa \P_H\frac{\bar Z_t}{\bar Z_{,\aa}}}+D_\aa\P_H(G^{(2)}).
\end{equation}
Since
\begin{equation}\label{128}
\begin{aligned}
&\P_H\paren{  \frac1{| Z_{,\aa}|^2}\partial_\aa \P_H    \frac1{| Z_{,\aa}|^2}  }- \frac1{| Z_{,\aa}|^2\bar Z_{,\aa}}\partial_\aa    \frac1{ Z_{,\aa}} \\& =-\bracket{\P_H, \frac1{| Z_{,\aa}|^2 }}\partial_\aa  \P_A  \frac1{ |Z_{,\aa}|^2}+\bracket{\P_H, \frac1{| Z_{,\aa}|^2 Z_{,\aa}}}\partial_\aa    \frac1{ \bar Z_{,\aa}} -\bracket{\P_A, \frac1{| Z_{,\aa}|^2 \bar Z_{,\aa}}}\partial_\aa    \frac1{  Z_{,\aa}}=:I,
\end{aligned}
\end{equation}
and by \eqref{hhalf42},
\begin{equation}\label{129}
\nm{I}_{\dot H^{1/2}}\lec \nm{\partial_\aa\frac1{Z_{,\aa}}}_{L^2}^2\lec \epsilon \nm{\partial_\aa\frac1{Z_{,\aa}}}_{L^2};
\end{equation}
 by \eqref{hhalf-2}, \eqref{Hhalf},
\begin{align}\label{130}
&\nm{\frac1{| Z_{,\aa}|^2\bar Z_{,\aa}}\partial_\aa    \frac1{ Z_{,\aa}}}_{\dot H^{1/2}}\lec \nm{\partial_\aa    \frac1{ Z_{,\aa}}}_{\dot H^{1/2}}+\nm{\partial_\aa\frac1{Z_{,\aa}}}_{L^2}^2,\\&\label{131}
\delta^3\nm{\partial_\aa    \frac1{ Z_{,\aa}}}_{\dot H^{1/2}}\lec \nm{\frac1{| Z_{,\aa}|^2\bar Z_{,\aa}}\partial_\aa    \frac1{ Z_{,\aa}}}_{\dot H^{1/2}}+\nm{\partial_\aa\frac1{Z_{,\aa}}}_{L^2}^2;
\end{align}
therefore
\begin{equation}\label{136}
\begin{aligned}
\delta^3\nm{\partial_\aa    \frac1{ Z_{,\aa}}}_{\dot H^{1/2}}-\epsilon \nm{\partial_\aa\frac1{Z_{,\aa}}}_{L^2}\lec  &\nm{\P_H\paren{  \frac1{| Z_{,\aa}|^2}\partial_\aa \P_H    \frac1{| Z_{,\aa}|^2}  }}_{\dot H^{1/2}}\\&
\lec \nm{\partial_\aa    \frac1{ Z_{,\aa}}}_{\dot H^{1/2}}+\epsilon \nm{\partial_\aa\frac1{Z_{,\aa}}}_{L^2}.
\end{aligned}
\end{equation}
Also, we know $\P_H b=\P_H\frac{\bar Z_t}{\bar Z_{,\aa}}=i\Th^{(2)}$, and 
\begin{equation}\label{132}
\begin{aligned}
&\partial_\aa\P_H\paren{\frac1{|Z_{,\aa}|^2}\partial_\aa \P_H\frac{\bar Z_t}{\bar Z_{,\aa}}}-\frac1{|Z_{,\aa}|^2}\partial_\aa ^2\P_H\frac{\bar Z_t}{\bar Z_{,\aa}}\\&=\P_H\paren{\partial_\aa\frac1{|Z_{,\aa}|^2}\partial_\aa \P_H\frac{\bar Z_t}{\bar Z_{,\aa}}}-\bracket{\P_A,\frac1{|Z_{,\aa}|^2}}\partial_\aa^2 \P_H\frac{\bar Z_t}{\bar Z_{,\aa}},
\end{aligned}
\end{equation}
applying \eqref{3.20} on the second term, we get
\begin{equation}\label{133}
\nm{\partial_\aa\P_H\paren{\frac1{|Z_{,\aa}|^2}\partial_\aa \P_H\frac{\bar Z_t}{\bar Z_{,\aa}}}-\frac1{|Z_{,\aa}|^2}\partial_\aa ^2\P_H\frac{\bar Z_t}{\bar Z_{,\aa}}}_{L^2}\lec \nm{b_\aa}_{L^\infty}\nm{\partial_\aa\frac1{Z_{,\aa}}}_{L^2}\lec\epsilon \nm{\partial_\aa\frac1{Z_{,\aa}}}_{L^2};
\end{equation}
now since
\begin{equation}\label{134}
\partial_\aa ^2\P_H\frac{\bar Z_t}{\bar Z_{,\aa}}-\frac{\bar Z_{t,\aa\aa}}{\bar Z_{,\aa}}=-\bracket{\P_A, \frac{1}{\bar Z_{,\aa}}}\partial_\aa \bar Z_{t,\aa}+2\bracket{\P_H, \bar Z_{t,\aa}}\partial_\aa \frac{1}{\bar Z_{,\aa}}+\bracket{\P_H, \bar Z_{t}}\partial_\aa^2 \frac{1}{\bar Z_{,\aa}},
\end{equation}
applying \eqref{eq:b10}, \eqref{eq:b11}, \eqref{eq:b23} yields
\begin{equation}\label{135}
\nm{\partial_\aa ^2\P_H\frac{\bar Z_t}{\bar Z_{,\aa}}-\frac{\bar Z_{t,\aa\aa}}{\bar Z_{,\aa}}}_{L^2}\lec \nm{\partial_\aa\frac1{Z_{,\aa}}}_{L^2}\nm{Z_{t,\aa}}_{\dot H^{1/2}}\lec \epsilon \nm{\partial_\aa\frac1{Z_{,\aa}}}_{L^2};
\end{equation}
therefore
\begin{equation}
\begin{aligned}\label{137}
\delta^4\nm{Z_{t,\aa\aa}}_{L^2}- \epsilon \nm{\partial_\aa\frac1{Z_{,\aa}}}_{L^2}\lec &\nm{D_\aa\P_H\paren{\frac1{|Z_{,\aa}|^2}\partial_\aa \P_H\frac{\bar Z_t}{\bar Z_{,\aa}}}}_{L^2}\\&\lec \nm{Z_{t,\aa\aa}}_{L^2}+\epsilon \nm{\partial_\aa\frac1{Z_{,\aa}}}_{L^2}.
\end{aligned}
\end{equation}
We next consider $\nm{\P_H\paren{\frac1{| Z_{,\aa}|^2}\partial_\aa\P_H(G^{(1)})}}_{\dot H^{1/2}}$. By \eqref{eq:G}, 
and applying \eqref{hhalf-2}, \eqref{126}, \eqref{eq:Gj},  \eqref{eq:b115}, \eqref{eq:b21}, \eqref{eq:b20} gives
\begin{equation}\label{138}
\begin{aligned}
&\nm{\P_H\paren{\frac1{| Z_{,\aa}|^2}\partial_\aa\P_H(G^{(1)})}}_{\dot H^{1/2}}\lec \nm{\partial_\aa\P_H(G^{(1)})}_{\dot H^{1/2}}+\nm{\partial_\aa\frac1{Z_{,\aa}}}_{L^2}^2\nm{Z_{t,\aa}}_{L^2}^2\\&\lec
\nm{\partial_\aa\frac1{Z_{,\aa}}}_{\dot H^{1/2}}\nm{Z_{t,\aa}}_{L^2}^2+\nm{\partial_\aa\frac1{Z_{,\aa}}}_{L^2}^2\nm{Z_{t,\aa}}_{L^2}^2+\nm{Z_{t,\aa}}_{\dot H^{1/2}}\nm{Z_{t,\aa}}_{L^2}\nm{\partial_\aa\frac1{Z_{,\aa}}}_{L^2}\\&\lec \epsilon^2 \nm{\partial_\aa\frac1{Z_{,\aa}}}_{\dot H^{1/2}}+\epsilon^2\nm{\partial_\aa\frac1{Z_{,\aa}}}_{L^2} .
\end{aligned}
\end{equation}
We now estimate $\nm{\partial_\aa\P_H \paren{G^{(2)}}}_{L^2}$. By \eqref{eq:Gj}-\eqref{17}, using \eqref{eq:c46}, \eqref{eq:c37} to expand, then applying  \eqref{eq:b18}, \eqref{eq:b21},  \eqref{eq:b10}, \eqref{eq:b23},  \eqref{eq:b115},  
we get
\begin{equation}\label{139}
\begin{aligned}
&\nm{\partial_\aa\P_H \paren{G^{(2)}}}_{L^2}\lec \nm{\partial_\aa\frac1{Z_{,\aa}}}_{L^2}^2\nm{Z_{t,\aa}}_{L^2}+\nm{b_\aa}_{\dot H^{1/2}} \nm{\partial_\aa\frac1{Z_{,\aa}}}_{L^2}\nm{Z_{t,\aa}}_{L^2}^2\\&+\nm{\partial_\aa D_t\frac1{Z_{,\aa}}}_{L^2}\nm{Z_{t,\aa}}_{L^2}^2+ \nm{ D_t\frac1{Z_{,\aa}}}_{L^\infty}\nm{\partial_\aa \frac1{Z_{,\aa}}}_{L^2}\nm{Z_{t,\aa}}_{L^2}^2\\&+\nm{\partial_\aa \frac1{Z_{,\aa}}}_{L^2}\nm{Z_{t,\aa}}_{L^2}\nm{Z_{tt,\aa}}_{L^2}+\nm{b_\aa}_{L^\infty} \nm{\partial_\aa\frac1{Z_{,\aa}}}_{L^2}\nm{Z_{t,\aa}}_{L^2}^2\\&\lec
\epsilon^2 \nm{Z_{t,\aa\aa}}_{L^2}+\epsilon^2\nm{\partial_\aa\frac1{Z_{,\aa}}}_{\dot H^{1/2}}+\epsilon^2\nm{\partial_\aa\frac1{Z_{,\aa}}}_{L^2}.
\end{aligned}
\end{equation}
And we estimate $\nm{\P_H(G^{(3)})}_{\dot H^{1/2}}$ by \eqref{eq:Gj}-\eqref{17}, expanding with \eqref{eq:c46}, \eqref{eq:c38} and \eqref{eq:c14}.
Again the estimate is straightforward  but tedious. We use \eqref{eq:b20} to estimate  $\nm{\P_H\paren{G^{(3)}-D_t\P_H G^{(2)}}}_{\dot H^{1/2}}$ and the $\dot H^{1/2}$ norm of the term 
$$\frac1{\pi }\int\frac{\mathfrak D_t^2\paren{(\bar Z_t(\aa)-\bar Z_t(\bb))^2\paren{\frac1{\bar Z_{,\aa}}-\frac1{\bar Z_{,\bb}}}}}{(\aa-\bb)^2}\,d\bb
$$
in the expansion of $D_t^2<\bar Z_t, i\frac1{\bar Z_{,\aa}}, \bar Z_t>$; 
we use  Sobolev inequality \eqref{eq:sobolev} to estimate the $\dot H^{1/2}$ norms of the remaining terms in the
expansion of $D_t^2<\bar Z_t, i\frac1{\bar Z_{,\aa}}, \bar Z_t>$, 
as well as 
$\nm{D_t<\bar Z_t, i\frac1{\bar Z_{,\aa}}, i\frac1{ Z_{,\aa}}  >}_{\dot H^{1/2}}$  
and $\nm{\bracket{b,\,b,\,\partial_\aa\braces{\frac1{\bar Z_{,\aa}}<\bar Z_t, i\frac1{\bar Z_{,\aa}}, \bar Z_t>}}
}_{\dot H^{1/2}}$. Observe that we have  \eqref{97}, and the estimate in \eqref{139} is useful for the estimate of the second term in the expansion of
$\P_H D_t\P_H G^{(2)}$ by \eqref{eq:c46}. We have
\begin{equation}\label{144}
\begin{aligned}
&\nm{\P_H(G^{(3)})}_{\dot H^{1/2}}\lec 
\epsilon^2 \nm{Z_{t,\aa\aa}}_{L^2}+\epsilon^2\nm{\partial_\aa\frac1{Z_{,\aa}}}_{\dot H^{1/2}}+\epsilon^2\nm{\partial_\aa D_t\frac1{Z_{,\aa}}}_{L^2}+\epsilon^2\nm{D_t^2\frac1{Z_{,\aa}}}_{\dot H^{1/2}}\\&+\epsilon^2\nm{\partial_\aa\frac1{Z_{,\aa}}}_{L^2}+\epsilon^2\nm{Z_{t,\aa}}_{\dot H^{1/2}}\\&\lec \epsilon^2 \nm{Z_{t,\aa\aa}}_{L^2}+\epsilon^2\nm{\partial_\aa\frac1{Z_{,\aa}}}_{\dot H^{1/2}}+\epsilon^2\nm{\partial_\aa\frac1{Z_{,\aa}}}_{L^2}+\epsilon^2\nm{Z_{t,\aa}}_{\dot H^{1/2}}.
\end{aligned}
\end{equation}
We estimate $\nm{\P_H(G^{(4)})}_{L^2}$ by \eqref{eq:Gj}-\eqref{17}, expanding with \eqref{eq:c46}-\eqref{eq:c35}, \eqref{eq:c38}, \eqref{eq:c49}. The estimate is routine, using the inequalities in Appendix~\ref{ineq} and the estimates in \S\ref{quan}.  Going through the terms carefully, we get 
\begin{equation}\label{159}
\nm{\P_H(G^{(4)})}_{L^2}\lec 
 \epsilon^2 \nm{Z_{t,\aa\aa}}_{L^2}+\epsilon^2\nm{\partial_\aa\frac1{Z_{,\aa}}}_{\dot H^{1/2}}+\epsilon^2\nm{\partial_\aa\frac1{Z_{,\aa}}}_{L^2}+\epsilon^2\nm{Z_{t,\aa}}_{\dot H^{1/2}}.
\end{equation}
Finally the correcting terms $C_{1,4}+C_{2,4}+F_4+H_4$ can be estimated similarly as in {\sf Step 2} by combining the terms in $C_{i, 4}$ with the factor $D_t^5 Z_t$ with those terms in $F_4$ with the factor $\bar {\mathcal P D_t^3 Z_t}$
and use the equation
\begin{equation}\label{160}
-D_t^5 Z_t+\bar {\mathcal P D_t^3 \bar Z_t}=-i\frac{A_1}{|Z_{,\aa}|^2}\partial_\aa D_t^3 Z_t
\end{equation}
and we have, by applying the inequalities in Appendix~\ref{ineq}, 
\begin{equation}\label{175}
\begin{aligned}
&\abs{C_{1,4}+C_{2,4}+F_4+H_4} \\&\lec\epsilon^2 \paren{\nm{Z_{t,\aa\aa}}_{L^2}+\nm{\partial_\aa\frac1{Z_{,\aa}}}_{\dot H^{1/2}}+\nm{\partial_\aa\frac1{Z_{,\aa}}}_{L^2}+\nm{Z_{t,\aa}}_{\dot H^{1/2}} +\nm{D_\aa\Th^{(4)}}_{L^2}  }^2.
\end{aligned}
\end{equation}
This together with \eqref{110}  shows that there is a $\epsilon_0=\epsilon_0(\delta)>0$, such that for all of $0<\epsilon\le \epsilon_0$,
\begin{equation}\label{176}
\mathcal E_3(t)+\mathcal E_4(t)\lec \nm{Z_{t,\aa\aa}}_{L^2}^2+\nm{\partial_\aa\frac1{Z_{,\aa}}}_{\dot H^{1/2}}^2+\nm{\partial_\aa\frac1{Z_{,\aa}}}_{L^2}^2+\nm{Z_{t,\aa}}_{\dot H^{1/2}}^2 \le c_2(\delta)^{-1}(\mathcal E_3(t)+ \mathcal E_4(t)),
\end{equation}
for some constant $c_2(\delta)>0$. 
Combining with \eqref{76}, \eqref{90} proves the inequality \eqref{46} in Proposition~\ref{prop:4.3}.
We also note that by \eqref{127}, \eqref{133}, \eqref{135}, \eqref{139}, \eqref{125}, \eqref{136}, \eqref{138}, \eqref{144}, \eqref{81},
\begin{equation}\label{179}
\nm{\partial_\aa \Th^{(4)}}_{L^2}+\nm{\Th^{(5)}}_{\dot H^{1/2}}+\nm{\partial_\aa\Th^{(3)}}_{\dot H^{1/2}} \lec \epsilon.
\end{equation}
Because
\begin{equation}\label{294}
\begin{aligned}
\partial_\aa D_t\Th^{(3)}&=\partial_\aa\Th^{(4)}+\partial_\aa\bracket{\P_A, b}\partial_\aa\Th^{(3)},\\
D_t^2\Th^{(3)}
&=\Th^{(5)}+\bracket{\P_A, b}\partial_\aa\Th^{(4)}+D_t\bracket{\P_A, b}\partial_\aa \Th^{(3)},\\
D_tD_\aa\Th^{(3)}&=D_\aa D_t\Th^{(3)}-D_\aa Z_t D_\aa \Th^{(3)},
\end{aligned}
\end{equation}
using the identity
\begin{equation}\label{285}
\bracket{b,b; \partial_\aa\Th^{(3)}}=-<b,b,\partial_\aa\Th^{(3)}>+\partial_\aa\Th^{(3)}\bracket{b,b;1},
\end{equation}
we have by \eqref{eq:b10}, \eqref{eq:b23}, \eqref{hhalf42}, \eqref{eq:c14}, \eqref{eq:b20}, \eqref{hhalf-2}, \eqref{eq:b22}, \eqref{eq:b12}, 
\begin{equation}\label{295}
\nm{\partial_\aa D_t\Th^{(3)}}_{L^2}+\nm{D_t^2\Th^{(3)}}_{\dot H^{1/2}}+\nm{ D_t D_\aa \Th^{(3)}}_{L^2} \lec \epsilon.
\end{equation}

\subsubsection{The estimates for $\frac d{dt}\mathfrak E_1(t)$, $\frac d{dt}\mathcal E_j(t)$, $j\ge 2$} In this section we prove the inequalities \eqref{47} and \eqref{51}. This requires us to estimate $\frac d{dt}\mathfrak E_1(t)$, $\frac d{dt}\mathcal E_j(t)$, $2\le j\le 4$. We will use Theorem~\ref{th:main1} for $\frac d{dt}\mathfrak E_1(t)$, and \eqref{34}-\eqref{33}-\eqref{remainder3} for $\frac d{dt}\mathcal E_j(t)$, $2\le j\le 4$. We refer the reader to Appendix~\ref{quantities} for a  list of the quantities controlled by $\epsilon$.

{\sf Step 5.} In this step we  want to use \eqref{eq:main1}-\eqref{remainder2} to show 
that for  $0<\epsilon\le \epsilon_0(\delta)$, where $\epsilon_0(\delta)$ is as in {\sf Steps 1-4}, 
\begin{equation}\label{180}
\frac d{dt}\mathfrak E_1(t)\lec \frac{\epsilon^3}{\delta^3} \mathfrak E_1(t),\qquad \text{for } t\in [0, T_0].
\end{equation}
The prove for \eqref{180} is quite straightforward. Observe that $D_\aa\Thb=i\paren{1-\frac1{Z_{,\aa}}}$ by \eqref{67}, and by \eqref{2dinterface}, 
\begin{equation}\label{182}
\bar Z_{tt}=i\paren{1-\frac1{Z_{,\aa}}}-i\,\frac{A_1-1}{Z_{,\aa}} =D_\aa\Thb-i\,\frac{A_1-1}{Z_{,\aa}};
\end{equation}
recall the notation \eqref{62}, \eqref{18}, \eqref{22}, we write
\begin{equation}\label{181}
\begin{aligned}
&\Re \int i\partial_\aa\bar{\Thb}\P_H\paren{G^{(2)}-D_t\P_H G^{(1)}}\,d\aa-I_{2,1}\\&=\frac1{2\pi}\Re \iint (\bar{D_\aa\Thb}(\aa,t)-Z_{tt}(\aa,t))\frac{\mathfrak D_t(\theta\,\bar\theta)\mathfrak D_t\theta}{(\aa-\bb)^2}\,d\bb\,d\aa\\&+
\frac1{2\pi}\Re \iint \bar{D_\aa\Thb}(\aa,t)\frac{\mathfrak D_t(\theta\,\bar\theta)(\lambda^1-\mathfrak D_t\theta)}{(\aa-\bb)^2}\,d\bb\,d\aa\\&+
\frac12\Re\int i \,  \bar{D_\aa\Thb}(\aa,t)\paren{<\bar Z_t, i\,\frac{1-A_1}{\bar Z_{,\aa}}, -i\frac1{Z_{,\aa}}>+  <i\,\frac{A_1-1}{ Z_{,\aa}}, Z_t,  -i\frac1{Z_{,\aa}}>  }  \,d\aa,
\end{aligned}
\end{equation}
so by \eqref{eq:b18}, and the estimates in \S\ref{quan}, 
\begin{equation}\label{184}
\begin{aligned}
&\abs{\Re \int i\partial_\aa\bar{\Thb}\P_H\paren{G^{(2)}-D_t\P_H G^{(1)}}\,d\aa-I_{2,1}}\lec \nm{A_1-1}_{L^2}\nm{Z_{t,\aa}}_{L^2}\nm{Z_{tt,\aa}}_{L^2}\nm{Z_{tt}}_{L^2}\\&+\nm{A_1-1}_{L^2}\nm{Z_{t,\aa}}_{L^2}\nm{1-\frac1{Z_{,\aa}}}_{L^2}\paren{\nm{Z_{tt,\aa}}_{L^2}+\nm{\partial_\aa\frac1{Z_{,\aa}}}_{L^2}}\\&\lec \epsilon^3\paren{\nm{Z_t}_{\dot H^{1/2}}+\nm{1-\frac1{Z_{,\aa}}}_{L^2}}^2;
\end{aligned}
\end{equation}
observe that $D_\aa\Tha=\bar Z_t$, we write
\begin{equation}\label{186}
\begin{aligned}
&2\Re \int i\, \partial_\aa\bar\Thb\P_H D_t\P_H\paren{G^{(1)}-D_t\P_H G^{(0)}}\,d\aa-I_{1,1}\\&=
\frac1{\pi}\Re\int Z_{tt} \paren{D_t\int\frac{\mathfrak D_t(\theta\,\bar\theta)\,\theta}{(\aa-\bb)^2}\,d\bb-\int\frac{\mathfrak D_t \paren{\mathfrak D_t(\theta\,\bar\theta)\,\theta}}{(\aa-\bb)^2}\,d\bb}\,d\aa\\&+
\frac1{\pi}\Re\int \partial_\aa\bar\Thb\paren{\P_H D_t\P_H\frac1{\bar Z_{,\aa}}-\frac1{\bar Z_{,\aa}}D_t}\int\frac{\mathfrak D_t(\theta\,\bar\theta)\,\theta}{(\aa-\bb)^2}\,d\bb\,d\aa\\&+
\frac1{\pi}\Re\int \paren{\bar{D_\aa\Thb}-Z_{tt}} D_t\int\frac{\mathfrak D_t(\theta\,\bar\theta)\,\theta}{(\aa-\bb)^2}\,d\bb\,d\aa
\\&+\Re \int i\, \partial_\aa\bar\Thb\P_H D_t \P_H\braces{\frac1{\bar Z_{,\aa}} \paren{<\bar Z_t, i\,\frac{1-A_1}{\bar Z_{,\aa}}, \bar Z_t >+  <i\,\frac{A_1-1}{ Z_{,\aa}}, Z_t,  \bar Z_t>  }  }\,d\aa;
\end{aligned}
\end{equation}
now 
\begin{equation}\label{187}
\P_H D_t\P_H\paren{\frac1{\bar Z_{,\aa}}f}-\P_H\paren{\frac1{\bar Z_{,\aa}}D_tf}
=\P_H \braces{D_t\paren{\frac1{\bar Z_{,\aa}}}f}-\P_H\bracket{b,\P_A}\partial_\aa\paren{\frac1{\bar Z_{,\aa}}f},
\end{equation}
so by \eqref{eq:c37}, \eqref{eq:c46}, \eqref{eq:b115}, \eqref{eq:b18}, \eqref{eq:b19}, \eqref{hhalf-1},
\begin{equation}\label{188}
\abs{2\Re \int i\, \partial_\aa\bar\Thb\P_H D_t\P_H\paren{G^{(1)}-D_t\P_H G^{(0)}}\,d\aa-I_{1,1}} \lec \frac{\epsilon^3}\delta\paren{\nm{Z_t}_{\dot H^{1/2}}^2+\nm{1-\frac1{Z_{,\aa}}}_{L^2}^2}.
\end{equation}
Now the estimates for all the terms in $R_{IC,1}$ are straightforward, using \eqref{eq:b115}, \eqref{eq:b18}, \eqref{eq:b19}, \eqref{eq:b46}, \eqref{eq:b47}, \eqref{eq:b24}, \eqref{eq:b28}, \eqref{eq:b37}  and the estimates in \S\ref{quan}, we have
\begin{equation}\label{189}
\abs{R_{IC,1}}\lec \epsilon^3\paren{\nm{Z_t}_{\dot H^{1/2}}+\nm{1-\frac1{Z_{,\aa}}}_{L^2}}^2.
\end{equation}
This together with \eqref{184}, \eqref{188}, \eqref{70}, \eqref{80}, \eqref{48} proves \eqref{180}. 

In the remainder of this paper we will show that for $0<\epsilon\le \epsilon_0(\delta)$, where $\epsilon_0(\delta)$ is as in Steps 2-4, 

\begin{equation}\label{191}
\frac d{dt} \mathcal E_2(t)\lec \epsilon^5, \quad \frac d{dt} \mathcal E_3(t)\lec \epsilon^3\paren{\nm{Z_{t,\aa}}_{\dot H^{1/2}}^2+\nm{\partial_\aa\frac1{Z_{,\aa}}}_{L^2}^2},\quad \frac d{dt} \mathcal E_4(t)\lec \epsilon^5.
\end{equation}
this together with \eqref{180}, \eqref{110} proves the inequalities \eqref{47}, \eqref{51} and finishes the proof for  Propositions~\ref{prop:4.3}, ~\ref{prop:4.4} as well as Theorem~\ref{thm:main2}. 

{\sf Step 6}. We begin with those terms in the remainder $\mathcal R_{IC, j}+\Re\frac d{dt} D_j(t)$, $j\ge 2$, (cf. \eqref{33}-\eqref{remainder3})  for which the desired estimates can be directly derived from the inequalities in Appendix~\ref{ineq}, keep in mind that   lower order norms such as  $\nm{1-\frac1{Z_{,\aa}}}_{L^2}$ or $\nm{Z_{t}}_{\dot H^{1/2}}$ are NOT allowed, as they are not controlled by $L(t)$; and since $\nm{\frac1{Z_{,\aa}}-1}_{L^\infty}\le 1$,  $\nm{Z_{tt}}_{L^\infty}\lec 1$, we can  involve 
$\nm{\frac1{Z_{,\aa}}-1}_{L^\infty}$ and $\nm{Z_{tt}}_{L^\infty}$ only in sextic or higher order terms, at most once.

We first consider the following type of terms
\begin{equation}\label{192}
M_{1,j}:=\iint\paren{b_\aa+b_\bb-2\frac{b(\aa)-b(\bb)}{\aa-\bb}}\frac{D_t^{j}Z_t\, \mathfrak D_t^l\theta\,\mathfrak D_t^i\bar\theta\,\mathfrak D_t^{k}\theta}{(\aa-\bb)^2}\,d\aa d\bb
\end{equation}
where $l+i+k=j$. This type of terms appears in \eqref{31}, \eqref{remainder1} and \eqref{25}, we use \eqref{eq:b46}, \eqref{eq:b47} and the estimates in \S\ref{quan} to handle $M_{1,j}$.

For $j=2$, $(l,i,k)=(0, 0,2), (0,1,1)$ and their permutations, so 
\begin{equation}\label{193}
\abs{M_{1,2}}\lec \nm{D_t^2 Z_t}_{L^2}\nm{b_\aa}_{L^\infty}\paren{\nm{Z_{t,\aa}}_{L^2}^2\nm{D_t^2 Z_t}_{L^2}+\nm{Z_{t,\aa}}_{L^2}\nm{Z_{tt}}_{\dot H^{1/2}}^2 }\lec \epsilon^5.
\end{equation}

For $j=3$, $(l,i,k)=(0,0, 3), (0, 1,2), (1,1,1)$ and their permutations, so 
\begin{equation}\label{194}
\begin{aligned}
\abs{M_{1,3}}&\lec \nm{D_t^3 Z_t}_{L^2}\nm{b_\aa}_{L^\infty}\paren{\nm{Z_{t,\aa}}_{L^2}^2\nm{D_t^3 Z_t}_{L^2}+\nm{Z_{t,\aa}}_{L^2} \nm{Z_{tt,\aa}}_{L^2} \nm{D_t^2 Z_t}_{L^2}}\\&+\nm{D_t^3 Z_t}_{L^2}\nm{b_\aa}_{L^\infty}\nm{Z_{tt,\aa}}_{L^2}\nm{Z_{tt}}_{\dot H^{1/2}}^2 \lec 
\epsilon^3\paren{\nm{Z_{t,\aa}}_{\dot H^{1/2}}^2+\nm{\partial_\aa\frac1{Z_{,\aa}}}_{L^2}^2}.
\end{aligned}
\end{equation}

For $j=4$, $(l,i,k)=(0,0, 4), (0, 1,3), (0,2,2), (1, 1,2)$ and their permutations, so a similar argument also gives
\begin{equation}\label{195}
\abs{M_{1,4}}\lec  \epsilon^5.
\end{equation}

We next consider terms of the type 
\begin{equation}\label{196}
M_{2,j}:=\iint \frac{ \bar{\mathcal P D_t^{m}\bar Z_t} \mathfrak D_t^l\theta\,\mathfrak D_t^i \bar\theta\,\mathfrak D_t^{k}\theta}{(\aa-\bb)^2}\,d\aa d\bb
\end{equation}
where $m+l+i+k=2j-1$. We discuss the following cases: 

1. $m=j-1$, $l+i+k=j$, as appeared in \eqref{31};

2. $i=j$, $m+l+k=j-1$, as appeared in \eqref{25}- \eqref{remainder1};

3. $i=0$, $m+l+k=2j-1$, $ 0\le m, l, k\le j-1$, as appeared in \eqref{25}-\eqref{remainder1};

4. $m=0$, $l+i+k=2j-1$, $ 0\le l,i, k\le j-1$, as appeared in \eqref{25}-\eqref{remainder1}.

Observe that we can use the symmetry \eqref{26} to rewrite the corresponding terms in \eqref{25} as \eqref{196}, and vice versa. 
We use \eqref{eq:b18}, \eqref{eq:b19} and the estimates in \S\ref{quan} to obtain the inequalities \eqref{197}, \eqref{198}, and \eqref{199} below. 

For $j=2$, we have the following cases: $m=1$, $(l,i,k)=(0, 0, 2), (0,1,1)$ and permutations; $i=2$, $(m, l, k)=(0,0, 1)$ and permutations; $i=0$, $(m, l, k)=(1,1, 1)$;  $m=0$, $( l,i, k)=(1,1, 1)$. It is clear that by \eqref{eq:b18}, \eqref{eq:b19} and the estimates in \S\ref{quan} 
we have 
\begin{equation}\label{197}
\abs{M_{2,2}}\lec \epsilon^5.
\end{equation}

For $j=3$, we have the following cases: $m=2$,  $(l,i,k)=(0, 0, 3), (0,1,2), (1,1,1)$ and permutations; $i=3$, $(m, l, k)=(0,0, 2),(0,1,1)$ and permutations; $i=0$, $(m, l, k)=(1,2, 2)$ and permutations; $m=0$, $(l, i, k)=(1,2, 2)$ and permutations. Observe that for the case where $i=3$, we need to use the symmetry \eqref{26} to rewrite $M_{2,3}$ as in \eqref{25}. 
We have  by \eqref{eq:b18}, \eqref{eq:b19} and the estimates in \S\ref{quan}, 
\begin{equation}\label{198}
\abs{M_{2,3}} \lec 
\epsilon^3\paren{\nm{Z_{t,\aa}}_{\dot H^{1/2}}^2+\nm{\partial_\aa\frac1{Z_{,\aa}}}_{L^2}^2}.
\end{equation}

For $j=4$, we have the following cases: $m=3$, $(l,i,k)=(0, 0, 4), (0,1,3), (0,2,2), (1, 1, 2)$ and permutations; $i=4$, $(m, l, k)=(0,0, 3),(0,1,2), (1,1,1)$ and permutations; $i=0$, $(m, l, k)=(3,3, 1), (3,2,2) $ and permutations; $m=0$, $( l, i,k)=(3,3, 1), (3,2,2) $ and permutations. Again we use the form in \eqref{25} for the case where $i=4$. We have
quite straightforwardly that
\begin{equation}\label{199}
\abs{M_{2,4}}\lec \epsilon^5.
\end{equation}

Observe that those terms  
$$\iint\frac{ D_t^jZ_t\,  \mathfrak D_t^{m}\braces{(\mathcal P\mathfrak D_t^l\theta)\,\mathfrak D_t^i\bar\theta\,\mathfrak D_t^k\theta - \mathfrak D_t^l\theta\, \overline{(\mathcal P\mathfrak D_t^i\theta)}\,\mathfrak D_t^k\theta +\mathfrak D_t^l\theta\,\mathfrak D_t^i\bar\theta\,(\mathcal P\mathfrak D_t^k\theta )}}{(\aa-\bb)^2}\,d\aa\,d\bb$$
in \eqref{25}, with $m+l+i+k=j-1$ can be treated exactly in the same way as $M_{2,j}$,  we do not specifically go over these terms.

Now we consider the terms $\iint \frac{\bar{\bracket{D_t,\mathcal P}D_t^{j-1}\bar Z_t }\mathfrak D_t^{m}(\mathfrak D_t^l\theta\,\mathfrak D_t^i \bar\theta\,\mathfrak D_t^{k}\theta)}{(\aa-\bb)^2}\,d\aa d\bb$ in \eqref{31}. We use product rules and complex conjugate to reduce it to the following form 
\begin{equation}\label{200}
M_{3,j}:=\iint \frac{\bracket{D_t,\mathcal P}D_t^{j-1}\bar Z_t\, \mathfrak D_t^l\bar\theta\,\mathfrak D_t^i \theta\,\mathfrak D_t^{k}\bar\theta}{(\aa-\bb)^2}\,d\aa d\bb,
\end{equation}
 where $l+i+k=j-1$. We know by \eqref{eq:c32},
 \begin{equation}\label{201}
 \bracket{D_t,\mathcal P}D_t^{j-1}\bar Z_t=
 i\,\paren{\frac{D_t A_1}{A_1}+b_\aa-2\Re D_\aa Z_t}\frac{A_1}{|Z_{,\aa}|^2}\partial_\aa D_t^{j-1}\bar Z_t,
 \end{equation}
 so $M_{3,j}$ is sextic; and we have, by \eqref{hhalf-1} and the estimates in \S\ref{quan},
\begin{equation}\label{202}
\begin{aligned}
&\abs{M_{3,j}}\lec \nm{D_t^{j-1} Z_t}_{\dot H^{1/2}} \nm{  \paren{\frac{D_t A_1}{A_1}+b_\aa-2\Re D_\aa Z_t}\frac{A_1}{|Z_{,\aa}|^2}<D_t^l Z_t, D_t^i\bar Z_t, D_t^k Z_t>}_{\dot H^{1/2}}\\&\lec
\nm{D_t^{j-1} Z_t}_{\dot H^{1/2}} \paren{\nm{ D_t A_1}_{L^\infty}+\nm{b_\aa-2\Re D_\aa Z_t}_{L^\infty}}\nm{<D_t^l Z_t, D_t^i\bar Z_t, D_t^k Z_t>}_{\dot H^{1/2}}\\&+
\nm{D_t^{j-1} Z_t}_{\dot H^{1/2}} \paren{\nm{ \frac{D_t A_1}{|Z_{,\aa}|^2}}_{\dot H^{1/2}}+\nm{\frac{(b_\aa-2\Re D_\aa Z_t)A_1}{|Z_{,\aa}|^2}}_{\dot H^{1/2}}}\nm{ <D_t^l Z_t, D_t^i\bar Z_t, D_t^k Z_t>}_{L^\infty};
\end{aligned}
\end{equation}
we can also use \eqref{hhalf-2} to get
\begin{equation}\label{203}
\begin{aligned}
&\abs{M_{3,j}} \lec
\nm{D_t^{j-1} Z_t}_{\dot H^{1/2}} \paren{\nm{ D_t A_1}_{L^\infty}+\nm{b_\aa-2\Re D_\aa Z_t}_{L^\infty}}\nm{<D_t^l Z_t, D_t^i\bar Z_t, D_t^k Z_t>}_{\dot H^{1/2}}\\&+
\nm{D_t^{j-1} Z_t}_{\dot H^{1/2}} \paren{\nm{ \frac{D_t A_1}{|Z_{,\aa}|^2}}_{\dot H^{1}}+\nm{\frac{(b_\aa-2\Re D_\aa Z_t)A_1}{|Z_{,\aa}|^2}}_{\dot H^{1}}}\nm{ <D_t^l Z_t, D_t^i\bar Z_t, D_t^k Z_t>}_{L^2}.
\end{aligned}
\end{equation}
For $j=2$, we have $(l ,i, k)=(0,0,1)$ and permutations; for $j=3$, we have $(l,i,k)=(0, 0, 2), (0,1,1)$ and permutations; for $j=4$, we have $(l,i,k)=(0,0,3), (0, 1,2), (1,1,1)$ and permutations.
We use \eqref{202} for $j=2, 3$ and \eqref{203} for $j=4$. By  \eqref{hhalf-1}, \eqref{eq:b20}, \eqref{eq:b18}, \eqref{eq:b19} and the estimates in \S\ref{quan} we get
\begin{equation}\label{204}
\abs{M_{3,2}}\lec \epsilon^5, \quad \abs{M_{3,3}} \lec \epsilon^3\paren{\nm{Z_{t,\aa}}_{\dot H^{1/2}}^2+\nm{\partial_\aa\frac1{Z_{,\aa}}}_{L^2}^2},\quad \abs{M_{3,4}}\lec \epsilon^6.
\end{equation}

Next we look at the terms of type 
\begin{equation}\label{205}
M_{4,j}:= \iint \paren{b_\aa+b_\bb-2\frac{b(\aa)-b(\bb)}{\aa-\bb}}\frac{ \bar{\mathcal P D_t^{j-1}\bar Z_t}  \mathfrak D_t^l\theta\,\mathfrak D_t^i \bar\theta\,\mathfrak D_t^{k}\theta}{(\aa-\bb)^2}\,d\aa d\bb
 \end{equation}
 and 
 \begin{equation}\label{206}
 M_{5,j}:=\iint
\paren{6\frac{(b(\aa)-b(\bb))^2}{(\aa-\bb)^2}-4(b_\aa+b_\bb)\frac {(b(\aa)-b(\bb))}{(\aa-\bb)} +2b_\aa b_\bb}\frac{D_t^{j}Z_t\,\mathfrak D_t^l\theta\,\mathfrak D_t^i\bar\theta\,\mathfrak D_t^{k}\theta}{(\aa-\bb)^2}\,d\aa d\bb
\end{equation}
in \eqref{31}, where $l+i+k=j-1$.
Observe that both terms are sextic, and we used product rules to reduce the terms in \eqref{31} to the forms of \eqref{205}, \eqref{206}. We take advantage of the fact that these terms are sextic and use \eqref{eq:b115}, \eqref{hhalf44} to deduce for $m=4,5$, 
\begin{equation}\label{207}
\abs{M_{m,2}}\lec \epsilon^5, \quad \abs{M_{m,3}} \lec \epsilon^3\paren{\nm{Z_{t,\aa}}_{\dot H^{1/2}}^2+\nm{\partial_\aa\frac1{Z_{,\aa}}}_{L^2}^2},\quad \abs{M_{m,4}}\lec \epsilon^5.
\end{equation}

Now we consider the term 
\begin{equation}\label{208}
M_{6,j}:=\iint\frac{ D_t^jZ_t\, \bracket{\mathcal P, \mathfrak D_t^{m}}\paren{\mathfrak D_t^l\theta\,\mathfrak D_t^i\bar\theta\,\mathfrak D_t^k\theta }}{(\aa-\bb)^2}\,d\aa\,d\bb=\iint\frac{ \bracket{\mathfrak P, D_t^{m}}\paren{\mathfrak D_t^l\theta\,\mathfrak D_t^i\bar\theta\,\mathfrak D_t^k\theta }\,  \mathfrak D_t^j\bar\theta}{(\aa-\bb)^2}\,d\aa\,d\bb
\end{equation}
in \eqref{25}, where $m+l+i+k=j-1$. Here we used the symmetry \eqref{26} to get the second equality above. We expand $\bracket{\mathfrak P, D_t^{m}}$ by \eqref{eq:c32}, \eqref{eq:c33}, and use \eqref{eq:b18} and the estimates in \S\ref{quan} to obtain 
\begin{equation}\label{209}
\abs{M_{6,2}}\lec \epsilon^6, \quad \abs{M_{6,3}} \lec \epsilon^4\paren{\nm{Z_{t,\aa}}_{\dot H^{1/2}}^2+\nm{\partial_\aa\frac1{Z_{,\aa}}}_{L^2}^2},\quad \abs{M_{6,4}}\lec \epsilon^6.
\end{equation}

Now we treat the terms
\begin{equation}\label{210}
M_{7,j}:=i\iint \paren{\partial_\bb\frac{A_1(\bb)}{|Z_{,\bb}|^2}-2\frac{\frac{A_1(\aa)}{|Z_{,\aa}|^2}-\frac{A_1(\bb)}{|Z_{,\bb}|^2}}{\aa-\bb}}\frac{D_t^j Z_t(\aa)\,\mathfrak D_t^l\theta\,\mathfrak D_t^i\bar\theta\,\mathfrak D_t^k\theta }{(\aa-\bb)^2}\,d\aa\,d\bb
\end{equation}
and 
\begin{equation}\label{211}
M_{8,j}:=\iint\paren{\partial_\bb D_tb(\bb)-2\frac{D_tb(\aa)-D_t b(\bb)}{\aa-\bb}}\frac{D_t^{j}Z_t(\aa)\,\mathfrak D_t^l\theta\,\mathfrak D_t^i\bar\theta\,\mathfrak D_t^{k}\theta}{(\aa-\bb)^2}\,d\aa d\bb
\end{equation}
where $l+i+k=j-1$ in \eqref{25}-\eqref{remainder1} and \eqref{31}. Here we used the product rules to convert the terms in \eqref{25} and \eqref{31} to the forms in \eqref{210}, \eqref{211}. We use \eqref{eq:b28} and \eqref{eq:b37} and the estimates in \S\ref{quan} to obtain for $m=7, 8$, 
\begin{equation}\label{212}
\abs{M_{m,2}}\lec \epsilon^5, \quad \abs{M_{m,3}} \lec \epsilon^3\paren{\nm{Z_{t,\aa}}_{\dot H^{1/2}}^2+\nm{\partial_\aa\frac1{Z_{,\aa}}}_{L^2}^2},\quad \abs{M_{m,4}}\lec \epsilon^5.
\end{equation}

We also have the following terms from $R_{\bar 0; l, \bar i, k}^{(0)}$, where $l+i+k=2j-1$, $0\le l,i,k\le j-1$,  in \eqref{25}-\eqref{remainder1},
\begin{equation}\label{219-1}
\iint\paren{b_\aa+b_\bb-2\frac{b(\aa)-b(\bb)}{\aa-\bb}}\frac{\paren{Z_t\mathfrak D_t- Z_{tt}}\paren{\mathfrak D_t^l\theta\,\mathfrak D_t^i\bar\theta\,\mathfrak D_t^k\theta }}{(\aa-\bb)^2}\,d\aa\,d\bb;
\end{equation}
we  use the symmetry \eqref{26} to rewrite it as 
\begin{equation}\label{219}
M_{9,j}:=\frac12\iint\paren{b_\aa+b_\bb-2\frac{b(\aa)-b(\bb)}{\aa-\bb}}\frac{\paren{\bar\theta\mathfrak D_t- \mathfrak D_t\bar\theta}\paren{\mathfrak D_t^l\theta\,\mathfrak D_t^i\bar\theta\,\mathfrak D_t^k\theta }}{(\aa-\bb)^2}\,d\aa\,d\bb;
\end{equation}
and use product rules to expand.  By \eqref{eq:b46}, \eqref{eq:b47} and the estimates in \S\ref{quan}, we have
\begin{equation}\label{220}
\begin{aligned}
\abs{M_{9,3}}&\lec \nm{b_\aa}_{L^\infty}\nm{Z_{t,\aa}}_{L^2}\paren{\nm{D_t^3Z_{t}}_{L^2}\nm{Z_{tt,\aa}}_{L^2}\nm{Z_{ttt}}_{L^2}+\nm{Z_{ttt}}_{\dot H^{1/2}}^2\nm{Z_{ttt}}_{L^2}}\\&+ \nm{b_\aa}_{L^\infty}\nm{Z_{tt,\aa}}_{L^2}^2\nm{Z_{ttt}}_{L^2}^2 \lec \epsilon^3\paren{\nm{Z_{t,\aa}}_{\dot H^{1/2}}^2+\nm{\partial_\aa\frac1{Z_{,\aa}}}_{L^2}^2},
\end{aligned}
\end{equation}
and 
 \begin{equation}\label{221}
\begin{aligned}
\abs{M_{9,4}}&\lec \epsilon^5.
\end{aligned}
\end{equation}
For $j=2$,  $(l, i, k)=(1,1,1)$,  we have, by \eqref{eq:b47}, 
\begin{equation}\label{222}
\begin{aligned}
&\abs{\iint\paren{b_\aa+b_\bb-2\frac{b(\aa)-b(\bb)}{\aa-\bb}}\frac{\bar\theta\mathfrak D_t \paren{\mathfrak D_t\theta\,\mathfrak D_t\bar\theta\,\mathfrak D_t\theta }}{(\aa-\bb)^2}\,d\aa\,d\bb}\\&\lec \nm{Z_{ttt}}_{L^2}\nm{Z_{t,\aa}}_{L^2}\nm{Z_{tt}}_{\dot H^{1/2}}^2\nm{b_\aa}_{L^\infty}\lec \epsilon^5;
\end{aligned}
\end{equation}
and by \eqref{eq:b48},
\begin{equation}\label{223}
\abs{\iint\paren{b_\aa+b_\bb-2\frac{b(\aa)-b(\bb)}{\aa-\bb}}\frac{\mathfrak D_t\bar\theta \mathfrak D_t\theta\,\mathfrak D_t\bar\theta\,\mathfrak D_t\theta }{(\aa-\bb)^2}\,d\aa\,d\bb}\lec \nm{b_\aa}_{L^\infty} \nm{Z_{tt}}_{\dot H^{1/2}}^4\lec \epsilon^5;
\end{equation}
so
\begin{equation}\label{224}
\abs{M_{9,2}}\lec \epsilon^5.
\end{equation}
We can estimate the second and third term on the right hand side of \eqref{226} similarly and obtain
\begin{equation}\label{229}
\begin{aligned}
&\abs{\iint  \mathbb H b_\aa 
\frac{\mathfrak D_t\paren{\bar\theta\,\mathfrak D_t\theta\mathfrak D_t\bar\theta\mathfrak D_t\theta}}{(\aa-\bb)^2}\,d\aa\,d\bb}\lec \nm{\mathbb H b_\aa}_{L^\infty} \nm{Z_{tt}}_{\dot H^{1/2}}^4\\ &+\nm{\mathbb H b_\aa}_{L^2} \nm{Z_{ttt}}_{L^\infty}\nm{Z_{tt}}_{\dot H^{1/2}}^2\nm{Z_{t,\aa}}_{L^2}\lec \epsilon^5;
\end{aligned}
\end{equation}
\begin{equation}\label{230} 
\begin{aligned} 
&\abs{\iint \paren{b_\bb-2\frac{b(\aa)-b(\bb)}{\aa-\bb}} \mathbb H b_\aa
\frac{\bar\theta\,\mathfrak D_t\theta\mathfrak D_t\bar\theta\mathfrak D_t\theta}{(\aa-\bb)^2}\,d\aa\,d\bb}\\&
\lec \nm{\mathbb H b_\aa}_{L^2} \nm{b_\aa}_{L^\infty}\nm{Z_{tt}}_{\dot H^{1/2}}^2\nm{Z_{t,\aa}}_{L^2}\nm{Z_{tt}}_{L^\infty}\lec \epsilon^5.
\end{aligned}
\end{equation}

Observe that the first two (non-zero) terms in $R_{\bar 0; l, \bar i, k}^{(0)}$, $l+i+k=2j-1$, $0\le l,i,k\le j-1$, cf. \eqref{25}-\eqref{remainder1}, have been covered in $M_{2,j}$, there is one more term left, which is 
\begin{equation}\label{227}
i\iint \paren{\partial_\aa\frac{A_1(\aa)}{|Z_{,\aa}|^2}+\partial_\bb\frac{A_1(\bb)}{|Z_{,\bb}|^2}-2\frac{\frac{A_1(\aa)}{|Z_{,\aa}|^2}-\frac{A_1(\bb)}{|Z_{,\bb}|^2}}{\aa-\bb}}\frac{ Z_t\, \mathfrak D_t^l\theta\,\mathfrak D_t^i\bar\theta\,\mathfrak D_t^k\theta }{(\aa-\bb)^2}\,d\aa\,d\bb.
\end{equation}
We use symmetry \eqref{26} to rewrite it as
\begin{equation}\label{228}
M_{10, j}:=\frac12 i\iint \paren{\partial_\aa\frac{A_1(\aa)}{|Z_{,\aa}|^2}+\partial_\bb\frac{A_1(\bb)}{|Z_{,\bb}|^2}-2\frac{\frac{A_1(\aa)}{|Z_{,\aa}|^2}-\frac{A_1(\bb)}{|Z_{,\bb}|^2}}{\aa-\bb}}\frac{ \bar\theta\, \mathfrak D_t^l\theta\,\mathfrak D_t^i\bar\theta\,\mathfrak D_t^k\theta }{(\aa-\bb)^2}\,d\aa\,d\bb.
\end{equation}
By \eqref{eq:b115}, \eqref{hhalf44} and the estimates in \S\ref{quan}, we have
\begin{equation}\label{231}
\abs{M_{10,3}}\lec \epsilon^3\paren{\nm{Z_{t,\aa}}_{\dot H^{1/2}}^2+\nm{\partial_\aa\frac1{Z_{,\aa}}}_{L^2}^2},\qquad \abs{M_{10,4}}\lec \epsilon^5.
\end{equation}
For $j=2$ and $(l,i,k)=(1,1,1)$, we have by \eqref{eq:b48} and H\"older's inequality,
\begin{equation}\label{232}
\abs{\iint \paren{\frac{\frac{A_1(\aa)}{|Z_{,\aa}|^2}-\frac{A_1(\bb)}{|Z_{,\bb}|^2}}{\aa-\bb}}\frac{ \bar\theta\, \mathfrak D_t\theta\,\mathfrak D_t\bar\theta\,\mathfrak D_t\theta }{(\aa-\bb)^2}\,d\aa\,d\bb}\lec \nm{Z_{tt}}_{\dot H^{1/2}}^3\nm{\frac{A_1}{|Z_{,\aa}|^2}}_{\dot H^{1/2}}\nm{Z_{t,\aa}}_{L^\infty}\lec \epsilon^5.
\end{equation}

{\sf Step 7}.
We have treated all the terms in $\mathcal R_{IC, j}+\Re\frac d{dt} D_j(t)$ but the following two: the first is 
\begin{equation}\label{233}
M_{11,j}:=\iint \partial_\aa\paren{i\,\frac{A_1(\aa)}{|Z_{,\aa}|^2}+ D_t b(\aa) } \frac{D_t^j Z_t(\aa)\,\mathfrak D_t^l\theta\,\mathfrak D_t^i\bar\theta\,\mathfrak D_t^k\theta }{(\aa-\bb)^2}\,d\aa\,d\bb
\end{equation}
where $l+i+k=j-1$, $2\le j\le 4$; the second is for $j=2$ only, from $\Re\paren{\frac 1{4\pi} R_{\bar 0; 1,\bar 1,1}^{(0)}+ \frac d{dt} D_2(t)}$:
\begin{equation}\label{234}
M_{12,2}:=\iint \partial_\aa\paren{i\,\frac{A_1(\aa)}{|Z_{,\aa}|^2}+ D_t\mathbb H  b(\aa) } \frac{\bar \theta\,\mathfrak D_t\theta\,\mathfrak D_t\bar\theta\,\mathfrak D_t\theta }{(\aa-\bb)^2}\,d\aa\,d\bb,
\end{equation}
here we used the symmetry \eqref{26} to rewrite the term from \eqref{25}, and used product rules to arrive at the term in \eqref{233}. 

By \eqref{eq:c45},
\begin{equation}\label{235}
\P_H \paren{D_t b+i\frac{A_1}{|Z_{,\aa}|^2}-i}=i\P_H\paren{\frac{A_1-1}{|Z_{,\aa}|^2}}+i\P_H(G^{(1)})+i\bracket{b,\P_H}\partial_\aa \bar {\Theta^{(2)}};
\end{equation}
using \eqref{eq:c21}, \eqref{eq:c44} gives
\begin{equation}\label{236}
\begin{aligned}
D_t \mathbb H b+i\frac{A_1}{|Z_{,\aa}|^2}-i&=\bracket {D_t,\mathbb H} b+ 2i\Im \P_H \paren{D_t b+i\frac{A_1}{|Z_{,\aa}|^2}-i}\\&=
i\, \frac{A_1-1}{|Z_{,\aa}|^2}+2i\Re \P_H(G^{(1)})+\frac12 \bracket{ b,\mathbb H }\partial_\aa b;
\end{aligned}
\end{equation}
 and we know for $j\ge 2$
\begin{equation}
\mathbb P_A D_t^j \bar Z_t =-i\bracket{\mathbb P_A, \frac{A_1}{|Z_{,\aa}|^2}}\partial_\aa D_t^{j-2}\bar Z_t+ \mathbb P_A \mathcal P D_t^{j-2}\bar Z_t;
\end{equation}
so by \eqref{eq:b10}, \eqref{eq:b11}, \eqref{eq:b23}, \eqref{126} and the estimates in \S\ref{quan}, we have for $2\le j\le 4$,
\begin{equation}\label{237}
\nm{ \partial_\aa\paren{D_t \mathbb H b+i\frac{A_1}{|Z_{,\aa}|^2}}    }_{L^2}+\nm{ \partial_\aa\P_H \paren{D_t b+i\frac{A_1}{|Z_{,\aa}|^2}}}_{L^2}+ \nm{\mathbb P_A D_t^j \bar Z_t}_{L^2}\lec \epsilon^2,
\end{equation}
in particular for $j=3$, 
\begin{equation}\label{242}
\nm{ \partial_\aa\P_H \paren{D_t b+i\frac{A_1}{|Z_{,\aa}|^2}}}_{L^2}+\nm{\mathbb P_A D_t^3 \bar Z_t}_{L^2}\lec \epsilon\paren{\nm{Z_{t,\aa}}_{\dot H^{1/2}}+\nm{\partial_\aa\frac1{Z_{,\aa}}}_{L^2}};
\end{equation}
by \eqref{hhalf42},  \eqref{hhalf-2}, \eqref{138}, and the estimates in \S\ref{quan}, the following also holds,
\begin{equation}\label{240}
\nm{ \partial_\aa\P_H \paren{D_t b+i\frac{A_1}{|Z_{,\aa}|^2}}}_{\dot H^{1/2}}\lec \epsilon^2.
\end{equation}
Observe that $M_{12,2}$ is sextic.  Applying \eqref{eq:b47} gives
\begin{equation}\label{239}
\abs{M_{12,2}}\lec \nm{ \partial_\aa\paren{D_t \mathbb H b+i\frac{A_1}{|Z_{,\aa}|^2}}    }_{L^2}\nm{Z_{t,\aa}}_{L^2}\nm{Z_{tt}}_{\dot H^{1/2}}^2\nm{Z_{tt}}_{L^\infty}\lec \epsilon^5.
\end{equation}

We rewrite $M_{11,j}$ as
\begin{equation}\label{238}
\begin{aligned}
M_{11,j}&=\iint \partial_\aa\P_H \paren{i\,\frac{A_1(\aa)}{|Z_{,\aa}|^2}+ D_t b(\aa) } \frac{D_t^j Z_t(\aa)\,\mathfrak D_t^l\theta\,\mathfrak D_t^i\bar\theta\,\mathfrak D_t^k\theta }{(\aa-\bb)^2}\,d\aa\,d\bb\\&+
\iint \partial_\aa\P_A \paren{i\,\frac{A_1(\aa)}{|Z_{,\aa}|^2}+ D_t b(\aa) } \frac{\P_H\paren{D_t^j Z_t}(\aa)\,\mathfrak D_t^l\theta\,\mathfrak D_t^i\bar\theta\,\mathfrak D_t^k\theta }{(\aa-\bb)^2}\,d\aa\,d\bb\\&+
\int \partial_\aa\P_A \paren{i\,\frac{A_1(\aa)}{|Z_{,\aa}|^2}+ D_t b(\aa) } \bracket{\P_H, <D_t^l\bar Z_t, D_t^i Z_t, D_t^k\bar Z_t> }\P_A\paren{D_t^j Z_t}(\aa)\,d\aa\\&=I_j+II_j+III_j,
\end{aligned}
\end{equation}
where in the last term we used the Cauchy integral formula to rewrite it as a commutator. We have for 
$j=2$, $(l,i,k)=(0,0,1)$ and permutations; $j=3$, $(l,i,k)=(0,0,2), (0,1,1)$ and permutations; $j=4$, $(l,i,k)=(0,0,3),(0,1,2),(1,1,1)$ and permutations. Observe that the first two terms in \eqref{238}, $I_j$, $II_j$, are sextic; we apply \eqref{eq:b115}, \eqref{hhalf44} to obtain
\begin{equation}\label{241}
\abs{I_2}+\abs{II_2}+\abs{I_4}+\abs{II_4}\lec \epsilon^5, \quad \abs{I_3}+\abs{II_3}\lec \epsilon^3\paren{\nm{Z_{t,\aa}}_{\dot H^{1/2}}^2+\nm{\partial_\aa\frac1{Z_{,\aa}}}_{L^2}^2};
\end{equation}
applying \eqref{eq:b10}, \eqref{eq:b20} on $III_j$ yields
\begin{equation}\label{243}
\abs{III_2}+\abs{III_4}\lec \epsilon^5, \quad \abs{III_3}\lec \epsilon^3\paren{\nm{Z_{t,\aa}}_{\dot H^{1/2}}^2+\nm{\partial_\aa\frac1{Z_{,\aa}}}_{L^2}^2},
\end{equation}
therefore 
$$\abs{M_{11,2}}+\abs{M_{11,4}}\lec \epsilon^5, \quad \abs{M_{11,3}}\lec \epsilon^3\paren{\nm{Z_{t,\aa}}_{\dot H^{1/2}}^2+\nm{\partial_\aa\frac1{Z_{,\aa}}}_{L^2}^2}.
$$
Sum up the estimates in {\sf Step 6} and {\sf Step 7} we have
\begin{equation}\label{244}
\abs{\mathcal R_{IC, 2}+\Re\frac d{dt} D_2(t)}+\abs{\mathcal R_{IC,4}}\lec \epsilon^5, \quad \abs{\mathcal R_{IC,3}}\lec \epsilon^3\paren{\nm{Z_{t,\aa}}_{\dot H^{1/2}}^2+\nm{\partial_\aa\frac1{Z_{,\aa}}}_{L^2}^2}.
\end{equation}

{\sf Step 8}. In this step we treat the remaining terms in \eqref{33}, namely 
\begin{align}
&2\Re\sum_{l=0}^{j-1}\int i\, \partial_\aa\bar\Thj(\P_H D_t)^{l+1}\P_H\paren{G^{(j-l)}-D_t\P_H G^{(j-1-l)}}\,d\aa-I_{1,j}-\frac d{dt}H_j,\\&
\Re\int i\,\partial_\aa \bar\Thj\P_H\paren{G^{(j+1)}-D_t\P_H G^{(j)}}\,d\aa-I_{2,j}.
\end{align}

Let  
\begin{equation}\label{190}
J_{l, j}=\frac1{2\pi}\iint D_t^j Z_t(\aa,t) \frac{\mathfrak D_t^{l+1}\braces{\mathfrak D_t\paren{\theta(\aa,\bb,t)\overline{\theta(\aa, \bb,t)}}\mathfrak D_t^{j-l-1}\theta(\aa,\bb,t)}}{(\aa-\bb)^2} \,d\bb\,d\aa,
\end{equation}
By \eqref{21}, \eqref{22}, 
$$I_{1,j}=2\Re \sum_{l=0}^{j-1} J_{l,j},\qquad I_{2,j}=\Re J_{-1, j}.$$
Because $$D_\aa\Thb=i(1-\frac 1{Z_{,\aa}})$$ by \eqref{eq:G}, \eqref{67} and the notations \eqref{18} and \eqref{62}\footnote{Observe that $\lambda^0=\theta$.} we can write 
\begin{equation}\label{248}
\P_H\paren{G^{(j-l)}-D_t\P_H G^{(j-1-l)}}=\frac1{2\pi i}\P_H\paren{\frac1{\bar Z_{,\aa}}\int\frac{(\theta\bar{\lambda^1}+\bar\theta\lambda^1)\lambda^{j-l-1}}{(\aa-\bb)^2}\,d\bb},
\end{equation}
and
\begin{equation}\label{245}
\begin{aligned}
&N_{l,j}:=2\pi\paren{\int i\,\partial_\aa \bar\Thj\paren{\P_H D_t}^{l+1}\P_H\paren{G^{(j-l)}-D_t\P_H G^{(j-1-l)}}\,d\aa-J_{l,j}}\\&=\int \,\partial_\aa \bar\Thj\braces{\paren{\P_H D_t}^{l+1}\P_H \frac1{\bar Z_{,\aa}}-\frac1{\bar Z_{,\aa}}D_t^{l+1}}\int\frac{(\theta\bar{\lambda^1}+\bar\theta\lambda^1)\lambda^{j-l-1}}{(\aa-\bb)^2}\,d\bb\,d\aa
\\&+ \int  \bar{D_\aa \Thj } \braces{D_t^{l+1} \int\frac{(\theta\bar{\lambda^1}+\bar\theta\lambda^1)\lambda^{j-l-1}}{(\aa-\bb)^2}\,d\bb-\int\frac{\mathfrak D_t^{l+1}\paren{(\theta\bar{\lambda^1}+\bar\theta\lambda^1)\lambda^{j-l-1}}}{(\aa-\bb)^2}\,d\bb}\,d\aa\\&+
 \int  \bar{D_\aa \Thj } \int\frac{\mathfrak D_t^{l+1}\paren{\theta(\bar{\lambda^1}-\mathfrak D_t\bar\theta)\lambda^{j-l-1}+\bar\theta(\lambda^1-\mathfrak D_t\theta)\lambda^{j-l-1}}}{(\aa-\bb)^2}\,d\bb\,d\aa\\&+
 \int  \bar{D_\aa \Thj } \int\frac{\mathfrak D_t^{l+1}\paren{\mathfrak D_t(\,\theta \,\bar\theta)(\lambda^{j-l-1}-\mathfrak D_t^{j-l-1}\theta)}}{(\aa-\bb)^2}\,d\bb\,d\aa\\&+\int  \paren{\bar{D_\aa \Thj }-D_t^j Z_t} \int\frac{\mathfrak D_t^{l+1}\paren{\mathfrak D_t(\,\theta \,\bar\theta)\mathfrak D_t^{j-l-1}\theta}}{(\aa-\bb)^2}\,d\bb\,d\aa
\\&= N_{l,j,1} +N_{l,j,2}+N_{l,j,3}+N_{l,j,4}+N_{l,j,5}.
\end{aligned}
\end{equation}
Observe that there is a derivative loss  in $N_{j-1,j,3}$ and $N_{j-1,j,5}$.\footnote{Namely $N_{j-1,j,3}$ and $N_{j-1,j,5}$ contain factors that cannot be controlled by $\mathcal E_j(t)$.} Because $\lambda^0=\theta$,  so $N_{j-1,j,4}=0$.  We combine $N_{j-1,j,3}+N_{j-1,j,5}$ with  $-\pi \frac d{dt} H_j$, 
cf. \eqref{64}, and write
\begin{equation}\label{246}
\begin{aligned}
&\Re (N_{j-1,j,3}+N_{j-1,j,5})-\pi \frac d{dt} H_j=\tilde N_{j,3}+\tilde N_{j, 4}+\tilde N_{j,5}-\pi R_{H,j}\\&:=\Re\int  \bar{D_\aa \Thj } \int\frac{\bar\theta \theta \paren{\mathfrak D_t^{j} \lambda^1-\mathfrak D_t \lambda^ j}+\theta \theta\paren{\mathfrak D_t^{j} \bar{\lambda^1}-\mathfrak D_t \bar{\lambda^j}} }{(\aa-\bb)^2}\,d\bb\,d\aa \\&+
 \Re \iint  
 D_t^j Z_t \frac{\mathfrak D_t^{j}\paren{\theta(\bar{\lambda^1}-\mathfrak D_t\bar\theta)\theta+\bar\theta(\lambda^1-\mathfrak D_t\theta)\theta}-\bar\theta \theta \mathfrak D_t^{j} (\lambda^1-\mathfrak D_t\theta)-\theta \theta\mathfrak D_t^{j} (\bar{\lambda^1}-\mathfrak D_t\bar\theta) }{(\aa-\bb)^2}\,d\bb\,d\aa\\&+
\Re\int  \paren{\bar{D_\aa \Thj }-D_t^j Z_t} \int\frac{\mathfrak D_t^{j}\paren{(\theta \,\bar{\lambda^1}+\bar\theta\lambda^1)\theta}- \bar\theta\,\theta\,\mathfrak D_t^{j}\lambda^1 -\theta\,\theta\,\mathfrak D_t^{j}\bar{\lambda^1} }{(\aa-\bb)^2}\,d\bb\,d\aa-\pi R_{H,j};
\end{aligned}
\end{equation}
this cancels out the derivative lossing terms for $j=4$, and enables us to get the desired estimates for $j=3$. 

We sum up the above decomposition:
\begin{equation}\label{247}
\begin{aligned}
&\pi\paren{2\Re\sum_{l=0}^{j-1}\int i\, \partial_\aa\bar\Thj(\P_H D_t)^{l+1}\P_H\paren{G^{(j-l)}-D_t\P_H G^{(j-1-l)}}\,d\aa-I_{1,j}-\frac d{dt} H_j(t)}\\&
=\Re\paren{\sum_{l=0}^{j-2} \sum_{k=1}^5N_{l,j, k}+\sum_{k=1}^2 N_{j-1,j, k}}+ \sum_{k=3}^5 \tilde N_{j,k}-\pi R_{H,j},
\end{aligned}
\end{equation}
and
\begin{align}\label{275}
&2\pi\paren{\Re\int i\,\partial_\aa \bar\Thj\P_H\paren{G^{(j+1)}-D_t\P_H G^{(j)}}\,d\aa-I_{2,j}}= \Re \sum_{k=1}^5N_{-1,j,k}.
\end{align}
Observe that all the terms in \eqref{245}, \eqref{246} are quintic.  

Now we use \eqref{eq:c35} to write
\begin{equation}\label{249}
\begin{aligned}
&\P_H\paren{(\P_H D_t)^{l+1}\P_H\frac 1{\bar Z_{,\aa}}- \frac1{\bar Z_{,\aa}}D_t^{l+1}}=  \paren{(\P_H D_t)^{l+1}\P_H-   \P_H D_t^{l+1}}\frac 1{\bar Z_{,\aa}} +\P_H \bracket{D_t^{l+1}, \frac1{\bar Z_{,\aa}}} \\&=
-\sum_{k=0}^l (\P_H D_t)^{k}\P_H \bracket{b,\P_A}\partial_\aa D_t^{l-k}\frac 1{\bar Z_{,\aa}}+\P_H\sum_{k=0}^l\binom{l+1}{k+1}\paren{D_t^{k+1}\frac1{\bar Z_{,\aa}}}D_t^{l-k};
\end{aligned}
\end{equation}
because
\begin{equation}\label{252}
D_\aa\Thb-\bar Z_{tt}=i\,\frac{A_1-1}{Z_{,\aa}},
\end{equation}
 by \eqref{eq:c48} we have
\begin{equation}\label{250}
\begin{aligned}
&D_\aa\Theta^{(l+1)}-D_t^{l+1}\bar Z_t=D_\aa\Theta^{(l+1)}-D_t^l D_\aa \Theta^{(1)}+D_t^l (D_\aa \Theta^{(1)}-\bar Z_{tt})\\&
=\sum_{k=0}^{l-1} (\P_HD_t)^k\bracket{\frac1{Z_{,\aa}},\P_H}\partial_\aa D_t\Theta^{(l-k)}+\sum_{k=0}^{l-1} (\P_H D_t)^k\bracket{\P_H, \frac1{Z_{,\aa}}D_\aa \Theta^{(l-k)}} Z_{t,\aa} \\&+i\,\sum_{k=0}^{l-1}D_t^k[\P_A, b]\partial_\aa (\P_H D_t)^{l-1-k}\frac1{Z_{,\aa}}+i\,D_t^l\paren{\frac{A_1-1}{Z_{,\aa}} },
\end{aligned}
\end{equation}
and
\begin{equation}\label{251}
\begin{aligned}
&D_tD_\aa\Theta^{(j)}-D_t^{j} D_\aa \Theta^{(1)}=\sum_{k=0}^{j-2} D_t(\P_HD_t)^k\bracket{\frac1{Z_{,\aa}},\P_H}\partial_\aa D_t\Theta^{(j-1-k)}\\&+\sum_{k=0}^{j-2} D_t(\P_H D_t)^k\bracket{\P_H, \frac1{Z_{,\aa}} D_\aa \Theta^{(j-1-k)}} Z_{t,\aa}+i\,\sum_{k=0}^{j-2}D_t^{k+1}[\P_A, b]\partial_\aa (\P_H D_t)^{j-2-k}\frac1{Z_{,\aa}}.
\end{aligned}
\end{equation}

We are now ready to do the estimates. We begin with $N_{l, j, 1}$. Observe that by the Cauchy integral formula, we can insert a $\P_H$  to write it as
\begin{equation}\label{257}
N_{l, j, 1}=\int \,\partial_\aa \bar\Thj\,\P_H\braces{\paren{\P_H D_t}^{l+1}\P_H \frac1{\bar Z_{,\aa}}-\frac1{\bar Z_{,\aa}}D_t^{l+1}}\int\frac{(\theta\bar{\lambda^1}+\bar\theta\lambda^1)\lambda^{j-l-1}}{(\aa-\bb)^2}\,d\bb\,d\aa
\end{equation}
 and it is clear that $N_{-1, j, 1}=0$. We estimate $\abs{N_{l, j,1}}$ for $0\le l\le j-1$, $2\le j\le 4$.  By \eqref{249} we need to estimate, for $0\le k\le l$, 
 \begin{align}\label{253}
 A_{k,l, j}:=\nm{(\P_H D_t)^{k}\P_H \bracket{b,\P_A}\partial_\aa D_t^{l-k}\paren{\frac 1{\bar Z_{,\aa}}\int\frac{(\theta\bar{\lambda^1}+\bar\theta\lambda^1)\lambda^{j-l-1}}{(\aa-\bb)^2}\,d\bb}}_{L^2},\\
 \label{254}
B_{k,l,j}:=\nm{\P_H\ \braces{\paren{D_t^{k+1}\frac1{\bar Z_{,\aa}}}D_t^{l-k} \int\frac{(\theta\bar{\lambda^1}+\bar\theta\lambda^1)\lambda^{j-l-1}}{(\aa-\bb)^2}\,d\bb} }_{L^2}.
\end{align}
By \eqref{eq:c38}, \eqref{eq:b18}, \eqref{eq:b19}, \eqref{eq:b115},  we have, for $k<l$ or $l<j-1$ and $2\le j\le 4$,
\begin{align}\label{258}
\nm{D_t^{l-k}\int\frac{(\theta\bar{\lambda^1}+\bar\theta\lambda^1)\lambda^{j-l-1}}{(\aa-\bb)^2}\,d\bb}_{L^2}\lec \epsilon^3,
\end{align}
and by  \eqref{97} and \eqref{eq:b20}, \eqref{eq:b21},
\begin{align}\label{255}
\nm{ \int\frac{(\theta\bar{\lambda^1}+\bar\theta\lambda^1)\lambda^{1}}{(\aa-\bb)^2}\,d\bb}_{\dot H^{1/2}}+\nm{D_t \int\frac{(\theta\bar{\lambda^1}+\bar\theta\lambda^1)\lambda^{0}}{(\aa-\bb)^2}\,d\bb}_{\dot H^{1/2}}\lec \epsilon^2\paren{\nm{\partial_\aa\frac1{Z_{,\aa}}}_{L^2}+\nm{Z_{t,\aa}}_{\dot H^{1/2}}},\\ \label{256}
\nm{\int\frac{(\theta\bar{\lambda^1}+\bar\theta\lambda^1)\lambda^{0}}{(\aa-\bb)^2}\,d\bb}_{\dot H^{1/2}}\lec \epsilon^3,\quad \nm{\partial_\aa \int\frac{(\theta\bar{\lambda^1}+\bar\theta\lambda^1)\lambda^{0}}{(\aa-\bb)^2}\,d\bb}_{L^2}\lec \epsilon^2\nm{\partial_\aa\frac1{Z_{,\aa}}}_{L^2},
\end{align}
therefore by \eqref{eq:c46}, \eqref{eq:c49},  \eqref{eq:b11}, \eqref{eq:b23}, \eqref{eq:b12}, \eqref{eq:b111}, \eqref{eq:b115}, we have, except for the cases where $l=k=j-1$ for $j=3,4$, 
\begin{equation}\label{259}
A_{k,l, 2}\lec \epsilon^4, \quad A_{k,l,3}\lec\epsilon^3 \paren{\nm{\partial_\aa\frac1{Z_{,\aa}}}_{L^2}+\nm{Z_{t,\aa}}_{\dot H^{1/2}}},\quad A_{k,l, 4}\lec \epsilon^4.
\end{equation}
We also have, by \eqref{eq:c38},  \eqref{eq:b18}, \eqref{eq:b115}, \eqref{eq:b24}, \eqref{hhalf44}, that except for the cases where $l=k=j-1$ for $2\le j\le 4$, 
\begin{equation}\label{270}
B_{k,l, 2}\lec \epsilon^4, \quad B_{k,l,3}\lec\epsilon^3 \paren{\nm{\partial_\aa\frac1{Z_{,\aa}}}_{L^2}+\nm{Z_{t,\aa}}_{\dot H^{1/2}}},\quad B_{k,l, 4}\lec \epsilon^4.
\end{equation}
Now all the terms in $A_{j-1,j-1, j}$, $j=3,4$, after expanding by \eqref{eq:c46}, \eqref{eq:c49},  can be handled similarly, except for one, namely
\begin{equation}\label{260}
\P_H\bracket{D_t^{j-1}b, \P_A}\partial_\aa \paren{\frac 1{\bar Z_{,\aa}}\int\frac{(\theta\bar{\lambda^1}+\bar\theta\lambda^1)\lambda^{0}}{(\aa-\bb)^2}\,d\bb};
\end{equation}
for this term, we write 
\begin{equation}\label{272}
D_t^{j-1}b=D_t^{j-1}b-2\Re \frac{D_t^{j-1} Z_t}{Z_{,\aa}}+ 2\Re \frac{D_t^{j-1} Z_t}{Z_{,\aa}}
\end{equation}
and use \eqref{168} and \eqref{216} and \eqref{eq:b11} to estimate the first term and \eqref{eq:b10} to estimate the second term, 
\begin{equation}\label{261}
\begin{aligned}
&\nm{\P_H\bracket{D_t^{j-1}b, \P_A}\partial_\aa \paren{\frac 1{\bar Z_{,\aa}}\int\frac{(\theta\bar{\lambda^1}+\bar\theta\lambda^1)\lambda^{0}}{(\aa-\bb)^2}\,d\bb}}_{L^2}\\&\lec \nm{\partial_\aa\paren{D_t^{j-1}b-2\Re \frac{D_t^{j-1} Z_t}{Z_{,\aa}}}}_{L^2}\nm{\frac 1{\bar Z_{,\aa}}\int\frac{(\theta\bar{\lambda^1}+\bar\theta\lambda^1)\lambda^{0}}{(\aa-\bb)^2}\,d\bb}_{\dot H^{1/2}}\\&+\nm{\frac{D_t^{j-1} Z_t}{Z_{,\aa}}}_{\dot H^{1/2}}\nm{\partial_\aa \paren{\frac 1{\bar Z_{,\aa}}\int\frac{(\theta\bar{\lambda^1}+\bar\theta\lambda^1)\lambda^{0}}{(\aa-\bb)^2}\,d\bb}}_{L^2},
\end{aligned}
\end{equation}
and we get
\begin{equation}\label{262}
A_{2,2,3} \lec\epsilon^3 \paren{\nm{\partial_\aa\frac1{Z_{,\aa}}}_{L^2}+\nm{Z_{t,\aa}}_{\dot H^{1/2}}},\qquad A_{3,3, 4}\lec \epsilon^4.
\end{equation}
To estimate $B_{j-1,j-1,j}$ for $2\le j\le 4$, we first compute using \eqref{eq:c26},
\begin{equation}\label{263}
\begin{aligned}
D_t^j\frac1{\bar Z_{,\aa}}&=D_t^{j-1}\paren{\frac1{\bar Z_{,\aa}}(b_\aa-2\Re D_\aa Z_t)}+\bracket{D_t^{j-1}, \frac1{|Z_{,\aa}|^2}\partial_\aa}Z_t\\&+\frac1{|Z_{,\aa}|^2}\partial_\aa\P_H D_t^{j-1}Z_t+\frac1{|Z_{,\aa}|^2}\partial_\aa\P_A D_t^{j-1}Z_t,
\end{aligned}
\end{equation}
and by \eqref{eq:c28}, we write
\begin{equation}\label{264}
\bracket{D_t^{j-1}, \frac1{|Z_{,\aa}|^2}\partial_\aa}Z_t=\sum_{k=0}^{j-2} D_t^k \paren{\frac{b_\aa-2\Re D_\aa Z_t}{|Z_{,\aa}|^2}\partial_\aa D_t^{j-2-k}Z_t},
\end{equation}
and using the fact $\P_H Z_t=0$ we write
\begin{equation}\label{265}
\P_H D_t^{j-1}Z_t=\sum_{k=0}^{j-2} D_t^k\bracket{\P_H, D_t} D_t^{j-2-k} Z_t=\sum_{k=0}^{j-2} D_t^k\bracket{\P_H, b} \partial_\aa D_t^{j-2-k} Z_t;
\end{equation}
using the expansions in \eqref{263}, \eqref{264} and \eqref{265} and the estimates in \S\ref{quan}, we have, 
\begin{align}\label{266}
\nm{ D_t^j\frac1{\bar Z_{,\aa}}-\frac1{|Z_{,\aa}|^2}\partial_\aa\P_A D_t^{j-1}Z_t}_{L^2}\lec \epsilon^2, \qquad j=2,4,\\ \label{269}\nm{ D_t^3\frac1{\bar Z_{,\aa}}-\frac1{|Z_{,\aa}|^2}\partial_\aa\P_A D_t^{2}Z_t}_{L^2}\lec \epsilon\paren{\nm{\partial_\aa\frac1{Z_{,\aa}}}_{L^2}+\nm{Z_{t,\aa}}_{\dot H^{1/2}}};
\end{align}
now we rewrite
\begin{equation}\label{267}
\begin{aligned}
&\P_H \braces{\paren{D_t^j\frac1{\bar Z_{,\aa}}}\int\frac{(\theta\bar{\lambda^1}+\bar\theta\lambda^1)\lambda^{0}}{(\aa-\bb)^2}\,d\bb}\\&=\P_H \braces{\paren{D_t^j\frac1{\bar Z_{,\aa}}-\frac1{|Z_{,\aa}|^2}\partial_\aa\P_A D_t^{j-1}Z_t}\int\frac{(\theta\bar{\lambda^1}+\bar\theta\lambda^1)\lambda^{0}}{(\aa-\bb)^2}\,d\bb}\\&+
\bracket{\P_H,\frac1{|Z_{,\aa}|^2}\int\frac{(\theta\bar{\lambda^1}+\bar\theta\lambda^1)\lambda^{0}}{(\aa-\bb)^2}\,d\bb}\partial_\aa\P_A D_t^{j-1}Z_t,
\end{aligned}
\end{equation}
by \eqref{eq:b18},  \eqref{eq:b21},
\begin{align}
\nm{\int\frac{(\theta\bar{\lambda^1}+\bar\theta\lambda^1)\lambda^{0}}{(\aa-\bb)^2}\,d\bb}_{L^\infty}&\lec \nm{Z_{t,\aa}}_{L^2}^2,\\
 \nm{\int\frac{(\theta\bar{\lambda^1}+\bar\theta\lambda^1)\lambda^{0}}{(\aa-\bb)^2}\,d\bb}_{\dot H^{1}}&\lec \nm{Z_{t,\aa}}_{L^2}^2\nm{\partial_\aa\frac1{Z_{,\aa}}}_{L^2};
\end{align}
therefore by  \eqref{eq:b11}, \eqref{266}, \eqref{269}, we have
\begin{equation}\label{268}
B_{2,2,3} \lec\epsilon^3 \paren{\nm{\partial_\aa\frac1{Z_{,\aa}}}_{L^2}+\nm{Z_{t,\aa}}_{\dot H^{1/2}}},\qquad B_{1,1,2}+B_{3,3, 4}\lec \epsilon^4. 
\end{equation}
Sum up \eqref{257}-\eqref{268} we conclude 
\begin{equation}\label{271}
\begin{aligned}
\abs{N_{l,j,1}}&\lec \epsilon^5, \qquad {-1\le l\le j-1,\ j=2,4}; \\
\abs{N_{l,3,1}}&\lec \epsilon^3\paren{\nm{\partial_\aa\frac1{Z_{,\aa}}}_{L^2}+\nm{Z_{t,\aa}}_{\dot H^{1/2}}}^2,\quad -1\le l\le 2.
\end{aligned}
\end{equation}

Now we handle $N_{l,j,2}$. Observe that $N_{-1, j, 2}=0$, so we work on the cases where $0\le l\le j-1$. We use \eqref{eq:c38} to expand. The estimates are routine for all the terms after expansion, using \eqref{eq:b28}, \eqref{eq:b37}, \eqref{eq:b115}, \eqref{hhalf44}, \eqref{eq:b46}, \eqref{eq:b43}, except that when $l=j-1$ for $j=3,4$, we need to again decompose $D_t^{j-1}b$ by \eqref{272}, and treat all the terms as usual, except 
\begin{equation}\label{273}
\int \partial_\bb\Re\frac{D_t^{j-1} Z_t(\bb)}{Z_{,\bb}}\frac{(\theta\bar{\lambda^1}+\bar\theta\lambda^1)\lambda^{0}}{(\aa-\bb)^2}\,d\bb,
\end{equation}
for which we first perform integration by parts, then apply \eqref{eq:b28}, \eqref{eq:b37} for $\rb= \Re\frac{D_t^{j-1} Z_t}{Z_{,\aa}}$. We have
\begin{equation}\label{274}
\begin{aligned}
\abs{N_{l,j,2}}&\lec \epsilon^5, \qquad {-1\le l\le j-1,\ j=2,4}; \\ 
 \abs{N_{l,3,2}}&\lec \epsilon^3\paren{\nm{\partial_\aa\frac1{Z_{,\aa}}}_{L^2}+\nm{Z_{t,\aa}}_{\dot H^{1/2}}}^2,\quad -1\le l\le 2.
\end{aligned}
\end{equation}

We consider  the remaining terms in \eqref{247}-\eqref{275}. Using \eqref{250} and the estimates in \S\ref{quan}, \S\ref{step1-4}, \eqref{eq:b10}, \eqref{eq:b23}, \eqref{eq:b11}, \eqref{eq:b12}, \eqref{eq:b115} we have
\begin{align}\label{276}
\nm{D_\aa\Thb-D_t\bar Z_t}_{\dot H^{1/2}}&\lec\epsilon^2,\\ \label{278}
 \nm{D_\aa\Thb-D_t\bar Z_t}_{\dot H^{1}}&\lec\epsilon \paren{\nm{\partial_\aa\frac1{Z_{,\aa}}}_{L^2}+\nm{Z_{t,\aa}}_{\dot H^{1/2}}},\\
\label{277}
\nm{D_t^l\paren{D_\aa\Th^{(j-l)}-D_t^{j-l}\bar Z_t}}_{L^2}&\lec\epsilon^2,\quad j=2,4, \ 0\le l\le j, \\ \label{279}
\nm{D_t^l\paren{D_\aa\Th^{(3-l)}-D_t^{3-l}\bar Z_t}}_{L^2}&\lec\epsilon \paren{\nm{\partial_\aa\frac1{Z_{,\aa}}}_{L^2}+\nm{Z_{t,\aa}}_{\dot H^{1/2}}},\quad 0\le l\le 3;
\end{align}
and using \eqref{251} we have\footnote{Again we use the decomposition \eqref{272} to treat the term $\bracket{\P_A, D_t^{j-1} b}\partial_\aa\frac1{Z_{,\aa}}$.}
\begin{equation}\label{280}
\begin{aligned}
\nm{D_tD_\aa\Theta^{(j)}-D_t^{j} D_\aa \Theta^{(1)}}_{L^2}&\lec \epsilon \paren{\nm{\partial_\aa\frac1{Z_{,\aa}}}_{L^2}+\nm{Z_{t,\aa}}_{\dot H^{1/2}}}, \quad \text{for }j=2,3,\\
\nm{D_tD_\aa\Theta^{(4)}-D_t^{4} D_\aa \Theta^{(1)}}_{L^2}&\lec \epsilon^2;
\end{aligned}
\end{equation}
this gives, by \eqref{eq:b18}, \eqref{eq:b19}, 
\begin{equation}\label{281}
\begin{aligned}
\sum_{l=-1}^{j-2} \sum_{k=3}^5\abs{N_{l,j, k}}+ \sum_{k=3}^5\abs{ \tilde N_{j,k}}+\abs{ R_{H,j}}&\lec \epsilon^5,\qquad \text{for }j=2,4;
\\
\sum_{l=-1}^{j-2} \sum_{k=3}^5\abs{N_{l,j, k}}+ \sum_{k=3}^5\abs{ \tilde N_{j,k}}+\abs{ R_{H,j}}&\lec \epsilon^3 \paren{\nm{\partial_\aa\frac1{Z_{,\aa}}}_{L^2}+\nm{Z_{t,\aa}}_{\dot H^{1/2}}}^2, \quad \text{for }j=3.
\end{aligned}
\end{equation}

Sum up the results in {\sf Steps 6-8}, we get
\begin{equation}\label{282}
\frac d{dt} \paren{\mathcal E_2(t)+\mathcal E_3(t)+\mathcal E_4(t)}\lec \epsilon^5,\qquad \frac d{dt} \mathcal E_3(t)\lec \epsilon^3 \paren{\nm{\partial_\aa\frac1{Z_{,\aa}}}_{L^2}+\nm{Z_{t,\aa}}_{\dot H^{1/2}}}^2.
\end{equation}
This together with \eqref{180}, \eqref{110} gives the inequalities \eqref{47} and \eqref{51}, and finishes the proof for  Propositions~\ref{prop:4.3},~\ref{prop:4.4} and Theorem~\ref{thm:main2}.

\begin{appendix}
\section{Notations and conventions}\label{notations}
We use the following notations and conventions throughout the paper: 
compositions are always in terms of the spatial variables and we write for $f=f(\cdot, t)$, $g=g(\cdot, t)$, $f(g(\cdot,t),t):=f\circ g(\cdot, t):=U_gf(\cdot,t)$. 
We identify $(x,y)$ with the complex number $x+iy$; $\Re z$, $\Im z$ are the real and imaginary parts of $z$; $\bar z=\Re z-i\Im z$ is the complex conjugate of $z$. $\overline \Omega$ is the closure of the domain $\Omega$, $\partial\Omega$ is the boundary of $\Omega$, ${\mathscr P}_-:=\{z\in \mathbb C: \Im z<0\}$ is the lower half plane.
$[A, B]:=AB-BA$ is the commutator of operators $A$ and $B$.

We use $z=x+iy=z(\alpha,t)$, $z_t=z_t(\alpha,t)$ and $z_{tt}(\alpha,t)$ to denote the position, velocity and acceleration of the interface in Lagrangian coordinate $\alpha$;  $Z=X+iY=Z(\aa,t)$, $Z_t=Z_t(\aa,t)$ and $Z_{tt}(\aa,t)$ denote the position, velocity and acceleration of the interface in the Riemann mapping variable $\aa$; $\rh(\a,t)=\aa$ is the coordinate change from the Lagrangian variable $\a$ to the Riemann mapping variable $\aa$;\footnote{ $\rh(\a,t)=\Phi(z(\a,t);t)$, where $\Phi(\cdot,t):\Omega\to\mathscr P_-$ is the Riemann mapping satisfying $\Phi(z(0,t);t)=0$ and $\lim_{z\to\infty} \Phi_z(z,t)=1$.} $b=\rh_t\circ \rh^{-1}$, and  the material derivative is $D_t=\partial_t+b\partial_\aa$. 
We write
$$Z_{,\aa}=\partial_\aa Z(\aa,t),\qquad Z_{t,\aa}=\partial_\aa Z_t,\qquad Z_{tt,\aa}=\partial_\aa Z_{tt},\qquad\text{etc.}$$

 Let $\mathbb H$ be the Hilbert transform associated with the lower half plane ${\mathscr P}_-$:
\begin{equation}\label{ht}
\mathbb H f(\alpha')=\frac1{\pi i}\text{pv.}\int\frac1{\alpha'-\beta'}\,f(\beta')\,d\beta'.
\end{equation}
We know $\HH^2=I$, and a function $f\in L^p(\mathbb R)$, $1\le p<\infty$, is the boundary value of a holomorphic function in $\mathscr P_-$ if and only if $f=\HH f$.  We define 
the projections to the space of holomorphic, and respectively, anti-holomorphic functions in the lower half plane  by
\begin{equation}\label{proj}
\mathbb P_H :=\frac12(I+\mathbb H),\qquad\text{and }\quad \mathbb P_A:=\frac12(I-\mathbb H).
\end{equation}
It is clear that the decomposition identity 
\begin{equation}\label{paph}
\P_H+\P_A=I
\end{equation}
and the projection identity
\begin{equation}\label{projid}
\P_H\P_A=\P_A\P_H=0
\end{equation}
hold. We will often call a function $f\in L^p$, $1\le p<\infty$, that is the boundary value of a holomorphic function in $\mathscr P_-$ simply by "holomorphic".

We define 
\begin{equation}\label{da-daa}
 D_\aa = \dfrac { 1}{Z_{,\aa}}\partial_\aa,
\end{equation}
\begin{equation}\label{eq:comm}
[f,g; h]:=\frac1{\pi i}\int\frac{(f(x)-f(y))(g(x)-g(y))}{(x-y)^2}h(y)\,dy,
\end{equation}
and
\begin{equation}\label{tri}
<f,g,h>:=\frac1{\pi i}\int\frac{(f(x)-f(y))(g(x)-g(y))(h(x)-h(y))}{(x-y)^2}\,dy.
\end{equation}

We use the following notations for functional spaces:    $H^s=H^s(\mathbb R)$ is the Sobolev space with norm $\|f\|_{H^s}:=(\int (1+|\xi|^2)^s|\hat f(\xi)|^2\,d\xi)^{1/2}$, $\dot H^{s}=\dot H^{s}(\mathbb R)$ is the homogeneous Sobolev space  with norm 
$\|f\|_{\dot H^{s}}= c(\int |\xi|^{2s} |\hat f(\xi)|^2\,d\xi)^{1/2}$, and we define
\begin{equation}\label{def-hhalf}
\|f\|_{\dot H^{1/2}}^2=\|f\|_{\dot H^{1/2}(\mathbb R)}^2:= \int i\mathbb H \partial_x f(x) \bar f(x)\,dx=\frac1{2\pi}\iint\frac{|f(x)-f(y)|^2}{(x-y)^2}\,dx\,dy.
\end{equation}
$L^p=L^p(\mathbb R)$ is the $L^p$ space with $\|f\|_{L^p}:=(\int|f(x)|^p\,dx)^{1/p}$ for $1\le p<\infty$, and $f\in L^\infty$ if $\|f\|_{L^\infty}:=\text{ ess sup }|f(x)|<\infty$. When not specified, all the 
 norms $\|f\|_{H^s}$, $\|f\|_{\dot H^{s}}$, $\|f\|_{L^p}$, $1\le p\le\infty$ are in terms of the spatial variable only, and $\|f\|_{H^s(\mathbb R)}$, $\|f\|_{\dot H^{s}(\mathbb R)}$, $\|f\|_{L^p(\mathbb R)}$,  $1\le p\le\infty$ are in terms of the spatial variable. 
 $C^j(X)$ is the space of $j$-times continuously differentiable functions on the set $X$; $C^j_0(\mathbb R)$ is the space of $j$-times continuously differentiable functions that decays at the infinity.

We use $c$, $C$  to denote universal constants. $c(a_1,  \dots )$, $C(a_1, \dots)$, $M(a_1, \dots)$  are constants depending on $a_1, \dots $; constants appearing in different contexts need not be the same. We write $f\lec g$ if there is a universal constant $c$, such that $f\le cg$.  

The following are some additional notations used in this paper:

$Q:=(I+\HH)\psi\circ \rh^{-1}$, 
where $\psi\circ \rh^{-1}$ is the trace of the velocity potential on the interface; 
$\Theta^{(0)}:=Q$, $\Thj:=(\P_HD_t)^jQ$, and 
$G^{(j)}:=D_t\mathbb P_H D_t \Thj+i\,\frac1{|Z_{,\aa}|^2}\partial_\aa \Thj$.  

We define
$\theta:=\bar Z_t(\aa,t)-\bar Z_t(\bb, t)$; $\lambda^j=D_\aa\Theta^{(j)}(\aa,t)-D_\bb\Theta^{(j)}(\bb,t)$; $\mathfrak D_t:=\partial_t+b(\aa,t)\partial_\aa+b(\bb,t)\partial_\bb$, and 
$\mathcal P:=\mathfrak D_t^2+i\,\frac{A_1(\aa,t)}{|Z_{,\aa}|^2}\partial_\aa+i\,\frac{A_1(\bb,t)}{|Z_{,\bb}|^2}\partial_\bb$. Observe that when acting on a function independent of $\bb$, i.e. $f=f(\aa,t)$, $\mathcal P f=( D_t^2+i\,\frac{A_1}{|Z_{,\aa}|^2} \partial_\aa)f$. 

We denote by $M(f)$ the Hardy-Littlewood maximum function of $f$.

\section{Equations and Identities}\label{iden}
Here we give some basic equations and identities that will be used in this paper. 
First we recall some of the equations and formulas derived in our earlier work,  see \cite{wu1, wu6} or \S2.2, \S2.3 and \S2.7 of \cite{wu8}. 

\subsection{Interface equations}\label{interface-eq}
We know that the interface equations for the 2d water waves is given by
 \begin{equation}\label{2dinterface}
 \begin{cases}
 \bar Z_{tt}-i=-i \dfrac {A_1}{Z_{,\aa}}\\
 \bar Z_t=\mathbb H \bar Z_t,\quad \frac1{Z_{,\aa}}-1=\mathbb H \paren{\frac1{Z_{,\aa}}-1},
 \end{cases}
 \end{equation} 
where the quantities $A_1$ and $b$ satisfy
\begin{equation}\label{A1b}
A_1=1-\Im \bracket{Z_t, \mathbb H}\bar Z_{t,\aa}=1-\frac12\Im[Z_t, \bar Z_t;1],\qquad b=\Re (I-\mathbb H)\paren{\frac{Z_t}{Z_{,\aa}}}.
\end{equation}
Let \begin{equation}\label{op}
\mathfrak P:=D_t^2+i\,\frac{A_1}{|Z_{,\aa}|^2} \partial_\aa.
\end{equation}
The quasi-linear equation for the water waves is
\begin{equation}\label{quasi}
 \mathcal P {\bar Z}_t=\mathfrak P {\bar Z}_{t}=\dfrac{\frak a_t}{\frak a}\circ \rh^{-1} ({\bar Z}_{tt}-i) 
\end{equation}
where 
\begin{equation}\label{at}
\dfrac{\frak a_t}{\frak a}\circ \rh^{-1}= \dfrac{D_t A_1}{A_1}+b_\aa -2\Re D_\aa Z_t,
\end{equation}
with
\begin{equation}\label{dta1}
D_t A_1= -\Im \paren{\bracket{Z_{tt},\mathbb H}\bar Z_{t,\alpha'}+\bracket{Z_t,\mathbb H}\partial_\aa \bar Z_{tt}-[Z_t, b; \bar Z_{t,\aa}]}=-\Im\paren{\bracket{\bar Z_t, Z_{tt};1}-\bracket{Z_t, b; \bar Z_{t,\aa}}},
\end{equation}
and\footnote{The second equality in \eqref{dta1} and \eqref{ba} are obtained by integration by parts.}
\begin{equation}\label{ba}
b_\aa-2\Re D_\aa Z_t=\Re \paren{\bracket{ \frac1{Z_{,\aa}}, \mathbb H}  Z_{t,\alpha'}+ \bracket{Z_t, \mathbb H}\partial_\aa \frac1{Z_{,\aa}}  }=\frac12\bracket{\frac1{Z_{,\aa}}, Z_t; 1}-\frac12\bracket{\bar Z_t, \frac1{\bar Z_{,\aa}}; 1}.
\end{equation}
We also have
\begin{align}\label{eq:c26}
D_t\paren{\frac{1}{Z_{,\alpha'}}}=\frac{1}{Z_{,\alpha'}}(b_\aa-D_\aa Z_t)=\frac{1}{Z_{,\alpha'}}(b_\aa-2\Re D_\aa Z_t)+\frac{\bar Z_{t,\aa}}{|Z_{,\alpha'}|^2}.
\end{align}

\subsection{Basic identities}

We  give some basic identities that will be used in our calculations. We begin with Proposition~\ref{prop:comm-hilbe}, which  is a consequence of the fact that the product of holomorphic functions is holomorphic and \eqref{projid}.

\begin{proposition}\label{prop:comm-hilbe}
Assume that $f,\ g \in L^2(\mathbb R)$. 

1. Assume either both $f$, $g$ are holomorphic: $f=\mathbb H f$, $g=\mathbb H g$, or both are anti-holomorphic: $f=-\mathbb H f$, $g=-\mathbb H g$. Then
\begin{equation}\label{comm-hilbe}
[f, \mathbb H]g=0.
\end{equation}

2. If  $\P_A f=\P_A g=0$, then $\P_A( fg)=0$; and if $\, \P_H f=\P_H g=0$, then $\P_H( fg)=0$.
\end{proposition}

\begin{proposition}\label{prop:tri-comm-iden}
1. We have
\begin{equation}\label{eq:c29}
[f,g;h]=[f,\HH]\partial_{\aa}(gh)+[g,\HH]\partial_\aa(fh)-[fg, \HH]\partial_\aa h.
\end{equation}
2. If $h$ is holomorphic, i.e. $\P_A h=0$, then 
\begin{equation}\label{eq:c30}
\P_H [f,g; h]=-2\P_H (f\partial_\aa \P_A(g h))-2\P_H (g\partial_\aa \P_A(f h)).
\end{equation}
\end{proposition}
\eqref{eq:c29} is obtained by integration by parts. \eqref{eq:c30} follows from \eqref{eq:c29}, \eqref{projid} and \eqref{paph} . 

We often use the following equalities to calculate the time and spatial derivatives to our energy functionals,  and the following version of the Cauchy integral formula to derive our equations and formulas. 

\begin{proposition}\label{prop:dte}
For $f$, $g$ smooth and decay fast at infinity, \begin{align}\label{dte}
\frac d{dt}\int f(\aa,t)\,d\aa&=\int (D_t+b_\aa)f(\aa,t)\,d\aa;\\
\label{dteab}
\frac d{dt}\iint g(\aa,\bb, t)\,d\aa d\bb&=\iint (\mathfrak D_t+b_\aa+b_\bb)g(\aa,\bb,t)\,d\aa d\bb;\\
\label{dta}
D_t\int g(\aa,\bb, t)\,d\bb&=\int (\mathfrak D_t+b_\bb)g(\aa,\bb,t)\, d\bb;\\
\label{da}
\partial_\aa \int g(\aa,\bb, t)\,d\bb&=\int (\partial_\aa+\partial_\bb) g(\aa,\bb,t)\, d\bb.
\end{align}
\end{proposition}
\eqref{dte}, \eqref{dteab}, \eqref{dta}, \eqref{da} follow from the simple fact that $\int (b\partial_\aa+b_\aa)f\,d\aa=\int\partial_\aa(bf)\,d\aa=0$,  $\int (b\partial_\bb+b_\bb)g\,d\bb=\int\partial_\bb(bg)\,d\bb=0$,
and $\int\partial_\bb g \,d\bb=0$.

\begin{proposition}[Cauchy integral formula]\label{prop:cif}
 For any  $\Th\in L^1(\mathbb R)$,  satisfying $\P_A\Th=0$ or $\P_H\Th=0$, 
\begin{equation}\label{eq:c31}
\int\Theta(\aa)\,d\aa=0.
\end{equation}
\end{proposition}

\subsection{Commutator identities}\label{comm-iden}

We include here various commutator identities that are necessary for our proofs. Some have already appeared in Appendix B.5 of \cite{kw} and Appendix B of \cite{wu8}. We have 
\begin{align}
  \label{eq:c1-1}
  [D_t,D_\aa] &= - (D_\aa Z_t) D_\aa,\\ 
  \label{eq:c7}
[D_t, \partial_\aa] &=-b_\aa\partial_\aa .
   \end{align}
By product rules, 
 \begin{equation}\label{eq:c34}
[D_t, \frac{1}{Z_{,\alpha'}}]f=D_t\paren{\frac{1}{Z_{,\alpha'}}}f.
\end{equation}
From \eqref{eq:c26} and \eqref{eq:c7} it implies that 
\begin{align}\label{eq:c28}
\bracket{D_t,\, \frac1{|Z_{,\aa}|^2}\partial_\aa}f&=\frac{b_\aa-2\Re D_\aa Z_t}{|Z_{,\aa}|^2}\partial_\aa f,\\
\label{eq:c32}
\bracket{D_t,\mathfrak P}f=
\bracket{D_t,\, i\,\frac{A_1}{|Z_{,\aa}|^2}\partial_\aa}f&=i\,\paren{\frac{D_t A_1}{A_1}+b_\aa-2\Re D_\aa Z_t}\frac{A_1}{|Z_{,\aa}|^2}\partial_\aa f.
\end{align}
We also have
\begin{equation}\label{eq:c21}
[D_t, \mathbb H]=2[D_t,\P_H]=[b,\mathbb H]\partial_{\alpha'}=\bracket{\frac{Z_t}{Z_{,\aa}},\HH}\partial_\aa\P_H+ \bracket{\frac{\bar Z_t}{\bar Z_{,\aa}},\HH}\partial_\aa\P_A,
\end{equation}
where the last equality  in \eqref{eq:c21} is a consequence of \eqref{b}, \eqref{paph} and \eqref{comm-hilbe}. 
In general,  for operators $A, B$ and $C$,
\begin{equation}\label{eq:c12}
[A, BC]=[A, B]C+ B[A, C].
\end{equation}

We use the following identities in our computations. We have
\begin{align}
\label{eq:c33}
\bracket{\mathcal P, \mathfrak D_t^m}&=\sum_{k=0}^{m-1} \mathfrak D_t^k \bracket{\mathcal P, \mathfrak D_t}\mathfrak D_t^{m-k-1},\\
\label{eq:c46}
\P_HD_t\P_H&=\P_H D_t-\P_H D_t\P_A=\P_H D_t-\P_H  \bracket{b, \P_A}\partial_\aa,
\end{align}
and we compute 
$$\begin{aligned}
&(\P_H D_t)^{l+1}\P_H=\P_H D_t^{l+1} +\sum_{k=0}^l \paren{(\P_H D_t)^{k+1}\P_H D_t^{l-k}-(\P_H D_t)^k \P_H D_t^{l+1-k}}\\&=\P_H D_t^{l+1} -\sum_{k=0}^l (\P_H D_t)^{k}\P_H D_t\P_A D_t^{l-k}=
\P_H D_t^{l+1} -\sum_{k=0}^l (\P_H D_t)^{k}\P_H \bracket{b,\P_A}\partial_\aa D_t^{l-k},
\end{aligned}
$$
so
\begin{equation}\label{eq:c35}
(\P_H D_t)^{l+1}\P_H=
\P_H D_t^{l+1} -\sum_{k=0}^l (\P_H D_t)^{k}\P_H \bracket{b,\P_A}\partial_\aa D_t^{l-k}.
\end{equation}
A similar computation  also gives 
\begin{equation}\label{eq:c36}
(\P_H D_t)^{l+1}\P_H=
D_t^{l+1}\P_H -\sum_{k=0}^l D_t^k\bracket{\P_A, b}\partial_\aa(\P_HD_t)^{l-k}\P_H.
\end{equation}

Let $\mathfrak f:=f(\aa)-f(\bb)$, $\mathfrak g:=g(\aa)-g(\bb)$ and $\mathfrak h:=h(\aa)-h(\bb)$.  We use \eqref{dta} and \eqref{eq:c7} to get
\begin{equation}\label{eq:c37}
D_t<f, \,g, \,h>=\frac1{\pi i} \int\frac{\mathfrak D_t (\mathfrak f\,\mathfrak g\,\mathfrak h)}{(\aa-\bb)^2}\,d\bb+\frac1{\pi i} \int \partial_\bb\mathfrak D_t\paren{\frac1{\aa-\bb}}  \mathfrak f\,\mathfrak g\,\mathfrak h\,d\bb;
\end{equation}
and by induction,
\begin{equation}\label{eq:c38}
D_t^n<f, \, g, \, h>=\frac1{\pi i}\sum_{k=0}^n \binom{n}{k}  \int \partial_\bb\mathfrak D_t^k\paren{\frac1{\aa-\bb}}  \mathfrak D_t^{n-k} (\mathfrak f\,\mathfrak g\,\mathfrak h)\,d\bb,
\end{equation}
Similarly,
\begin{align}\label{eq:c39}
D_t^n[f,\, g;\, h]&=\frac1{\pi i}\sum_{k=0}^n \binom{n}{k}  \int \partial_\bb\mathfrak D_t^k\paren{\frac1{\aa-\bb}}  \mathfrak D_t^{n-k} \paren{(f(\aa)-f(\bb))( g(\aa)-g(\bb)) h(\bb)}\,d\bb;\\ \label{eq:c40}
D_t^n[f,\, g;\, \partial_\aa h]&=\frac1{\pi i}\sum_{k=0}^n \binom{n}{k}  \int  \mathfrak D_t^{k} \paren{\frac{(f(\aa)-f(\bb))( g(\aa)-g(\bb))}{(\aa-\bb)^2}} \partial_\bb D_t^{n-k} h(\bb) \,d\bb;
\\ \label{eq:c49}
D_t^n[f,\mathbb H]\partial_\aa g&=\frac1{\pi i}\sum_{k=0}^n \binom{n}{k}  \int  \mathfrak D_t^{k} \paren{\frac{f(\aa)-f(\bb)}{\aa-\bb}} \partial_\bb D_t^{n-k} g(\bb) \,d\bb;
\end{align}
where
\begin{align}\label{eq:c41}
\mathfrak D_t\paren{\frac1{\aa-\bb}}=-\frac{b(\aa)-b(\bb)}{(\aa-\bb)^2},\ \  \mathfrak D_t^2\paren{\frac1{\aa-\bb}}=-\frac{D_tb(\aa)-D_tb(\bb)}{(\aa-\bb)^2}+2\frac{(b(\aa)-b(\bb))^2}{(\aa-\bb)^3}, \ \text{etc;} 
\end{align}
in particular, 
\begin{align}\label{eq:c14}
D_t [f,\mathbb H]\partial_{\alpha'}g=
[D_tf,\mathbb H]\partial_{\alpha'}g+ [f,\mathbb H]\partial_{\alpha'}D_t g-[f, b; \partial_{\alpha'}g].
\end{align}

\subsection{Some additional equations}
We know by definition \eqref{Thj},  \eqref{Gj} and equation \eqref{q1} that
\begin{align}\label{eq:c42}
\Theta^{(0)}=Q, \quad \Theta^{(1)}=i(Z-\aa),\quad  \Theta^{(2)}=-i\P_H b=-i\P_H\frac{\bar Z_t}{\bar Z_{,\aa}},
\\ \label{eq:c43}
\Theta^{(j+2)}=-i\P_H(\frac1{| Z_{,\aa}|^2}\partial_\aa\Theta^{(j)})+\P_H(G^{(j)}), \qquad j\ge 1.
\end{align}
This implies
\begin{equation}\label{eq:c50}
D_\aa\Theta^{(0)}=\bar Z_t,\qquad D_\aa\Theta^{(1)}=i\paren{1-\frac1{Z_{,\aa}}};
\end{equation}
\begin{equation}\label{eq:c44}
b=i\,\Theta^{(2)}-i\,\bar {\Theta^{(2)}}, \qquad D_t b=i\,D_t\Theta^{(2)}-i\,\bar {D_t\Theta^{(2)}}, \qquad\text{and}
\end{equation}
\begin{equation}\label{eq:c45}
\P_H \paren{D_t b+i\frac{A_1}{|Z_{,\aa}|^2}-i}=i\P_H\paren{\frac{A_1-1}{|Z_{,\aa}|^2}}+i\P_H(G^{(1)})+i\bracket{b,\P_H}\partial_\aa \bar {\Theta^{(2)}}.
\end{equation}

We compute, by \eqref{eq:c33}, \eqref{eq:c32} that 
\begin{equation}\label{eq:c47}
\begin{aligned}
\bracket{\mathcal P, D_t^l} \bar Z_t &=\sum_{k=0}^{l-1}D_t^{l-1-k}\bracket{\mathcal P, D_t}D_t^k \bar Z_t
\\&=-i\sum_{k=0}^{l-1}D_t^{l-1-k}\braces{
\paren{\frac{D_t A_1}{A_1}+b_\aa-2\Re D_\aa Z_t}\frac{A_1}{|Z_{,\aa}|^2}\partial_\aa D_t^k \bar Z_t}.
\end{aligned}
\end{equation}
and we have, by \eqref{eq:c12}, \eqref{eq:c1-1} and \eqref{eq:c36}, that
\begin{equation}\label{eq:c48}
\begin{aligned}
D_\aa\Theta^{(l+1)}&-D_t^l D_\aa \Theta^{(1)}=\sum_{k=0}^{l-1} (\P_HD_t)^k\bracket{\frac1{Z_{,\aa}},\P_H}\partial_\aa D_t\Theta^{(l-k)}\\&+\sum_{k=0}^{l-1} (\P_H D_t)^k\P_H(D_\aa Z_t D_\aa \Theta^{(l-k)})-\sum_{k=0}^{l-1}D_t^k[\P_A, b]\partial_\aa (\P_H D_t)^{l-1-k}D_\aa\Theta^{(1)}.
\end{aligned}
\end{equation}

\section{Basic inequalities}\label{ineq}
We will use the following equalities or inequalities in this paper. The first set: Lemma~\ref{hhalf1} through  Proposition~\ref{prop:hhalf-2}  are either classical results, or simple consequences of definitions and classical results, and some have already appeared in Appendix A of \cite{wu8}. 
Lemma~\ref{hhalf2} and Proposition~\ref{hhalf4} are from  Section 5.1 of \cite{wu8}.

\begin{lemma}\label{hhalf1}
For any function $f\in \dot H^{1/2}(\mathbb R)$,
\begin{align}
\nm{f}_{\dot H^{1/2}}^2&=\nm{\mathbb P_H f}_{\dot H^{1/2}}^2+\nm{\mathbb P_A f}_{\dot H^{1/2}}^2;\label{hhalfp}\\
\int i\partial_\aa f \, \bar f\,d\aa&=\nm{\mathbb P_H f}_{\dot H^{1/2}}^2-\nm{\mathbb P_A f}_{\dot H^{1/2}}^2.\label{hhalfn}
\end{align}

\end{lemma}

\begin{proposition}\label{prop:Hhalf}
Let $f,\ g\in C^1(\mathbb R)$. Then
\begin{align}
\nm{fg}_{\dot H^{1/2}}&\lec \|f\|_{L^\infty}\|g\|_{\dot H^{1/2}}+\|g\|_{L^\infty}\|f\|_{\dot H^{1/2}};\label{hhalf-1}\\
\|fg\|_{\dot H^{1/2}}&\lec \|f\|_{L^\infty}\|g\|_{\dot H^{1/2}}+\|f'\|_{L^2}\|g\|_{L^2};\label{hhalf-2}
\\ \|g\|_{\dot H^{1/2}}&\lesssim \|f^{-1}\|_{L^\infty}(\|fg\|_{\dot H^{1/2}}+\|f'\|_{L^2}\|g\|_{L^2}).\label{Hhalf}
\end{align}

\end{proposition}
\eqref{hhalf-2} is straightforward from the definition of $\dot H^{1/2}$ and Hardy's inequality. The remaining two are from Appendix A of \cite{wu8}.

\begin{proposition}[Sobolev inequality]\label{sobolev}
Let $f\in C^1_0(\mathbb R)$. Then
\begin{equation}\label{eq:sobolev}
\|f\|_{L^\infty}^2\le 2\|f\|_{L^2}\|f'\|_{L^2},\qquad \|f\|_{\dot H^{1/2}}^2\le \|f\|_{L^2}\|f'\|_{L^2}.
\end{equation}
\end{proposition}

\begin{proposition}[Maximum inequality]\label{maximum}
Let $1<p\le \infty$, 
Then for all $f\in L^p$,
\begin{equation}\label{eq:maximum}
\nm{M(f)}_{L^p}\lec \nm{f}_{L^p}.
\end{equation}
\end{proposition}

\begin{proposition}[Hardy's inequalities]
\label{hardy-inequality}  
Let $1<p<\infty$, $f \in C^1(\mathbb R)$, 
with $f' \in L^p(\mathbb R)$. Then 
\begin{equation}
  \label{eq:77}
\sup_{x\in\mathbb R}\int \f{|f(x) - f(y)|^p}{|x-y|^p} dy\lec \nm{f'}_{L^p}^p;
\end{equation}
and
\begin{equation}
  \label{eq:771}
\iint \f{|f(x) - f(y)|^{2p}}{|x-y|^{2p}} \,dx dy \lec \nm{f'}_{L^p}^{2p}.
\end{equation}
\end{proposition}

Let $ H\in C^1(\mathbb R; \mathbb R^d)$, $A_i\in C^1(\mathbb R)$, $i=1,\dots m$, and $F\in C^\infty(\mathbb R)$. Define
\begin{equation}\label{3.15}
C_1(A_1,\dots, A_m, f)(x)=\text{pv.}\int F\paren{\frac{H(x)-H(y)}{x-y}} \frac{\Pi_{i=1}^m(A_i(x)-A_i(y))}{(x-y)^{m+1}}f(y)\,dy.
\end{equation}

\begin{proposition}\label{B1} There exist  constants $c_1=c_1(F, \|H'\|_{L^\infty})$, $c_2=c_2(F, \|H'\|_{L^\infty})$, such that 

1. For any $f\in L^2,\ A_i'\in L^\infty, \ 1\le i\le m, $
\begin{equation}\label{3.16}
\|C_1(A_1,\dots, A_m, f)\|_{L^2}\le c_1\|A_1'\|_{L^\infty}\dots\|A_m'\|_{L^\infty}\|f\|_{L^2}. 
\end{equation}
2. For any $ f\in L^\infty, \ A_i'\in L^\infty, \ 2\le i\le m,\ A_1'\in L^2$, 
\begin{equation}\label{3.17}
\|C_1(A_1,\dots, A_m, f)\|_{L^2}\le c_2\|A_1'\|_{L^2}\|A'_2\|_{L^\infty}\dots\|A_m'\|_{L^\infty}\|f\|_{L^\infty}.
\end{equation}
\end{proposition}
\eqref{3.16} is a result of Coifman, McIntosh and Meyer \cite{cmm}. \eqref{3.17} is a consequence of the Tb Theorem, a proof  is given in \cite{wu3}.

Let  $H$, $A_i$ $F$ satisfy the same assumptions as in \eqref{3.15}. Define
\begin{equation}\label{3.19}
C_2(A, f)(x)=\int F\paren{\frac{H(x)-H(y)}{x-y}}\frac{\Pi_{i=1}^m(A_i(x)-A_i(y))}{(x-y)^m}\partial_y f(y)\,dy.
\end{equation}
The following are consequences of Proposition~\ref{B1} and integration by parts.
\begin{proposition}\label{B2} There exist constants $c_3$, $c_4$ and $c_5$, depending on $F$ and $\|H'\|_{L^\infty}$, such that 

1. For any $f\in L^2,\ A_i'\in L^\infty, \ 1\le i\le m, $
\begin{equation}\label{3.20}
\|C_2(A, f)\|_{L^2}\le c_3\|A_1'\|_{L^\infty}\dots\|A_m'\|_{L^\infty}\|f\|_{L^2}.
\end{equation}

2. For any $ f\in L^\infty, \ A_i'\in L^\infty, \ 2\le i\le m,\ A_1'\in L^2$,
\begin{equation}\label{3.21}
\|C_2(A, f)\|_{L^2}\le c_4\|A_1'\|_{L^2}\|A'_2\|_{L^\infty}\dots\|A_m'\|_{L^\infty}\|f\|_{L^\infty}.\end{equation}

3. For any $f'\in L^2, \ A_1\in L^\infty,\ \ A_i'\in L^\infty, \ 2\le i\le m, $
\begin{equation}\label{3.22}
\|C_2(A, f)\|_{L^2}\le c_5\|A_1\|_{L^\infty}\|A'_2\|_{L^\infty}\dots\|A_m'\|_{L^\infty}\|f'\|_{L^2}.\end{equation}

\end{proposition}

\begin{proposition}\label{prop:half-dir} Assume that $f, g, h$ are smooth and decay fast at infinity. Then
\begin{align} 
 &\nm{[f,\HH]  g}_{L^2} \lec \nm{f}_{\dot{H}^{1/2}}\nm{g}_{L^2};\label{eq:b10}\\&
    \nm{[f,\HH] g}_{L^\infty} \lec \nm{f'}_{L^2} \nm{g}_{L^2};  \label{eq:b13}\\&
  \nm{[f,\HH] \partial_\aa g}_{L^2} \lec \nm{f'}_{L^2} \nm{g}_{\dot{H}^{1/2}};\label{eq:b11}\\&
 \nm{[f,\HH]\partial_\aa g}_{L^2}\lec \nm{f'}_{\dot H^{1/2}}\nm{g}_{L^2};\label{eq:b23} \\&
  \nm{[f, h; \partial_\aa g]}_{L^2}\lec \nm{f'}_{L^2} \nm{h'}_{L^\infty}\nm{g}_{\dot{H}^{1/2}},\label{eq:b111}\\&
   \nm{[f, h; \partial_\aa g]}_{L^2}\lec \nm{f'}_{L^2} \nm{h'}_{\dot H^{1/2}}\nm{g}_{L^\infty}+\nm{f'}_{\dot H^{1/2}} \nm{h'}_{L^2}\nm{g}_{L^\infty}.\label{eq:b121}
  \end{align}
\end{proposition}
\eqref{eq:b23} follows from integration by parts, Cauchy-Schwarz inequality and the definition \eqref{def-hhalf}, and \eqref{eq:b121} follows from integration by parts, Cauchy-Schwarz inequality, Hardy's inequality and the definition \eqref{def-hhalf}. The remaining inequalities are from Appendix A of \cite{wu8}.

\begin{proposition}\label{prop:double-comm} For any $f, \ g, \ h$ smooth and decay fast at spatial infinity, we have 
  \begin{align}
    \label{eq:b12}
    \nm{[f,g;h]}_{L^p} &
    \lec \nm{f'}_{L^2} \nm{g'}_{L^2} \nm{h}_{L^p}, \qquad{1\le p\le \infty};
    \\ \label{eq:b22}
    \nm{[f,g;h]}_{L^2} &
    \lec \nm{f'}_{L^2} \nm{g}_{\dot H^{1/2}}\nm{h}_{L^\infty};
    \\  \label{eq:b112} 
  \nm{[f, g; h]}_{L^2}&\lec \nm{f}_{\dot H^{1/2}} \nm{g'}_{L^\infty}\nm{h}_{L^2};
     \\ \label{eq:b15}
     \nm{[f,g;h]}_{L^\infty}&\lec \nm{f'}_{L^2} \nm{g'}_{L^\infty} \nm{h}_{L^2}.
  \end{align}
\end{proposition}
\eqref{eq:b12} directly follows from H\"older's inequality and Hardy's inequality. \eqref{eq:b22} and \eqref{eq:b112} are direct consequences of Cauchy-Schwarz inequality and the definition \eqref{def-hhalf}. 
 \eqref{eq:b15} is from Appendix A of \cite{wu8}.

\begin{proposition} Assume that $f_i\in C^1(\mathbb R)$, with $f_i'\in L^{p_i}$, $g\in L^q$, where 
      $$ \sum_{i=1}^n\frac1{p_i}+\frac1q=\frac1p+1, \qquad 1<p_i\le\infty, \ 1\le q\le\infty, \ p>0.$$ Then
\begin{align}
     \label{eq:b115}
      \nm{\int\frac{\Pi_{i=1}^n(f_i(x)-f_i(y))}{(x-y)^n}g(y)\,dy}_{L^p}\lec \Pi_{i=1}^n\nm{f'_i}_{L^{p_i}} \nm{g}_{L^q}.     \end{align}
      
\end{proposition}
Observe that $\abs{\frac{f(x)-f(y)}{x-y}}\le \min\{M(f')(x), M(f')(y)\}$. \eqref{eq:b115} is a direct consequence of the Maximum inequality \eqref{eq:maximum} and H\"older's inequality.

We also have the following inequalities for the cubic form $<\cdot,\cdot,\cdot>$.

\begin{proposition}\label{prop:triple} For any $f, \ g, \ h$ smooth and decay fast at spatial infinity, we have
  \begin{align}
    \label{eq:b17}
    \nm{<f,g,h>}_{L^2} &
    \lec \nm{f'}_{L^2} \nm{g}_{L^\infty} \nm{h}_{\dot H^{1/2}};\\ \label{eq:b18}
     \nm{<f,g,h>}_{L^p}&\lec \nm{f'}_{L^2} \nm{g'}_{L^2} \nm{h}_{L^p},\qquad{1\le p\le \infty};   \\
     \label{eq:b19}
     \nm{<f,g,h>}_{L^2}&\lec \nm{f'}_{L^2} \nm{g}_{\dot H^{1/2}} \nm{h}_{\dot H^{1/2}};
     \\ \label{eq:b20}
     \nm{<f,g,h>}_{\dot H^{1/2}}&\lec \nm{f'}_{L^2} \nm{g'}_{L^2} \nm{h}_{\dot H^{1/2}};
     \\
     \label{eq:b21}
     \nm{<f,g,h>}_{\dot H^1}&\lec \nm{f'}_{L^2} \nm{g'}_{L^2} \nm{h'}_{L^2};\\
     \label{eq:b24}
     \nm{<f,g,h>}_{L^\infty}&\lec \nm{f'}_{L^2} \paren{\nm{g'}_{L^2} \nm{h}_{\dot H^{1/2}}+\nm{h'}_{L^2} \nm{g}_{\dot H^{1/2}}}.
       \end{align}

\end{proposition}

\eqref{eq:b17}, \eqref{eq:b18}, and \eqref{eq:b21} are easy consequences of  H\"older's inequality, Hardy's inequality,  and the definition \eqref{def-hhalf}. \eqref{eq:b19} and \eqref{eq:b20} follow from interpolation; \eqref{eq:b24} follows from the inequality  $$\nm{\partial_\aa<f,g,h>}_{L^1}\le \nm{f'}_{L^2} \paren{\nm{g'}_{L^2} \nm{h}_{\dot H^{1/2}}+\nm{h'}_{L^2} \nm{g}_{\dot H^{1/2}}},$$ which in turn follows from Cauchy-Schwarz inequality and Hardy's inequality.

\begin{proposition}\label{prop:b10}
 Assume that $f$, $g$, $h$ are holomorphic, i.e. $f=\P_H f$, $g=\P_H g$ and $h=\P_H h$. Then
\begin{equation}\label{eq:b39}
\nm{<f,g,h>}_{L^\infty}\lec \|f'\|_{L^2}\|g'\|_{L^2}\|h\|_{\dot H^{1/2}}.
\end{equation}

\end{proposition} 
Observe that for $f$, $g$ holomorphic, $[f,g;1]=0$. The same argument for \eqref{eq:b24} gives  \eqref{eq:b39}.

\begin{proposition}\label{prop:b13}
Assume that $f,\ g,\ h$ are smooth and decay fast at infinity, ${\rm b}={\rm b}(x,y)$ is bounded. Then
\begin{align}
\nm{\int {\rm b}(x,y)\frac{(f(x)-f(y))(g(x)-g(y))(h(x)-h(y))}{(x-y)^2}\,dy}_{L^2}&\lec \nm{\rm b}_{L^\infty(\mathbb R^2)}\nm{f'}_{L^2}\nm{g'}_{L^2}\nm{h}_{L^2}; \label{eq:b46}\\
\nm{\int {\rm b}(x,y)\frac{(f(x)-f(y))(g(x)-g(y))(h(x)-h(y))}{(x-y)^2}\,dy}_{L^2}&\lec \nm{\rm b}_{L^\infty(\mathbb R^2)}\nm{f'}_{L^2}\nm{g}_{ \dot H^{1/2}  }\nm{h}_{\dot H^{1/2}}. \label{eq:b47}
\end{align}
\end{proposition}
\eqref{eq:b46} follows from Cauchy-Schwarz inequality and Hardy's inequality;  \eqref{eq:b47} follows from  interpolation. 
 
 \begin{proposition}\label{prop:hhalf-2} Assume that $f_k\in \dot H^{1/2}(\mathbb R)$, $k=1,\, 2,\, 3,\, 4$. Then
\begin{equation}\label{eq:b48}
\abs{\int \frac{\Pi_{k=1}^4 (f_k(x)-f_k(y))}{(x-y)^2}\,dx\,dy}\lec \Pi_{k=1}^4 \nm{f_k}_{\dot H^{1/2}}.
\end{equation}

\end{proposition}
\begin{proof}
Observe that by symmetry, we can write 
$$\int \frac{\Pi_{k=1}^4 (f_k(x)-f_k(y))}{(x-y)^2}\,dx\,dy=2\int f_1(x)\int \frac{\Pi_{k=2}^4 (f_k(x)-f_k(y))}{(x-y)^2}\,dy\,dx.$$
By \eqref{eq:b19}, we have, for  $f_k\in H^1(\mathbb R)$, $k=1,\, 2$,
$$\abs{\int \frac{\Pi_{k=1}^4 (f_k(x)-f_k(y))}{(x-y)^2}\,dx\,dy}\lec \nm{f_1}_{L^2}\nm{f_2'}_{L^2}\nm{f_3}_{\dot H^{1/2}}\nm{f_4}_{\dot H^{1/2}};$$
similarly we also have
$$\abs{\int \frac{\Pi_{k=1}^4 (f_k(x)-f_k(y))}{(x-y)^2}\,dx\,dy}\lec \nm{f_1'}_{L^2}\nm{f_2}_{L^2}\nm{f_3}_{\dot H^{1/2}}\nm{f_4}_{\dot H^{1/2}}.$$
\eqref{eq:b48} then follows by interpolation. 
\end{proof}

The equality and inequalities in Lemma~\ref{hhalf2} and Proposition~\ref{hhalf4} are from  Section 5.1 and Lemma 6.3 of \cite{wu8}.

\begin{lemma}[cf. Lemma 5.4, \cite{wu8}]\label{hhalf2} 
 Assume that $f, \ g,\ f_1,\ g_1 \in H^1(\mathbb R)$ are the boundary values of some holomorphic functions on $\mathscr P_-$. Then
\begin{equation}\label{halfholo}
\int \partial_\aa \mathbb P_A (\bar f g)(\aa)f_1(\aa) \bar g_1(\aa)\,d\aa=-\frac1{2\pi i}\iint \frac{(\bar f(\aa)-\bar f(\bb))( f_1(\aa)- f_1(\bb))}{(\aa-\bb)^2}g(\bb)\bar g_1(\aa)\,d\aa d\bb.
\end{equation}

\end{lemma}

\begin{proposition}[cf. Proposition 5.6, Lemma 6.3, \cite{wu8}]\label{hhalf4}
Assume  that $f, \ g \in H^1(\mathbb R)$. We have
\begin{align}
\nm{\bracket{f, \mathbb H} g}_{\dot H^{1/2}}&\lec \|f\|_{\dot H^{1/2}}(\|g\|_{L^\infty} +\|\mathbb H g\|_{L^\infty});\label{hhalf41}\\
\nm{ \bracket{f, \mathbb H} g }_{\dot H^{1/2}}&\lec \|\partial_\aa f\|_{L^2}\|g\|_{L^2};\label{hhalf42}\\
\nm{\bracket{f, \mathbb H} \partial_\aa g}_{\dot H^{1/2}}&\lec \|g\|_{\dot H^{1/2}}(\|\partial_\aa f\|_{L^\infty}+\|\partial_\aa \mathbb H f\|_{L^\infty});\label{hhalf43}\\
\nm{f}^2_{L^4}&\lec \nm{f}_{L^2}\nm{f}_{\dot H^{1/2}}.\label{hhalf44}
\end{align}
\end{proposition}

Finally, we need the following inequalities for our proof. Let $\avg_x^y \rb=\frac{\int_y^x\rb(\a)\,d\a}{x-y}$.   We have

\begin{proposition}\label{quartic-inq} Let $\rb\in BMO(\mathbb R)$, $f_1,\dots,\ f_n\in C^1(\mathbb R)$ such that $f_1',\dots,\ f_n'\in L^2(\mathbb R)$, and $g\in L^2(\mathbb R)$. Then

\begin{align}
\nm{\int\paren{\rb(y)-\avg_x^y \rb }\frac{\Pi_{i=1}^2(f_i(x)-f_i(y))}{(x-y)^2}\,dy}_{L^\infty}&\lec \nm{\rb}_{BMO}\nm{f'_1}_{L^2}\nm{f'_2}_{L^2}; \label{eq:b25}\\
\nm{\int\paren{\rb(y)-\avg_x^y \rb }^2\frac{ \Pi_{i=1}^2(f_i(x)-f_i(y)) }{(x-y)^2}\,dy}_{L^\infty}&\lec \nm{\rb}_{BMO}^2\nm{f'_1}_{L^2}\nm{f'_2}_{L^2}; \label{eq:b26}\\
\nm{\int\paren{\rb(x)+\rb(y)-2\avg_x^y\rb}\frac{\Pi_{i=1}^2(f_i(x)-f_i(y))}{(x-y)^2}g(y)\,dy}_{L^2}&\lec \nm{\rb}_{BMO}\nm{f'_1}_{L^2}\nm{f'_2}_{L^2}\nm{g}_{L^2}; \label{eq:b27}
\\
\nm{\int\frac{(\rb(x)-\rb(y))\Pi_{i=1}^2(f_i(x)-f_i(y))}{(x-y)^2}g(y)\,dy}_{L^2}&\lec \nm{\rb}_{BMO}\nm{f'_1}_{L^2}\nm{f'_2}_{L^2}\nm{g}_{L^2}; \label{eq:b28}\\
\nm{\int\paren{\rb(x)+\rb(y)-2\avg_x^y\rb}\frac{\Pi_{i=1}^3(f_i(x)-f_i(y))}{(x-y)^3}\,dy}_{L^2}&\lec \nm{\rb}_{BMO}\nm{f'_1}_{L^2}\nm{f'_2}_{L^2}\nm{f'_3}_{L^2}; \label{eq:b36}\\
\nm{\int\frac{(\rb(x)-\rb(y))\Pi_{i=1}^3(f_i(x)-f_i(y))}{(x-y)^3}\,dy}_{L^2}&\lec \nm{\rb}_{BMO}\nm{f'_1}_{L^2}\nm{f'_2}_{L^2}\nm{f_3'}_{L^2};\label{eq:b37}
\end{align}
and for $ n\ge 3$, 
\begin{equation}\label{eq:b43}
\nm{\int\frac{(\rb(x)-\rb(y))\Pi_{i=1}^n(f_i(x)-f_i(y))}{(x-y)^n}g(y)\,dy}_{L^2}\lec \nm{\rb}_{BMO}\nm{f'_1}_{L^2}\nm{f'_2}_{L^2}\nm{g\Pi_{i=3}^nM(f'_i)}_{L^2},
\end{equation}
where $M(f)$ is the Hardy-Littlewood maximum function of $f$.

\end{proposition}
\begin{remark}
Observe that $ \dot H^{1/2}(\mathbb R)\subset BMO(\mathbb R)$, and $\nm{\rb}_{BMO}\lec\nm{\rb}_{\dot H^{1/2}}$ for all $\rb\in \dot H^{1/2}$. We will primarily use Proposition~\ref{quartic-inq} for $\rb\in\dot H^{1/2}(\mathbb R)$. We sometimes also use Proposition~\ref{quartic-inq} for $\rb\in L^\infty(\mathbb R)$, with the inequality $\nm{\rb}_{BMO}\lec\nm{\rb}_{L^\infty}$.
\end{remark}
\begin{proof}
It suffices to prove for 
the case that $f_1=f_2$. 
We begin with the following lemmas.
\begin{lemma}[Schur test]\label{Schur}
Let $$Tg(x)=\int e^{\rb(x)} K(x,y) e^{\rb(y)} g(y)\,dy.$$ 
where $\rb$ is a real valued measurable function, $K$ is measurable on $\mathbb R^2$.  Assume that
$$\max\{\,\sup_{x}\int e^{2\rb(y)} |K(x,y)|\,dy, \,\sup_{y}\int e^{2\rb(x)} |K(x,y)|\,dx\,\}:=M<\infty.$$
Then for any $g\in L^2$, 
$$\nm{Tg}_{L^2}\le M \nm{g}_{L^2}.$$
\end{lemma}

Let $\rb \in BMO(\mathbb R)$ be real valued, and $h(x)=\int_0^x e^{\rb(y)}\,dy$.
We know there is a constant $\gamma_0>0$, such that for any $\rb\in BMO(\mathbb R)$ satisfying $\nm{\rb}_{BMO}\le \gamma_0$, $e^{\rb}$ is a $A_2$ weight and $U_{h^{-1}}$ is a bounded map from $BMO(\mathbb R)$ to $BMO(\mathbb R)$, with $\nm{f\circ h^{-1}}_{BMO}\le 2\nm{f}_{BMO}$ for all $f\in BMO(\mathbb R)$.\footnote{These are consequences of John-Nirenberg's inequality and the theory of $A_p$ weights, cf. \cite{cm, cm1}.}
And for any real valued $BMO$ function $\rb$ with $\|\rb\|_{BMO}\le \gamma<1$, $z(x)=\int_0^x e^{i\rb(\a)}\,d\a$ defines a chord-arc curve, cf. \cite{cm, cm1}, with 
\begin{equation}\label{eq:b29}
(1-\gamma)|x-y|\le|z(x)-z(y)|\le |x-y|,\qquad\text{for all }x,y\in\mathbb R.
\end{equation}
We have 

\begin{lemma}\label{Schur-assump}
Let $\rb \in BMO(\mathbb R)$ be real valued with $\nm{\rb}_{BMO}\le \gamma_0/2$, and $h'(x)=e^{\rb(x)}$.
For any  $f$ such that $f'\in L^2(\mathbb R)$, there is a constant $c(\gamma_0)>0$, such that
\begin{equation}
\sup_{x}\int e^{2\rb(y)}\abs{\frac{f(x)-f(y)}{h(x)-h(y)}}^2\,dy\le c(\gamma_0) \nm{f'}_{L^2}^2.
\end{equation}

\end{lemma}
\begin{proof}
We know
$$f(x)-f(y)=\int_y^x f'(\a)\,d\a=\int_{h(y)}^{h(x)} f'\circ h^{-1}(\b) (h^{-1})'(\b)\,d\b,$$
hence
$$\sup_{x}\abs{\frac{f(x)-f(y)}{h(x)-h(y)}}\le M\paren{f'\circ h^{-1} (h^{-1})'}\paren{h(y)}$$
where $M(g)$ is the Hardy-Littlewood Maximal function of $g$. Because $e^{\rb\circ h^{-1}}$ is a $A_2$ weight, we have
\begin{equation}
\begin{aligned}
&\sup_{x}\int e^{2\rb(y)}\abs{\frac{f(x)-f(y)}{h(x)-h(y)}}^2\,dy\le \int e^{2\rb(y)}M^2\paren{f'\circ h^{-1} (h^{-1})'}\paren{h(y)}\,dy\\&
=
\int e^{\rb\circ h^{-1}(y)}M^2\paren{f'\circ h^{-1} (h^{-1})'}(y)\,dy\lec \int e^{\rb\circ h^{-1}(y)}\abs{f'\circ h^{-1} (h^{-1})'}^2(y)\,dy
= \nm{f'}_{L^2}^2.
\end{aligned}
\end{equation}
This proves Lemma~\ref{Schur-assump}.
\end{proof}
Now let $\rb$ be a $BMO$ function satisfying $\nm{\rb}_{BMO}\le 1$, and $z$ be a complex number in  $D:=\braces{|z|< \min\{\frac14,\frac{\gamma_0}2\}}$. 
Let $h_z(x)=\int_0^x e^{z\rb(\a)}\,d\a$ and 
\begin{align}
T_1(f,\rb; z)(x):&=\int e^{2z\rb(y)}\paren{\frac{f(x)-f(y)}{h_z(x)-h_z(y)}}^2\,dy,\\
T_2(f,g,\rb; z)(x):&=\int e^{z\rb(x)}e^{z\rb(y)}\paren{\frac{f(x)-f(y)}{h_z(x)-h_z(y)}}^2g(y)\,dy.
\end{align}
By Lemma ~\ref{Schur-assump}, and \eqref{eq:b29}, we have for all $f$ satisfying $f'\in L^2$,
\begin{equation}
\nm{T_1(f,\rb; z)}_{L^\infty}\le c(\gamma_0)\|f'\|_{L^2}^2,
\end{equation}
and by Lemmas~\ref{Schur}, ~\ref{Schur-assump}, and \eqref{eq:b29}, we have for all $f$, $g$, satisfying $f',\ g \in L^2$, 
\begin{equation}
\nm{T(f,g, \rb; z)}_{L^2(\mathbb R)}\le c(\delta_o)\|f'\|_{L^2}^2\|g\|_{L^2},
\end{equation}
where $c(\gamma_0)>0$ is a constant depending only on $\gamma_0$. Let $q_1\in L^1(\mathbb R)$,  $q_2\in L^2(\mathbb R)$, and 
\begin{align}
F_1(z)&:=\int q_1(x) T_1(f,\rb; z)(x)\,dx,\\
F_2(z)&:=\int q_2(x) T_2(f,g,\rb; z)(x)\,dx.
\end{align}
Then $F_1$, $F_2$ are holomorphic functions in the domain $D$, satisfying 
\begin{equation}\label{eq:b30}
|F_1(z)|\le c(\gamma_0)\|q_1\|_{L^1}\|f'\|_{L^2}^2, \quad |F_2(z)|\le c(\gamma_0)\|q_2\|_{L^2}\|f'\|_{L^2}^2\|g\|_{L^2},\qquad\text{for all }z\in D.
\end{equation}
And by Cauchy integral theorem, we have 
\begin{align}|F_1'(0)|\lec c(\gamma_0)\|q_1\|_{L^1}\|f'\|_{L^2}^2,\quad |F_1''(0)|\lec c(\gamma_0)\|q_1\|_{L^1}\|f'\|_{L^2}^2,\label{eq:b31}\\
\quad \text{and } \quad |F_2'(0)|\lec c(\gamma_0)\|q_2\|_{L^2}\|f'\|_{L^2}^2\|g\|_{L^2}\label{eq:b32}.
\end{align}
We compute $F_1'(0)$, $F_1''(0)$, and obtain
\begin{equation}\label{eq:b33}
\begin{aligned}
F_1'(0)&=2\iint q_1(x)\paren{\rb(y)-\avg_x^y \rb }\frac{(f(x)-f(y))^2}{(x-y)^2}\,dy\,dx\\
F_1''(0)&=\iint\frac{q_1(x)(f(x)-f(y))^2}{(x-y)^2}\paren{4\paren{\rb (y)-\avg_x^y \rb}^2-2\avg_x^y(\rb(\a)-\avg_x^y\rb)^2\,d\a }\,dxdy.
\end{aligned}
\end{equation}
Because $\sup_{x,y}\avg_x^y\abs{\rb(\a)-\avg_x^y\rb}^2\,d\a\lec \|\rb\|_{BMO}^2\le 1$ and by Hardy's inequality, 
$$ \iint\abs{q_1(x)}\abs{\frac{(f(x)-f(y))^2}{(x-y)^2}}\,dx\,dy\lec \|q_1\|_{L^1}\|f'\|_{L^2}^2,$$
so
$$\abs{\iint\frac{q_1(x)(f(x)-f(y))^2}{(x-y)^2}\paren{\avg_x^y(\rb(\a)-\avg_x^y\rb)^2\,d\a }\,dxdy}\lec \|q_1\|_{L^1}\|f'\|_{L^2}^2,$$
\eqref{eq:b25} and \eqref{eq:b26} then follow from \eqref{eq:b31} and \eqref{eq:b33}. We compute
\begin{equation}\label{eq:b34}
F_2'(0)=\iint q_2(x)\paren{\rb(x)+\rb(y)-2\avg_x^y \rb }\frac{(f(x)-f(y))^2}{(x-y)^2}g(y)\,dy\,dx.
\end{equation}
By \eqref{eq:b32} this gives us \eqref{eq:b27}. 

Now the Schur test, \eqref{eq:b26} and Hardy's inequality yields
\begin{equation}\label{eq:b35}
\nm{\int\paren{\rb(y)-\avg_x^y\rb}\frac{(f(x)-f(y))^2}{(x-y)^2}g(y)\,dy}_{L^2} \lec \nm{\rb}_{BMO}\nm{f'}_{L^2}^2\nm{g}_{L^2}.
\end{equation}
This together with \eqref{eq:b27} gives \eqref{eq:b28}. 

\eqref{eq:b36}, \eqref{eq:b37}, and \eqref{eq:b43} can be proved similarly by considering
\begin{equation}\label{eq:b38}
T_3(f, f_3, \rb; z):=\int e^{z\rb(x)}e^{z\rb(y)}\paren{\frac{f(x)-f(y)}{h_z(x)-h_z(y)}}^2\frac{f_3(x)-f_3(y)}{x-y}\,dy,
\end{equation}
\begin{equation}\label{eq:b45}
T_4(f, f_3,\dots f_n, g, \rb; z):=\int e^{z\rb(x)}e^{z\rb(y)}\paren{\frac{f(x)-f(y)}{h_z(x)-h_z(y)}}^2\frac{\Pi_{i=3}^n(f_i(x)-f_i(y))}{(x-y)^{n-2}}g(y)\,dy,
\end{equation}
and using the fact that 
$$
\abs{\frac{f_i(x)-f_i(y)}{x-y}}\le M(f_i')(y),\qquad 3\le i\le n.$$
We omit the details.

\end{proof}

\section{Main quantities controlled by $L(t)$ and $\nm{\partial_\aa\frac1{Z_{,\aa}}}_{L^2}+\nm{Z_{t,\aa}}_{\dot H^{1/2}}$} \label{quantities}
Assume that $$L(t)=\nm{Z_{t,\aa}}_{L^2}+\nm{\frac1{Z_{,\aa}}}_{\dot H^{1/2}}+\nm{Z_{t,\aa\aa}}_{L^2}+ \nm{\partial_\aa\frac1{Z_{,\aa}}}_{\dot H^{1/2}} \le 2\epsilon, \quad \nm{\frac1{Z_{,\aa}}-1}_{L^\infty}\le 1-\delta<1.$$
We have shown in \S\ref{quan} and \S\ref{step1-4} that the following quantities are controlled by $\epsilon$:
\begin{equation}\label{2020-1}
\begin{aligned}
& \nm{Z_{t,\aa}}_{L^\infty},\ \|Z_{t,\aa}\|_{\dot H^{1/2}},\ \nm{\partial_\aa\frac1{Z_{,\aa}}}_{L^2},\  \nm{Z_{tt}}_{\dot H^{1/2}}, \ \nm{b_\aa}_{L^2},\ \nm{\partial_\aa\Th^{(2)}}_{L^2}, \ \nm{D_t^2 Z_t}_{L^2},\ \nm{b_\aa}_{L^\infty},\\& \nm{D_t\frac1{Z_{,\aa}}}_{L^\infty}, \ \nm{Z_{ttt}}_{L^\infty},\ \nm{\partial_\aa b_\aa}_{L^2},\ \nm{\partial_\aa D_t\frac1{Z_{,\aa}}}_{L^2},\ \nm{Z_{tt,\aa}}_{\dot H^{1/2}},\ \nm{\partial_\aa D_t b}_{\dot H^{1/2}}, \\& \nm{D_t^2\frac1{Z_{,\aa}}}_{\dot H^{1/2}},\ \nm{Z_{ttt,\aa}}_{L^2},\ \nm{\partial_\aa D_t^2 b}_{L^2},\ \nm{D_t^3\frac1{Z_{,\aa}}}_{L^2},\ \nm{D_t^2 Z_{tt}}_{\dot H^{1/2}},\ \nm{D_t^3 Z_{tt}}_{L^2},\\& \nm{\partial_\aa\Th^{(4)}}_{L^2},\ \nm{\Th^{(5)}}_{\dot H^{1/2}},\ \nm{\partial_\aa\Th^{(3)}}_{\dot H^{1/2}},\ \nm{\partial_\aa D_t\Th^{(3)}}_{L^2},\ \nm{D_t^2\Th^{(3)}}_{\dot H^{1/2}},\ \nm{ D_t D_\aa \Th^{(3)}}_{L^2}; 
\end{aligned}
\end{equation}
the following quantities are controlled by $\epsilon^2$:
\begin{equation}\label{2020-2}
\begin{aligned}
& \nm{A_1-1}_{L^\infty}, \ \nm{A_1}_{\dot H^{1/2}}, \  \nm{D_t A_1}_{L^2}, \ \nm{b_\aa-2\Re D_\aa Z_t}_{L^2}, \ \nm{\mathcal P \bar Z_t}_{L^2}, \ \\& \nm{\partial_\aa A_1}_{\dot H^{1/2}},\
\nm{D_t(b_\aa-2\Re D_\aa Z_t)}_{\dot H^{1/2}},\  \nm{D_t(b_\aa-2\Re D_\aa Z_t)}_{L^\infty},\ \\& \nm{D_t^2 A_1}_{\dot H^{1/2}},\ \nm{D_t^2 A_1}_{L^\infty},\ \nm{\bracket{ D_t, \mathcal P} \bar Z_{ttt}}_{L^2},\ \nm{D_t\bracket{ D_t, \mathcal P} \bar Z_{tt}}_{L^2},\ \nm{D_t^2 \bracket{ D_t, \mathcal P} \bar Z_t}_{L^2},\\& \nm{D_t^3(b_\aa-2\Re D_\aa Z_t)}_{L^2},\ \nm{\partial_\aa \paren{ D_t^3 b-2\Re\frac{D_t^3 Z_t}{Z_{,\aa}}}}_{L^2},\  \nm{D_t^4 A_1}_{L^2},\ \nm{\mathcal P D_t^3 \bar Z_t}_{L^2};
\end{aligned}
\end{equation}
the following quantities are controlled by $\nm{\partial_\aa\frac1{Z_{,\aa}}}_{L^2}+\nm{Z_{t,\aa}}_{\dot H^{1/2}}$:
\begin{equation}\label{2020-3}
\begin{aligned}
&\nm{b_\aa}_{\dot H^{1/2}},\ \nm{D_t\frac1{Z_{,\aa}}}_{\dot H^{1/2}},\ \nm{\partial_\aa\Th^{(2)}}_{\dot H^{1/2}},\  \nm{D_\aa\Th^{(2)}}_{\dot H^{1/2}},\ \nm{Z_{tt,\aa}}_{L^2},\ \nm{Z_{ttt}}_{\dot H^{1/2}},\\& \nm{\partial_\aa D_t b}_{L^2},\ \nm{D_t D_\aa Z_t}_{L^2},\ \nm{D_t^3 Z_t}_{L^2},\ \nm{D_t^2\frac1{Z_{,\aa}}}_{L^2},\ \\&  \nm{\partial_\aa\Th^{(3)}}_{L^2},\ \nm{D_tD_\aa\Th^{(2)}}_{L^2},\ 
\nm{\Th^{(4)}}_{\dot H^{1/2}},\  \nm{D_t\Th^{(3)}}_{\dot H^{1/2}},\ \nm{D_t^2\Th^{(2)}}_{\dot H^{1/2}};
\end{aligned}
\end{equation}
the following quantities are controlled by $\epsilon\paren{\nm{\partial_\aa\frac1{Z_{,\aa}}}_{L^2}+\nm{Z_{t,\aa}}_{\dot H^{1/2}}}$:
\begin{equation}\label{2020-4}
\begin{aligned}
& \nm{b_\aa-2\Re D_\aa Z_t}_{\dot H^{1/2}},\ \nm{b_\aa-2\Re D_\aa Z_t}_{L^\infty},\ \nm{\partial_\aa A_1}_{L^2}, \ \nm{D_t A_1}_{\dot H^{1/2}},\ \nm{\mathcal P\bar Z_t}_{\dot H^{1/2}},\\& \nm{D_t A_1}_{L^\infty},\ \nm{\mathcal P\bar Z_t}_{L^\infty},\
\nm{D_t(b_\aa-2\Re D_\aa Z_t)}_{L^2},\  \nm{D_t^2A_1}_{L^2},\ \nm{\mathcal P\bar Z_{tt}}_{L^2},\ \nm{\bracket{\mathcal P, D_t}\bar Z_{t}}_{L^2},\\& \nm{D_t^2(b_\aa-2\Re D_\aa Z_t)}_{L^2},\ \nm{D_t^3 A_1}_{L^2},\ \nm{\partial_\aa\paren{ D_t^2 b-2\Re \frac{Z_{ttt}}{Z_{,\aa}}}}_{L^2},\ \nm{\mathcal P D_t^2 \bar Z_{t}}_{L^2},\\& \nm{\bracket{D_t, \mathcal P}\bar Z_{tt}}_{L^2},\ \nm{\bracket{D_t^2, \mathcal P}\bar Z_{t}}_{L^2},\ \nm{\partial_\aa(b_\aa-2\Re D_\aa Z_t)}_{L^2},\ \nm{\partial_\aa D_tA_1}_{L^2}.
\end{aligned}
\end{equation}

\end{appendix}


\begin{thebibliography}{10}
\newcommand{\msn}[1]{\href{http://www.ams.org/mathscinet-getitem?mr=#1}{\sc MR#1}}

\bibitem{ai}
A. Ai,  M. Ifrim \& D. Tataru {\it  Two dimensional gravity waves at low regularity I: energy estimates} arXiv:1910.05323

\bibitem{ait}
A. Ai,  M. Ifrim \& D. Tataru {\it  Two dimensional gravity waves at low regularity II: global solutions} arXiv:2009.11513

\bibitem{agr1} S. Agrawal {\it Rigidity of singularities of 2d gravity water waves} J. Diff. Equ. 268 (2020) no.3, 1220-1249


\bibitem{agr2} S. Agrawal {\it Angled crested type water waves with surface tension: Wellposedness of the problem}  To appear in Comm. Math. Phys. 

\bibitem{agr3} S. Agrawal {\it  Angled crested type water waves with surface tension II: Zero surface tension limit} arXiv:2009.13469

\bibitem{abz}
 T. Alazard, N. Burq \& C. Zuily {\it On the Cauchy problem for gravity water
waves.} Invent. Math. Vol.198 (2014)  pp.71-163

 \bibitem{abz14}
 T. Alazard, N. Burq \& C. Zuily {\it Strichartz estimates and the Cauchy problem for the gravity water waves equations.} 
 Mem. Amer. Math. Soc. 256 (2018), no. 1229
 
\bibitem{ad}
 T. Alazard \& J-M. Delort {\it Global solutions and asymptotic behavior for two dimensional
gravity water waves} Ann. Sci. \'Ec. Norm. Sup\'er. (4) 48 (2015), no. 5, 1149-1238.

\bibitem{am}
D. Ambrose, N. Masmoudi 
\emph{The zero surface tension limit of two-dimensional water waves}. Comm. Pure Appl. Math. 58 (2005), no. 10, 1287-1315

\bibitem{bhl} 
T. Beale, T. Hou \& J. Lowengrub {\it Growth rates for the linearized
		  motion of fluid interfaces away from equilibrium}
		  Comm. Pure Appl. Math. 46 (1993), no.9, 1269-1301.

\bibitem{bd}
M. Berti \& J-M. Delort {\it Almost global existence of solutions for capillary-gravity water wave equations with periodic spatial boundary conditions} arXiv:1702.04674

\bibitem{bfp}
M. Berti, R. Feola \& F. Pusateri {\it Birkhoff normal form and long time existence for periodic gravity water waves}
arXiv:1810.11549


\bibitem{bmsw}
L. Bieri, S. Miao, S. Shahshahani, \& S. Wu
{\it On the Motion of a Self-Gravitating Incompressible Fluid with Free Boundary }
Comm. Math. Phys.  Vo1. 355, no. 1, (2017), 161-243.  
 
\bibitem{bi}
G. Birkhoff {\it Helmholtz and Taylor instability} Proc. Symp. in
Appl. Math. Vol. XIII, pp.55-76.


\bibitem{chensu} 
G. Chen, Q. Su {\it Nonlinear modulational instability of the Stokes waves in 2d full water waves} arXiv:2012.15071, (2020)

\bibitem{cl} 
D. Christodoulou, H. Lindblad {\it On the motion of the free surface of a liquid} Comm. Pure Appl. Math. 53 (2000) no. 12, 1536-1602

\bibitem{cs} 
D. Coutand, S. Shkoller {\it Wellposedness of the free-surface incompressible Euler equations with or without surface tension} 
J. AMS. 20 (2007), no. 3, 829-930.

		  
\bibitem{cm} R.R. Coifman, Y. Meyer {\it Nonlinear harmonic analysis} Beijing Lectures in Harmonic Analysis, ed. by E. M. Stein, 1986.	  

\bibitem{cm1} R.R. Coifman, Y. Meyer {\it Lavrentiev's curves and conformal mappings} Rep.5, 1983, Mittag-Leffler Inst. Sweden.

\bibitem{cmm} 
R. Coifman, A. McIntosh and Y. Meyer {\it L'integrale de Cauchy definit un operateur borne sur $L^2$ pour les courbes lipschitziennes}  Annals of Math, 116 (1982), 361-387.





\bibitem{cr} 
W. Craig {\it An existence theory for water waves and the Boussinesq
and Korteweg-devries scaling limits} Comm. in P. D. E. 10(8), 1985
pp.787-1003

\bibitem{cw}
W. Craig \& P. Worfolk {\it An integrable normal form for water waves in infinite depth} Physica D 84 (1995) 513-531

\bibitem{cs2}
W. Craig \& C. Sulem {\it  Mapping properties of normal forms transformations for water waves} Boll. Unione Mat.
Ital. 9, no. 2, 289-318, 2016.

\bibitem{dipp}
Y. Deng, A. D. Ionescu, B. Pausader, \& F. Pusateri {\it Global solutions of the gravity-capillary water wave system in 3
dimensions} Acta Math. 219 (2017), no. 2, 213-402. 


\bibitem{ebi}
D. G. Ebin. {\it The equations of motion of a perfect fluid with free boundary are not well posed} Comm. P.  D.  E. 12(10), 1987, pp.1175-1201.

\bibitem{gms}
P. Germain, N. Masmoudi, \& J. Shatah {\it Global solutions of the gravity water wave equation in dimension 3} Ann. of Math (2) 175 (2012), no.2, 691-754.







\bibitem{hit}
J. Hunter, M. Ifrim \& D. Tataru {\it   
Two dimensional water waves in holomorphic coordinates} Comm. Math. Phys. 346 (2016), no. 2, 483-552. 


\bibitem{it}
M. Ifrim \& D. Tataru
{\it Two dimensional water waves in holomorphic coordinates II: global solutions}   Bull. Soc. Math. France 144 (2016), no. 2, 369-394. 


\bibitem{ig1} 
T. Iguchi {\it Well-posedness of the initial value problem for capillary-gravity waves} Funkcial. Ekvac. 44 (2001) no. 2, 219-241.

\bibitem{ip}
A. Ionescu \& F. Pusateri. {\it Global solutions for the gravity water waves system in
2d,}  Invent. Math. 199 (2015), no. 3, 653-804. 




\bibitem{kw}
R. Kinsey \& S. Wu {\it A priori estimates for two-dimensional water waves with angled crests} Cambridge J. Math. vol.6, no.2, (2018), pp.93-181



 


\bibitem{la} 
D. Lannes {\it Well-posedness of the water-wave equations} J. Amer. Math. Soc. 18 (2005), 605-654

\bibitem{le}
T. Levi-Civita. {\it D\'etermination rigoureuse des ondes permanentes d'ampleur finie.} Math. Ann.,
93(1), 1925. pp.264-314

\bibitem{li} 
H. Lindblad {\it Well-posedness for the motion of an incompressible liquid with free surface boundary} Ann. of Math. 162 (2005), no. 1, 109-194.


 





\bibitem{na} 
V. I. Nalimov {\it The Cauchy-Poisson Problem} (in Russian),
Dynamika Splosh. Sredy 18, 1974, pp. 104-210.

\bibitem{newton}
Sir Issac Newton {\it 
Philosophi Naturalis Principia Mathematica} 1726


\bibitem{ot} 
M. Ogawa, A. Tani {\it Free boundary problem for an incompressible ideal fluid with surface tension} Math. Models Methods Appl. Sci. 12, (2002), no.12, 1725-1740.
 



\bibitem{s}
E. M. Stein {\it Singular integrals and differentiability properties of functions} Princeton University Press, 1970. 





\bibitem{sz} 
J. Shatah, C. Zeng {\it Geometry and a priori estimates for free boundary problems of the Euler's equation} Comm. Pure Appl. Math. V. 61. no.5 (2008) pp.698-744



 \bibitem{su1}
Q. Su {\it Long time behavior of 2d water waves with point vortices},  To appear in Comm. Math. Phys.

\bibitem{su2}
Q. Su {\it  On the transition of the Rayleigh-Taylor instability in 2d water waves} arXiv:2007.13849


\bibitem{st} 
G. G. Stokes. {\it On the theory of oscillatory waves.} Trans. Cambridge Philos. Soc., 8: 1847, pp.441- 455.







\bibitem{ta} 
G. I. Taylor {\it The instability of liquid surfaces when accelerated in
		  a direction perpendicular to their planes I.}
		  Proc. Roy. Soc. London A 201, 1950, 192-196
		  





\bibitem{wang1}
X. Wang {\it Global infinite energy solutions for the 2D gravity water waves system}  Comm. Pure Appl. Math. 71 (2018), no. 1, 90-162. 

\bibitem{wang2}
X. Wang {\it Global solution for the 3D gravity water waves system above a flat bottom}
 Adv. Math. 346 (2019), 805-886.
 
\bibitem{wu1} 
S. Wu {\it  Well-posedness in Sobolev spaces of the full water wave problem
in 2-D}  Inventiones Mathematicae 130, 1997, pp. 39-72

\bibitem{wu2} 
S. Wu {\it  Well-posedness in Sobolev spaces of the full water wave problem
in 3-D} Journal of the AMS. 12. no.2 (1999), pp. 445-495. 

\bibitem{wu3} 
S. Wu {\it Almost global wellposedness of the 2-D full water wave problem} Invent. Math,  177, (2009), no.1, pp. 45-135.

\bibitem{wu4}
S. Wu {\it Global wellposedness of the 3-D full water wave problem} Invent. Math. 184 (2011), no.1, pp.125-220.


\bibitem{wu6}
S. Wu {\it  Wellposedness and singularities of the water wave equations} Notes of the lectures given at the Newton Institute, Cambridge, UK, Aug. 2014.




\bibitem{wu8}
S. Wu
{\it Wellposedness of the 2d full water wave equation in a regime that allows for non-$C^1$ interfaces} Invent. Math. 217 (2019),  no.2,  pp.241-375


\bibitem{yo} 
H. Yosihara {\it Gravity waves on the free surface of an
incompressible perfect fluid of finite depth,} RIMS Kyoto 18, 1982,
pp. 49-96

\bibitem{za}
V. E. Zakharov, {\it Stability of periodic waves of finite amplitude on the surface of a deep fluid} J. Appl. Mech. Tech. Phys. Vol. 9, 1968, 190-194.

\bibitem{zd}
A. I. Dyachenko \& V.E. Zakharov  {\it Is free-surface hydrodynamics an integrable system?} Physics Letters A 190
144-148, 1994.

\bibitem{dlz}
A. I. Dyachenko, Y. V. Lvov \& Zakharov {\it Five-wave interaction on the surface of deep fluid} Physica D 87, 233-261, 1995


\bibitem{dksz}
V. E. Zakharov, A. I. Dyachenko \& O. A. Vasilyev
 {\it New method for numerical simulation of a nonstationary potential flow of incompressible fluid with a free surface} European Journal of Mechanics B/Fluids, Vol. 21, (2002), 283-291


\bibitem{zz} 
P. Zhang, Z. Zhang {\it On the free boundary problem of 3-D incompressible Euler equations}. Comm. Pure. Appl. Math. V. 61. no.7 (2008), pp. 877-940

\bibitem{fz}
F. Zheng {\it  Long-term regularity of 3D gravity water waves} arXiv:1910.01912

\end{thebibliography}
\end{document}